\documentclass[11pt]{amsart}
\usepackage{amsmath, amssymb, amsthm, latexsym, amscd, enumerate, bbm,etex,nicefrac,mathrsfs,euscript,xspace,ifpdf,hyperref}
\usepackage[all]{xy}
\usepackage[margin=3.5cm,footskip=25pt,headheight=20pt]{geometry}
\ifpdf
\usepackage{pdfsync,hyperref,enumerate}
\fi
\usepackage{graphicx,tikz}
\numberwithin{equation}{section}
\usetikzlibrary{decorations.pathreplacing,backgrounds,decorations.markings,arrows}
\tikzset{wei/.style={draw=red,double=red!40!white,double distance=1.5pt,thin}}
\tikzset{bdot/.style={fill,circle,color=blue,inner sep=3pt,outer sep=0}}
\tikzset{dir/.style={postaction={decorate,decoration={markings,mark=at position .8 with {\arrow[scale=1.3]{>}}}}}}
\usepackage[vcentermath]{youngtab} \keywords{}  
\usepackage{todonotes}
 \newlength{\baseunit}               % the basic unit length
 \newcount{\numlines}                % depth of picture (in number of lines)
\allowdisplaybreaks
\newcommand{\arxiv}[1]{\href{http://arxiv.org/abs/#1}{\tt arXiv:\nolinkurl{#1}}}

\newtheorem{theorem}{Theorem}[section]
\newtheorem{thm}[theorem]{Theorem}
\newtheorem{lemma}[theorem]{Lemma}
\newtheorem{remark}[theorem]{Remark}
\newtheorem{prop}[theorem]{Proposition}
\newtheorem{corollary}[theorem]{Corollary}
\newtheorem{ex}[theorem]{Example}

\newtheorem*{grad}{Grading convention}

\newtheorem{definition}[theorem]{Definition}

\newtheorem{problem}[theorem]{Problem}
  \newcommand{\nc}{\newcommand}
 \newcommand{\renc}{\renewcommand}
\nc{\Stosic}{Sto{\v{s}}i{\'c}\xspace}
     \newcommand{\cF}{{\mathcal F}} 
\newcommand{\cH}{{\mathcal H}}   \newcommand{\cL}{{\mathcal
L}}   \newcommand{\cO}{{\mathcal O}}
\newcommand{\cP}{{\mathcal P}} \newcommand{\cR}{{\mathcal R}} \newcommand{\cQ}{{\mathcal Q}} \newcommand{\cS}{{\mathcal
S}}   \newcommand{\cV}{{\mathcal V}}

\newcommand{\Bi}{\mathbf{i}}
\newcommand{\charge}{\mathbf{\mathfrak{z}}}

\newcommand{\oneone}{1_1}
\newcommand{\onetwo}{1_2}
\newcommand{\onethree}{1_3}
\newcommand{\twoone}{2_1}
\newcommand{\twotwo}{2_2}
\newcommand{\twothree}{2_3}
\newcommand{\threeone}{3_1}
\newcommand{\threetwo}{3_2}
\newcommand{\threethree}{3_3}
\newcommand{\fourone}{4_1}
\newcommand{\fourtwo}{4_2}
\newcommand{\fivetwo}{5_2}

 \newcommand{\mg}{\mathfrak{g}}
 \newcommand{\mD}{\mathfrak{D}}

 \newcommand{\fQ}{\mathfrak{Q}} \newcommand{\FH}{\mathfrak{H}} 
  \newcommand{\mV}{\mathbb{V}} 
\newcommand{\mC}{\mathbb{C}} \newcommand{\C}{\mathbb{C}}  \newcommand{\mZ}{\mathbb{Z}}

\newcommand{\bLa}{\boldsymbol{\Lambda}} \newcommand{\La}{\Lambda} \newcommand{\K}{\mathbbm{k}}
 \newcommand{\la}{\lambda} \newcommand{\al}{\alpha}

\newcommand{\gmu}{{\grave{\mu}}}
\newcommand{\gla}{{\grave{\la}}}
\newcommand{\Q}{\cQ(\bmuh)} \newcommand{\Rep}{\operatorname{Rep}}
\newcommand{\GRep}{\operatorname{GRep}}

\newcommand{\END}{\operatorname{End}}

\newcommand{\op}{\operatorname}

\newcommand{\linj}{\longrightarrow}

\newcommand{\becircled}{\mathaccent "7017}
\newcommand{\bla}{{\boldsymbol{\la}}}

\newcommand{\bnu}{{\underline{\boldsymbol{\nu}}}}
\newcommand{\CST}{C_{\sS,\sT}}
\newcommand{\bmu}{{\boldsymbol{\mu}}}
\newcommand{\bmuh}{{\hat{\bmu}}}
\newcommand{\bnuh}{{\hat{\boldsymbol{\nu}}}}
\newcommand{\blah}{{\hat{\boldsymbol{\la}}}}

 \newcommand{\de}{\bf} 
  \newcommand{\bc}{{\bf c}} \newcommand{\excise}[1]{} \newcommand{\bd}{{\bf d}}
\newcommand{\bb}{{\bf b}}
\newcommand{\be}{{\bf e}}
\newcommand{\bk}{{\bf k}}
\newcommand{\bbf}{{\bf f}}
 \newcommand{\eT}{e_{{T}}} \newcommand{\tAb}{\tilde{\bf{A}}^\bnu}

  \newcommand{\sS}{\mathsf{S}} \newcommand{\sT}{\mathsf{T}}
\newcommand{\Bf}{\mathbf{F}}
\newcommand{\Be}{\mathbf{E}}
\newcommand{\Bk}{\mathbf{K}}
\newcommand{\Compe}{\operatorname{VComp}_e}
\newcommand{\fCompe}{\operatorname{VCompf}_e}

\newcommand{\Fock}{\mathbb{F}}

\newcommand{\lsy}{(}
\newcommand{\rsy}{)}
\newcommand{\Serre}{\mathfrak{S}}
\newcommand{\mmod}{{-}{\operatorname{mod}}}
\newcommand{\mgmod}{{-}{\operatorname{mod}}}
\newcommand{\mpmod}{{-}{\operatorname{pmod}}}

\nc{\slehat}{\mathfrak{\widehat{sl}}_e}
\nc{\glehat}{\mathfrak{\widehat{gl}}_e}
\nc{\tU}{\mathcal{U}}
\nc{\eE}{\EuScript{E}}
\nc{\eF}{\EuScript{F}}
\nc{\fF}{\mathfrak{F}}
\nc{\fE}{\mathfrak{E}}
\nc{\fK}{\mathfrak{K}}
\nc{\res}{\op{res}}

\DeclareMathOperator{\Hom}{Hom} \DeclareMathOperator{\RHom}{RHom}
 
 \DeclareMathOperator{\End}{End}

 \DeclareMathOperator{\Ext}{Ext}

\excise{ \newenvironment{proof}{}{} }

\begin{document} \pagestyle{plain}
\begin{center}
\noindent {\Large \bf Quiver Schur algebras and $q$-Fock space}
\bigskip

  \begin{tabular}{c@{\hspace{2cm}}c}
{\sc Catharina Stroppel}& {\sc Ben Webster}\\
{\tt stroppel@uni-bonn.de} & {\tt bwebster@virginia.edu}\\
Universit\"at Bonn & University of Virginia\\
Bonn, Deutschland & Charlottesville, VA, USA
\end{tabular}
\end{center}
\bigskip

{\small
\begin{quote}
\noindent {\em Abstract.}
We develop a graded version of the theory of
cyclotomic $q$-Schur algebras, in the spirit of the work of
Brundan-Kleshchev on Hecke algebras and of Ariki on $q$-Schur
algebras. As an application, we identify the coefficients of the
canonical basis on a higher level Fock space with $q$-analogues of
the decomposition numbers of cyclotomic $q$-Schur algebras.

We present cyclotomic $q$-Schur algebras as a quotient of a
convolution algebra arising in the geometry of quivers - we call it {\bf quiver Schur algebra} - and also diagrammatically, similar in flavor to a recent construction of
Khovanov and Lauda.  They are manifestly graded and so equip
the cyclotomic $q$-Schur algebra with a non-obvious grading. On the way we construct a {\bf graded cellular basis} of
this algebra, resembling similar constructions for cyclotomic Hecke algebras.

The quiver Schur algebra is also interesting from the perspective of
higher representation theory. The sum of Grothendieck groups
of certain cyclotomic quotients is known to agree with a higher level Fock space.
We show that our graded version defines a higher $q$-Fock space (defined as a tensor product of
level 1 $q$-deformed Fock spaces).  Under this identification, the
indecomposable projective modules are identified with the canonical
basis and the Weyl modules with the standard basis. This allows us to
prove the already described relation between decomposition numbers and
canonical bases.
\end{quote}
}


\tableofcontents
\section{Introduction}

\renc{\thetheorem}{\Alph{theorem}}
Recent years have seen remarkable advances in higher
representation theory; the most exciting from the perspective of
classical representation theory were probably the proof of Brou\'e's conjecture for the symmetric groups by Chuang and
Rouquier \cite{CR04} and the introduction and study of graded versions of Hecke
algebras by Brundan and Kleshchev \cite{BKKL} with its Lie theoretic origins (\cite{BK1}, \cite{BS3}).
At the same time, the question of
finding categorical analogues of the usual structures of Lie theory
has proceeded in the work of Khovanov, Lauda, Rouquier,
Vazirani, the authors and others.
In this paper, we address a question of interest from both perspectives, representation
theory and higher categorical structures.

As classical (or quantum) representation theorists, we ask
\begin{center} {\em Is there a natural graded version of the $q$-Schur
    algebra and its higher level analogues, the cyclotomic $q$-Schur
    algebra?}
\end{center}
This question has been
addressed already in the special case of level 1 by Ariki \cite{Ariki}. We give here a more general
construction that both illuminates connections to geometry and is more
explicit. Our construction includes the case of ordinary Schur algebras, see Remark \ref{Schurlevel1} and the explicit example at the end of the paper.

As higher representation theorists, we ask
\begin{center}
 {\em Is there a natural categorification of $q$-Fock space
    and its higher level analogues with a categorical action of $\widehat{\mathfrak{sl}}_n$?}
\end{center}
We will show that the above two questions not only have natural
answers; they have the {\it same} answer.

Our main theorem is a {\it graded version} with a {\bf graded cellular basis} of the cyclotomic $q$-Schur algebra of Dipper, James and Mathas, \cite{DJM} and a combinatorics of {\it graded decomposition numbers} using higher Fock space.\\

To describe the results more precisely, let $\K$ be an algebraically
closed field and $n,\ell,e$ natural numbers with $e>1$ (we will allow
the possibility that $e=\infty$).  Let $q\in \K$
be a scalar with $e$  the smallest integer such that
$1+q+\cdots +q^{e-1}=0$ and $(Q_1,\ldots, Q_\ell)$ an $\ell$-tuple of
elements of $\K$ satisfying the same equation. In particular,
$Q_i=q^{\charge_i}$ for some $\charge_i\in \mathbb Z/e\mathbb Z$ (since $q^e=1$),
unless $q=1$, in which case $\K$ has characteristic $e$, and we set
$Q_i=\charge_i$.  The associated {\bf cyclotomic Hecke algebra} or
{\bf Ariki-Koike algebra} $$\mathfrak{H}(n;q,Q_1,\ldots,
Q_\ell)=\mathfrak{H}(S_n\wr \mathbb{Z}/\ell\mathbb{Z};q,Q_1,\ldots,
Q_\ell)$$ is the associative unitary $\K$-algebra with generators
$T_i$, $1\leq i\leq n-1$ modulo the following relations, for $1\leq
i<j-1< n-1$,
\[(T_0-Q_1)\cdots(T_0-Q_\ell)=1,\qquad(T_i-q)(T_i+1)=1,\]
\[T_iT_j=T_jT_i,\qquad T_iT_{i+1}T_i=T_{i+1}T_iT_{i+1},\qquad T_0T_1T_0T_1=T_1T_0T_1T_0.\]
The {\bf cyclotomic $q$-Schur algebra} we consider is the endomorphism ring
\begin{eqnarray}
\label{DefqSchur}
\mathbf{S}(n;q,Q_1,\ldots, Q_\ell)&=&\op{End}_{\mathfrak{H}(n;q,Q_1,\ldots, Q_\ell)}
\left(\bigoplus_{\hat{\mu}\in\La} M(\hat{\mu})\right)
\end{eqnarray}
where the sum runs over all {\it admissible} $\ell$-multi-compositions $\hat{\mu}$ of $n$ and $M(\hat{\mu})$ denotes the (signed) permutation module associated to $\hat{\mu}$. The admissibility condition we choose (see Lemma~\ref{admiss}) defines a certain subset $\La$ of all $\ell$-multi-compositions; in this way we pick via \eqref{DefqSchur} a specific representative out of the family of cyclotomic $q$-Schur algebras from \cite[\S6]{DJM}.

To make contact with the theory of quiver Hecke algebras, we encode the
parameters $(q,Q_1,\dots,Q_\ell)$ as a sequence $\bnu=(w_{\charge_1},\ldots,
w_{\charge_\ell})$ of fundamental weights for the affine Lie algebra
$\slehat$. The corresponding cyclotomic Hecke algebra only depends on the
multiplicities of each charge, that is, only on the weight
$\nu=\sum_{i}\nu_i$.  Thus, we write
\begin{eqnarray}
\label{SH}
\mathbf{S}^\bnu=\bigoplus_{n\geq
  0}\mathbf{S}(n;q,Q_1,\ldots, Q_\ell)&\text{ resp.}&
  \mathfrak{H}^\nu:=\bigoplus_{n\geq
  0}\mathfrak{H}(n;q,Q_1,\ldots, Q_\ell)
  \end{eqnarray}
  for the sums of corresponding cyclotomic $q$-Schur and Ariki-Koike  algebras of all different ranks. To this data we introduce then a certain $\mZ$-graded algebra $A^{\bnu}$ which we call the {\bf cyclotomic quiver Schur algebra}. The name stems from the fact that the algebra is related to the cyclotomic quiver Hecke
algebras $R^\nu$ and their tensor product analogues $T^\bnu$ (defined
in \cite{Webmerged}) like the finite Schur algebra is to the classical Hecke algebra. Our main result (Theorem \ref{schur-isomorphism}) says that  $A^{\bnu}$ is a graded version of $\mathbf{S}^\bnu$ given by an extension of the Brundan-Kleshchev isomorphism $\Phi^\nu:R^\nu\cong \mathfrak{H}^\nu$
from \cite{BKKL} between the diagrammatic cyclotomic quiver Hecke
algebras $R^\nu$ introduced in \cite{KL1} and the sum of cyclotomic Hecke algebra $\mathfrak{H}^\nu$.

\begin{theorem}
\label{theorem A}
  There is an isomorphism $\Phi^\bnu$ from $A^\bnu$ to the cyclotomic $q$-Schur
  algebra $\mathbf{S}^\bnu$, extending the isomorphism
  $\Phi^\nu:R^\nu\cong \mathfrak{H}^\nu$. In particular, the cyclotomic $q$-Schur algebra
  $ \mathbf{S}^\bnu$ from \eqref{SH} inherits a $\mZ$-grading.
\end{theorem}
Like the cyclotomic quiver Hecke algebra, the algebra $A^\bnu$ can be realized as a natural quotient of a geometrically defined convolution algebra. This construction is based on the geometry of quivers using a certain category of flagged nilpotent representations (quiver partial flag varieties) of the cyclic affine type $A$ quiver. It naturally extends the work of Varagnolo and Vasserot \cite{VV} and clarifies the origin of the grading.\\

The construction of $A^\bnu$ (and its summand $A^\bnu_n$ for fixed $n$) proceeds in three steps. We first define an infinite dimensional convolution algebra, which we call {\bf quiver Schur algebra}, using flagged representations of the cyclic quiver $\Gamma$. This algebra only depends on $e$ (and has summands depending on $n$). The second step is to add some extra shadow vertices to the quiver and define an convolution algebra working with flagged representations of the extended quiver $\Gamma$ depending on the parameters $Q_i$ and $\ell$. Finally the last step is to pass to a certain finite dimensional quotient which is the desired algebra $A^\bnu$ (with the direct summand $A^\bnu_n$ for fixed $n$).

Although we focus here on the cyclotomic quotients, we want to stress that the quiver Schur algebra appears naturally in representation theory. It is, after completion, isomorphic to Vigneras' Schur algebra \cite{Vigneras} for the general linear {$p$}-adic group, see \cite{MS}.\\

To make explicit calculations we describe the quiver Schur algebra algebraically by considering a faithful representation on a direct sum of polynomial rings, extending the corresponding result for quiver Hecke algebras. Moreover, we give a diagrammatical description of the algebra by extending the diagram calculus of Khovanov and Lauda \cite{KL1}. In contrast to their work, however, we are not able to give a complete list of relations diagrammatically. Still, we have enough information to construct (signed) permutation modules for the cyclotomic Hecke algebra and show that our algebra $A^\bnu$ is isomorphic to the cyclotomic $q$-Schur algebra using the known three different description (geometric, algebraic and diagrammatical) of the quiver Hecke algebras. 

Note that this can also be viewed as an extension of work of the second author \cite[5.31]{Webmerged}, which showed that similar diagrammatic algebras were the endomorphism algebras of some, but not all, (signed) permutation
modules.

Independently, Ariki \cite{Ariki} introduced graded $q$-Schur
algebras and studied their permutation modules. This is a special
case of our results when $\ell=1$ (and $e$ and $n$ are not too small)
and indeed, we show that our grading coincides with Ariki's.

In the course of the proof we also establish, similar in spirit to the arguments in \cite{BS3}, the existence of a {\it graded cellular basis} (Theorem \ref{A-is-cellular}) in the sense of Hu and Mathas \cite{HM}:

To avoid case by case arguments for the combinatorics in the special case $e=2$  we restrict ourselves in the combinatorial part (starting in Section \ref{sec:basis-defined}) of the paper and for the remaining results of the introduction to the case $e>2$, although the main theorems extend to the case $e=2$ as well.

\begin{theorem}
\label{Theorem B}
The cyclotomic quiver Schur algebra $A^\bnu$ is a graded cellular algebra. Moreover,
\begin{itemize}
\item The cellular ideals coincide under $\Phi^\bnu$ with those of Dipper-James-Mathas.
\item The cell modules define graded lifts of the Weyl modules.
\end{itemize}
\end{theorem}
The grading on the basis vectors comes from a degree function defined on semistandard multitableaux extending known degree functions on standard tableaux. We show that this grading comes naturally from geometry and extends the grading on the tensor algebras from \cite{Webmerged}.

After circulating a draft of this paper, we received a preprint of Hu and
Mathas \cite{HMQ} which gives a definition of a different, but closely
related algebra which they also call a quiver Schur algebra.  They prove Theorems~\ref{Theorem B},~\ref{Theorem C} and
\ref{Theorem D} in this context for the (special) case of a linear,
rather than cyclic, quiver (that is, when $e=\infty$).  In \cite[5.31]{Webmerged}, the second author shows that Hu
and Mathas's algebra is Morita equivalent to certain
 tensor product algebras $T^\bla$; these algebras are, in turn,
Morita equivalent to those defined here when $e=\infty$ by Proposition \ref{T-morita}.

From both the geometric and the diagrammatic sides, our construction fits quite snugly inside Rouquier's program of categorical representation theory, \cite{Rou2KM}:
\begin{theorem}
  \label{Theorem C}
 The category of graded $A^\bnu$-modules carries a categorical
  action of $U_q(\slehat)$ in the sense of Rouquier.  As a
  $U_q(\slehat)$-module, its complexified graded Grothendieck group is canonically isomorphic to the $\ell$-fold tensor product
  $$\Fock_\ell=\Fock_1(\charge_\ell)\otimes \cdots \otimes \Fock_1(\charge_1),$$
  of level $1$ (fermionic) Fock space with central charge $\charge=({\charge_1},\ldots, {\charge_\ell})$
 given by $\bnu$.
\end{theorem}
Here the parameter $q$ corresponds to the effect of grading shift on
the Grothendieck group, and is thus a formal variable not a complex
number. The action of the standard Chevalley generators $E_i,F_i$ is given by $i$-induction and $i$-restriction functors. In the ungraded case these functors and their connection to undeformed Fock space were independently studied by \cite{Wada}.

Our constructions can also be described using an affine version of the ``thick calculus''
introduced by Khovanov, Lauda, Mackaay and \Stosic for the
upper half of $\slehat$. Unlike in \cite{KLMS} however, our category has objects
which do not appear in the ``thin calculus'' and categorifies
the upper half of $U_q(\glehat)$ instead of $U_q(\slehat)$.

It is tempting to think this could easily be extended to a
graphical categorification of the Lie algebra $\glehat$, which is the
direct sum of $\slehat$ and a Heisenberg Lie algebra $\mathbb{H}$
modulo an identification of their centers. This is especially
intriguing and promising given the various interesting
categorification of Heisenberg algebras (\cite{CaLa}, \cite{KhovHeis}, \cite{LicSav}, \cite{survey}, \cite{CLS}) which appeared recently in the
literature. However, the connection cannot be as straightforward as
one might hope at first, since the functors associated to the standard
generators in the categorification of the upper half of $U_q(\glehat)$
simply do not have biadjoints (unlike those in $U_q(\slehat)$). This was already pointed out in \cite[5.1,5.2]{ShanFock}. In
their action on the categories of $A^\bnu$-modules, they send
projectives to projectives, but not injectives to injectives, so their
right adjoints are exact, but not their left adjoints.
%Chuang and
%Rouquier have done unpublished work on the nature of
%$\widehat{\mathfrak{gl}}_e$ categorification \cite{RouPC}, but there is much that
%remains to be explored.
\\

The Grothendieck group of graded $A^\bnu$-modules is naturally a $\mZ[q,q^{-1}]$-module and comes also along with several distinguished lattices and bases.
To describe them combinatorially we introduce a bar-involution on the
tensor product $\Fock_\ell$ of Fock spaces appearing in
Theorem~\ref{Theorem C} which allows us to define, apart from the
standard basis, two other distinguished bases: the {\bf canonical} and
{\bf dual canonical} bases. This canonical basis is a ``limit'' of
that for higher level $q$-Fock spaces defined by Uglov, \cite{Uglov}, in a sense we
describe later.  Our canonical isomorphism induces correspondences
\begin{samepage}
  \begin{eqnarray*}
    \text{bar involution}&\longleftrightarrow&\text{Serre-twisted duality}\\
    \text{canonical basis}&\longleftrightarrow& \text{indecomposable
      projectives}\hspace{.69cm}  (\operatorname{char}(\K)=0)\\
    \text{dual canonical basis}&\longleftrightarrow& \text{simple
      modules} \hspace{2.7cm}
    (\operatorname{char}(\K)=0)\\
    \text{standard basis}&\longleftrightarrow& \text{Weyl modules}
  \end{eqnarray*}
\end{samepage}
%\[
%\tikz[thick,xscale=1.2,outer sep=10pt]{\node (ac) at (3,2) {dual Serre functor} ;\node (bc) at (-3,2)
%  {bar involution}; \draw[<->] (ac)--(bc);
%\node (aa) at (-3,0) {canonical basis} ;\node (ba) at (3,0)
%  {indecomposable projectives};
%\node (ab) at (-3,1) {dual canonical basis}; \node (bb) at (3,1)
%  {simple modules}; \draw[<->] (aa)--(ba); \draw[<->] (ab)--(bb);
%\node (ad) at (-3,-1) {standard basis}; \node (bd) at (3,-1)
%  {Weyl modules}; \draw[<->] (ad)--(bd);
%}
%\]
As a consequence we get  information about
(graded) decomposition numbers of cyclotomic $q$-Schur algebras:
\begin{theorem}[Theorem \ref{decompnumb}]\label{Theorem D}
 If $\K$ has characteristic 0,  the graded decomposition numbers of the cyclotomic $q$-Schur algebra
 are the coefficients of the canonical basis in terms of the standard basis on the higher level $q$-Fock space $\mathbb{F}_\ell$.
\end{theorem}
Again this problem was studied independently by Ariki \cite{Ariki} for
the level $\ell=1$ case. The above theorem combines and generalizes
therefore results of Ariki on Schur algebras and
Brundan-Kleshchev-Wang \cite{BKWSpecht} on cyclotomic Hecke algebras.
It is very similar in spirit to a conjecture of Yvonne
\cite[2.13]{Yv}; however, there are several small differences between
Yvonne's conjecture and our results.  The most important is that
Yvonne used Jantzen filtrations to define a $q$-analogue of
decomposition numbers instead of a grading.  This approach has been worked out in level 1 by Ram-Tingley
\cite{RamTing} and Shan \cite{ShanFock}.  Their results show that the
same $q$-analogue of decomposition numbers arise from counting
multiplicities with respect to depth in Jantzen filtrations.  We
expect that this will hold in higher level as well; it should follow from the following fact (which was conjectured in a first draft of this paper and) proved in \cite{M}:
\begin{theorem}
\label{Koszul}
The (underlying basic algebra of) the cyclotomic quiver Schur algebra $A^\bnu$ is Koszul.
\end{theorem}

We prefer working with gradings instead of
filtrations, since they are easier to handle in practice. (A similar
phenomenon appears for the classical category $\mathcal{O}$ for
semi-simple complex Lie algebras, where the Jantzen filtration can also
be described in terms of a grading, \cite{BGS}, \cite{Strgrad}. This
grading is actually directly connected with the grading on the
algebras $R^\nu\cong H^\nu$ in case $e=\infty$, see \cite{BS3}, \cite{HM}).
Again, for level $\ell=1$, the Theorem \ref{Koszul} was already
known to be true, \cite{ChuangMiyachi}; an elementary argument for $e=\infty$
is given in \cite{BS2}.  \\

We should emphasize that at the moment, this approach only allows us
to understand the higher-level Fock spaces which are constructed as
tensor products of level $1$ Fock spaces (or their irreducible
$\glehat$ constituents).  This does not include the twisted
higher level Fock spaces studied by Uglov, \cite{Uglov}, which will
require a generalization of the algebras we consider here.  The same Fock spaces are
categorified by category $\cO$ of certain Cherednik algebras, see
e.g. \cite{ShanFock}, \cite{GordonLosev}. The action is again given by
induction and restriction functors as in \cite{Wada}, but it remains
to be clarified how our work fits into this framework.

Let us briefly summarize the paper.
\begin{itemize}
\item Section \ref{q-Schur} contains preliminaries of the geometry of quiver representations needed to define
  the quiver Schur algebra both as a geometric convolution
  algebra and in terms of an action on a polynomial ring which then is related to Demazure operators in Section \ref{sec:splitmerge}. We connect its graded Grothendieck group with the generic nilpotent Hall algebra of the cyclic quiver.
\item In Section \ref{sec:high-levels}, we discuss a generalization of
  this algebra using extended (or shadowed) quiver representations that will allow us to deal with higher level Fock
  spaces.
\item In Section \ref{sec:cyclotomic-quotients}, we define cyclotomic
 quotients, equip them with a (graded) cellular structure and establish the isomorphism to cyclotomic $q$-Schur algebras in Section \ref{Section6}.
\item In Section \ref{sec:q-fock-space}, we consider the connection of
  these constructions to higher representation theory, describe the
  categorical action of $\slehat$ on these categories, and show that
  they categorify $q$-Fock spaces. In particular, we consider the
  relationship between projective modules, canonical bases, and
  decomposition numbers.
%\item We finish with an explicit small example describing a block of the ordinary level $\ell=1$ Schur algebra.
\end{itemize}

{\bf Acknowledgment:} We thank Geordie Williamson, Michela Varagnolo,
Peter Tingley, Iain Gordon, Michael Ehrig, Bernard Leclerc and
Peter Littelmann
for useful input and fruitful discussions.  We are in particularly
grateful to Peng Shan, Vanessa Miemitz and Daniel Tubbenhauer for their careful reading
and probing questions.
We also would like to thank
the Hausdorff Center for supporting B.W.'s visit to Bonn at the
genesis of this work, and the
organizers of the Oporto Meeting on Geometry, Topology and Physics
2010 for facilitating our collaboration. B.W. was supported by an NSF Postdoctoral Research Fellowship and  by the NSA under Grant H98230-10-1-0199.

\section{The quiver Schur algebra}
\label{q-Schur}
\renc{\thetheorem}{\arabic{section}.\arabic{theorem}}
Throughout this paper, we will fix an integer $e> 1$; as mentioned
before, we also allow
the possibility that $e=\infty$. (We will assume for simplicity $e>2$ at some point later in the text.) Let $\Gamma$
be the Dynkin diagram for $\hat{\mathfrak{sl}}_e$;  with the fixed
clockwise orientation if $e$ is finite and with the fixed linear orientation
 if $e=\infty$, Figure \ref{fig:circle}. Let
$\mathbb{V}=\{1,2,\ldots, e\}$ (or $\mathbb{V}=\mZ$ for $e=\infty$) be the set
of vertices of $\Gamma$, identified with the set of remainders of
integers modulo $e$. Let $h_{\overline i}:i\rightarrow i+1$ be the
arrow from the vertex $i$ to the vertex $i+1$
where here and in the following all formulas should be read taking
indices corresponding to vertices of $\Gamma$ modulo $e$ when $e<\infty$.

\begin{definition}
 A (finite-dimensional) {\de representation} $(V,f)$ of $\Gamma$ over a field $\K$ is
  \begin{itemize}
 \item a collection of $\K$-vector spaces $V_i, i\in \mathbb{V}$
   such that $\sum_i\dim V_i <\infty$, together with
\item $\K$-linear maps
   $f_i:V_i\rightarrow V_{i+1}$.
  \end{itemize}
A {\de subrepresentation} is a collection of vector subspaces $W_i\subset V_i$ such that $f_i(W_i)\subset W_{i+1}$.
A representation $(V, f)$ of $\Gamma$ is called {\de nilpotent} if the map
  $f_e\cdots f_{2}f_1:V_1\to V_1$ is nilpotent (when $e=\infty$, all
  representations are called nilpotent).

\end{definition}

\subsection{Quiver representations and quiver flag varieties}
\label{sec:quiv-repr-grassm}
  The {\bf dimension vector} of a
  representation $(V,f)$ is the tuple ${\bd}=(d_1,\dots,d_e)$, where
  $d_i=\op{dim} V_i$. We let $|\bd|=\sum d_i$ and denote by  $\al_i$ the special dimension vector where $d_j=\delta_{ij}$.
  Mapping it to the simple root  $\al_i$ of $\slehat$ identifies the set of dimension vectors
 with the positive cone in the root lattice of $\slehat$, and with semi-simple nilpotent representations of $\Gamma$:

\begin{lemma}
\label{ss}
There is a unique irreducible nilpotent representation $S_j$ of dimension vector $\al_j$. Any semi-simple nilpotent representation is of the form $(V, f)$ with $f_i=0$ for all $i$.
\end{lemma}

\begin{proof}
  Obviously $S_j$ equals $(V,f)$, where $f_i$ and $V_i$ are zero
  except of $V_j=\K$. Assume $(V, f)$ is a non-trivial irreducible
  nilpotent representation. If, for given $j$, $f_j\not=0$ then $f_j$ is injective, since
  otherwise $W_j={\op{ker}f_j}$ and $W_r=\{0\}$ for $r\not=j$ defines a
  non-trivial proper subrepresentation. Not all $f_i$'s are injective,
  since the representation is nilpotent and non-trivial. Pick $i$ such that $f_i$ is
  not injective with $V_i\not=0$. Then $f_i=0$ and hence $(V, f)$ is isomorphic to
  $S_i$, since $S_i$ is a subrepresentation. Any representation $(V, f)$ with $f_i=0$ for all $i$ is
  obviously semi-simple. Conversely, assume $(V, f)$ is semi-simple,
  hence isomorphic to $\bigoplus_{i=1}^e S_i^{d_i}$. In particular,
  $S_i^{d_i}$ is a direct summand (for any $i$) which implies that
  $f_i=0$.
\end{proof}

Let ${\op{Rep}}_{{\bd}}$ be the affine space of representations of $\Gamma$ with dimension vector ${\bd}$, i.e.
\begin{equation}
\label{Rep}
{\op{Rep}}_{{\bd}}=\bigoplus_{i\in \mathbb{V}}\Hom(\C^{d_i},\C^{d_{i+1}}).
\end{equation}
\begin{figure}[h]
   \begin{tikzpicture}[thick,scale=2]
      \node (z) at (0,0) {$e$};
      \node (n) at (-1,-.2) {$e-1$};
      \node (o) at (1,-.2) {$1$};
      \draw[->] (n) to[out=20,in=180] (z);
   \draw[->] (z) to[out=0,in=160] (o);
    \draw[->] (-1.7,-.7) to[out=60,in=-160] (n);
  \draw[->] (o) to[out=-20,in=120] (1.7,-.7);
 \node[rotate=120] at (1.9,-.9){$\cdots$} ;
\node[rotate=60] at (-1.85,-.9){$\cdots$} ;

\node[circle,fill=red,inner
sep=2pt,outer sep=3pt] (rz) at (0,.8) {};
\draw[->,red] (z) to (rz);
\node[circle,fill=red,inner
sep=2pt,outer sep=3pt] (ro) at (1.3,.6) {};
\draw[->,red] (o) to (ro);
\node[circle,fill=red,inner
sep=2pt,outer sep=3pt] (rn) at (-1.3,.6) {};
\draw[->,red] (n) to (rn);

    \end{tikzpicture}
\caption{The oriented Dynkin quiver $\Gamma$ and the extended
Dynkin quiver $\tilde{\Gamma}$.}
\label{fig:circle}
\end{figure}
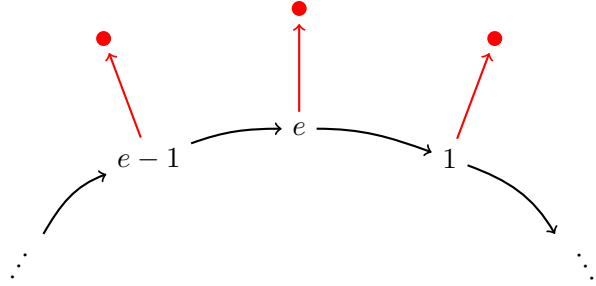
This space has a natural algebraic action by conjugation of the algebraic group $G_\bd=\op{GL}(d_1)\times\cdots\times\op{GL}(d_e)$, and we are interested in the {\bf moduli
 space of representations}, i.e. the quotient $\op{GRep}_{{\bf
   d}}={\op{Rep}}_{{\bd}}/G_\bd$
   parametrizing isomorphism classes of representations. Since the $G_\bd$-action is very
far from being free, we must interpret this quotient in an intelligent way.
One option is to consider it as an Artin stack.  While this
is perhaps the most elegant approach, it is more technical than
necessary for our purposes. Instead the reader is encouraged to
interpret this quotient as a formal symbol where, by convention for
any complex algebraic $G$-variety $X$.
\begin{itemize}
\item $H^*(X/G)$ is the $G$-equi\-vari\-ant cohomology $H_{G}^*(X)$, and
  $H^{BM}_*(X/G)$ is the $G$-equivariant Borel-Moore homology of $X$
  (for a discussion of equivariant Borel-Moore homology, see
  \cite[section 19]{Fultonint}, or \cite[\S 1.2]{VV}).
\item $D(X/G)$ (resp.\ $D^+(X/G)$) is the bounded (resp.\ bounded below) equivariant derived category of
  Bernstein-Lunts \cite{BeLu}, with the usual six functor formalism
 described therein. See also \cite{WW1} as an additional reference for our purposes.
\end{itemize}

Note that $H^*_{G_\bd}({\op{Rep}}_{{\bd}})=H^*(BG_\bd)$, where
$BG_\bd$ denotes the classifying space of $G_\bd$ (or the classifying
space of its $\mC$-points in the analytic topology for the
topologically minded) since ${\op{Rep}}_{{\bd}}$ is
contractible. Thus, we can use the usual Borel isomorphism to identify
the $H^*(\op{GRep}_{\bd})$'s with polynomial rings.

We will consider the direct sum $H^*(\op{GRep})=\bigoplus_{\bd} H^*(\op{GRep}_{{\bf
    d}})$ which corresponds to taking the union of the quotients $ \op{GRep}=\bigsqcup_{\bd} \op{GRep}_{{\bf
    d}}$ as Artin stacks.  One can think of this as a quotient by the groupoid $G=\bigsqcup_{\bd}
G_{\bd}$, and we will speak of the $G$-equivariant cohomology of $\op{Rep}=\bigsqcup_{\bd} \op{Rep}_{{\bf
    d}}$, etc.\\

We'll be interested in spaces of quiver representations equipped with
compatible flags.  These have appeared several times in the literature,
most importantly in work of Lusztig, e.g. \cite{Lus91}, on the geometric construction of
the canonical basis; our work builds on his ideas. Now, we introduce
the
combinatorics underlying these spaces.

A {\de composition} of length $r$ of $n\in\mZ_{>0}$ is a tuple $\mu=(\mu_1,\mu_2,\ldots,\mu_r)\in\mZ_{>0}^r$ such that $\sum_{i=1}^r \mu_i=n$. In contrast, a {\de vector composition}\footnote{This differs from the notion of an $r$-multi-composition where the tuples could be of different lengths or an $r$-multi-partition where the tuples are partitions, but again not necessarily of the same length $m$.} of type $m\in\mZ_{>0}$ and length  $r=r(\hat{\bmu})$ is a tuple $\bmuh=(\bmu^{(1)},\bmu^{(2)},\ldots,\bmu^{(r)})$ of nonzero elements from
$\mZ_{\geq 0}^m$. If $m=e$ we call a vector composition also {\de residue data} and denote their set by $\op{Comp}_e$.
Alternatively, $\bmuh$ can be viewed as an $r\times m$ matrix $\bmu[{i},j]$ with the $i$th row $\bmu^{(i)}$.
We call the column sequence, i.e. the result of reading
the columns, the {\bf flag type sequence} $\op{t}(\bmuh)$,
whereas the {\de residue sequence} $\op{res}(\bmuh)$ is the sequence
\begin{equation}
\label{res}
1^{\bmu_1^{(1)}},2^{\bmu_2^{(1)}},\cdots, e^{\bmu_e^{(1)}} |
1^{\bmu_1^{(2)}},2^{\bmu_2^{(2)}},\cdots, e^{\bmu_e^{(2)}} |
\cdots|
1^{\bmu_1^{(r)}},2^{\bmu_2^{(r)}},\cdots, e^{\bmu_e^{(r)}}.
\end{equation}
The parts separated by the vertical lines $|$ are called {\bf blocks} of $\op{res}(\bmuh)$.
We say that $\bmuh$ has {\bf complete flag type} if every
$\bmu^{(i)}$ is a unit vector which has exactly one non-zero entry.
Hence blocks of $\op{res}(\bmuh)$ contain in this case at most one element.

The transposed vector composition of $\bmuh$ of type $r(\bmu)$
and length $m$ is defined to be $\check{\bmu}=(\check{\bmu}^{(1)}, \check{\bmu}^{(2)}, \ldots,
\check{\bmu}^{(m)})$ where
$\check{\bmu}^{(i)}_j=\bmu[{j},i]=\bmuh_i^{(j)}$, which we call in case $m=e$ also the {\de
flag data}. Given a composition $\mu$ of $n$ of length $r$, we denote by
${\mathcal{F}}(\mu)$ the {\bf variety of flags of type $\mu$}, that is
the variety of flags $$F^1(\mu)\subset F^2(\mu)\subset\cdots\subset
F^r(\mu)=\mC^n,$$ where $F^i(\mu)$ is a subspace of dimension
$\sum_{j=1}^i\mu_j$.  For a vector composition $\bmuh$ we let
$\cF(\bmuh)=\prod_i\cF(\check{\bmu}^{(i)})$, a product of partial flag varieties inside $\mC^d=\prod_{i=1}^e\mC^{d_i}$ of type given by the flag data. Here $\bd=\bd(\bmuh)=(d_1,\ldots, d_e)$ denotes the {\bf dimension vector} of the vector composition $\bmuh$ which is simply the sum
$\bd=\sum_{i=1}^r \bmuh^{(i)}$, and $d$ denotes the total dimension. For a given dimension vector $\bd$ denote \[\Compe(\bd)=\{\bmuh\in\Compe\mid \bd(\bmuh)=\bd\}.\]

\begin{ex}
\label{ourex}
{\rm
Let $e=2$ and $\bmuh=((2,1),(1,1),(2,3),(0,1))$, a residue
data  of length $r=4$ with dimension vector $\bd=(5,6)$. The flag data
is
\[\check{\bmu}=((2,1,2,0),(1,1,3,1))\text{ and
}\res(\bmuh)=1,1,2|1,2|1,1,2,2,2|2.\]
There are $\binom{11}{5}$ elements of complete flag type in $\Compe(\bd)$.
}
\end{ex}

\begin{definition}
 For a given vector composition $\bmuh\in\op{Comp}_e$ (of length $r$) a {\de
   representation with compatible flags of type $\check\bmu$} is a
  nilpotent representation $(V,f)$ of $\Gamma$ with dimension
  vector ${\bd}=\bd(\bmuh)$ together with a flag $F(i)$ of type
  $\check{\bmu}^{(i)}$ inside $V_i$ for each $1\leq i\leq e$ such that $f_i(F(i)^{j})\subset F({i+1})^{j-1}$ for $1\leq j\leq r$.
We denote by \[\cQ(\bmuh)=\big\{
(V,f,F)\in \Rep_\bd\times \cF(\bmuh)\mid f_i(F(i)^{j})\subset F({i+1})^{j-1} \;\forall\; i,j\big\}\] the subset of $\Rep_\bd\times \cF(\bmuh)$ of representations with compatible flag. It comes equipped with the obvious action of the group $G:=G_{\bf{d}}$ by
change of basis.
\end{definition}

Alternatively (see Lemma \ref{ss}), a representation with compatible flags is a tuple
$((V,f), F)$ consisting of a representation
$(V,f)\in\Rep_{{\bd}}$ equipped with a filtration by
subrepresentations with semi-simple successive quotients (with
dimension vectors given by $\bmuh$). In Example
\ref{ourex}, we have a quiver $\Gamma$ with two vertices and simple representations $S_1, S_2$ and the residue data or flag data gives flags of the form
\begin{eqnarray*}
F_2^1(1)\subseteq F_3^2(1)\subseteq F_5^3(1)\subseteq F_5^4(1)=\mC^5,&&
F_1^1(2)\subseteq F_2^2(2)\subseteq F_5^3(2)\subseteq F_6^4(2)=\mC^6,
\end{eqnarray*}
where the subindex denotes the dimension of the subspaces. In particular, if we set $V_i=F^i(1)\oplus F^i(2)$, then the semisimple subquotients
are the representations given by the dimension vector $\bmuh$, namely $$V_4\cong S_2,\quad V_3/V_4\cong S_1^2\oplus S_2^3,\quad
V_2/V_3\cong S_1\oplus S_2,\quad V_1/V_2\cong S_1^2\oplus S_2^1.$$

The flag data defines a Young subgroup $S_{\check\bmu}\subseteq S_{\bd}$.
 Note that ${\cF}({\bmuh})\cong G_\bd/P_{\check\bmu}$ is the partial flag variety defined by the parabolic subgroup $P_{\check{\bmu}}$ of $G_\bd$, given by upper triangular block matrices with block sizes determined by the flag sequence. The flags which appear will be complete if and only if $\bmuh$ has
complete flag type, as suggested by the name.  This
special case is particularly important; it played a key role in
earlier papers of Lusztig
\cite{Lus91} and Varagnolo and Vasserot \cite{VV}.

Forgetting either the flag or the representation defines two $G_{\bd}$-equivariant morphisms 
\begin{eqnarray}
p:\quad\Q\longrightarrow \Rep_{\bd},&& ((V,f),F)\mapsto (V,f)\in\Rep_{\bd},\label{p}\\
\pi:\quad\Q\longrightarrow \cF({\bmuh}),&& ((V,f),F)\mapsto F\in\cF({\bmuh}).\label{pq}
\end{eqnarray}
of algebraic varieties. Generalizing the notion of a quiver Grassmannian we call the fibres of $p$ the {\de quiver partial flag varieties}.

To each $i\in\mV$ assign a polynomial ring over $\K$ in $d_i$ variables $x_{i,1}, \ldots , x_{i,d_i}$ and set
\begin{equation}
\label{alphabeth}
R({\bd})=\bigotimes_{j=1}^e\K[x_{j,1},\dots,x_{j,d_j}]=
\K[x_{1,1},\dots,x_{1,d_1},\dots, x_{e,1},\dots,x_{e,d_e}].
\end{equation}
This algebra carries an action of the Coxeter group $S_{\bd}=S_{d_1}\times \cdots \times S_{d_e}$ by permuting the variables in the same tensor factor. Thus, the Borel presentation identifies the rings
\begin{equation}
\label{GH}
H^*(\cF({\bmuh})/G)\cong R({\bd})^{S_{\bmuh}}=: \bLa(\hat{\boldsymbol{\mu}}),
\end{equation}
by sending Chern classes of tautological bundles to elementary
symmetric functions, \cite[Proposition 1]{Brionlectures}.

In the extreme case where $\bd=(r,0,\dots,0)$ we have $\bmuh^{(i)}=(1,0,\ldots,0)$ for all $i$ and $S_{\bmuh}$ is trivial. Hence we obtain the $\op{GL}_d$-equivariant cohomology of the variety of full flags in $\mC^d$. If $\bmuh=({\bd})$ then $\bmuh=((d_1),(d_2),\ldots, (d_e))$, the variety $\cF({\bmuh})$ is
just a point and $\bLa({\bmuh})=R({\bd})^{S_{\bd}}$. We call this
the {\bf ring of total invariants}.
\begin{lemma}
\label{lem:vb}
The map $\pi$ is a vector bundle with affine fibre; in particular, we have a natural isomorphism
$\pi^*: H^*(\cF({\bmuh})/G)\cong H^*(\Q/G)$, and thus $H^*(\Q/G)\cong \bLa({\bmuh})$.

\end{lemma}

\begin{grad}{\rm
Throughout this paper, we will use a somewhat unusual grading convention
on cohomology rings: we shift the (equivariant) cohomology ring or
equivariant Borel-Moore homology of a
smooth variety $X$ downward by the complex dimension of $X$, so that
the identity class is in degree $-\dim_\C X$.  This choice of grading
has the felicitous effect that pull-back and push-forward maps in cohomology are more symmetric.
Usually, for a map
$f\colon X\to Y$ we have that
pull-back has degree $0$, and push-forward has degree $2\dim_\C Y-2\dim_\C
X$, whereas in our grading convention
\begin{equation}
\label{gradconv}
\text{\it both push-forward and pull-back have degree $\dim_\C
  Y-\dim_\C X$.}
\end{equation}
Those readers comfortable with the theory of the constructible derived
category will recognize this as replacing the usual constant sheaf
with the intersection cohomology sheaf, denoted $\K_{\Q/G}\in D^+(\Q/G)$, of $\Q/G$.

Since $\Q$ is smooth, $\K_{\Q/G}$ is simply a homological shift of the
usual constant sheaf on each individual component (but by different
amounts on each component).  A more conceptual
explanation for the convention is the invariance of $\K_{\Q/G}$ under Verdier duality.  This
grading shift also provides a straight-forward explanation of the
grading on the representation $\mathcal{P}o\ell$ of the quiver Hecke algebra
in \cite{KL1}, \cite{KLII}.
}
\end{grad}

Since $p$ is proper
and the constant sheaf of geometric origin, the
Beilinson-Bernstein-Deligne decomposition theorem from \cite{BBD} applies (in the formulation \cite[Theorem 4.22]{CM}), and $p$ sends the
$\K_{\Q/G}$ on $\Q/G$ to a direct sum $L$ of shifts of
simple perverse sheaves on $\Rep_{\bd}$.  Our first object of study in
this paper will be the algebra of extensions of these sheaves.

\subsection{The convolution algebra}
Let ${\bmuh, \blah}\in \op{Comp}_e$ be vector compositions with associated dimension vector ${{\bd}}$, and
consider the corresponding ``Steinberg variety''
\begin{equation}
\label{Steinberg}
Z(\bmuh, \blah)=\cQ(\bmuh)\times_{\Rep_{{\bd}}} \cQ(\blah).
\end{equation}
Let  $G_\bd$ act diagonally and abbreviate $\cH({\bmuh, \blah})=Z(\bmuh, \blah)/G_\bd$.
Note that $\Q$ is smooth and $p$ is proper. So, by \cite[Theorem
8.6.7]{CG97}, we can identify the algebra of self-extensions of $L$
(although not as a graded algebra) with the equivariant Borel-Moore homology $H_*^{BM,G}$ of our Steinberg variety: we have the natural identification
\[\Ext^*_{D^b(\op{GRep}_{\bd})}(p_*\K_{\cQ({\bmuh})} ,p_*\K_{\cQ(\blah)})=H_*^{BM,G}(Z({\bmuh, \blah})),\]
such that the Yoneda product
\begin{eqnarray*}
&\Ext^*_{D^b(\op{GRep}_{\bd})}(p_*\K_{\cQ(\bmuh)},p_*\K_{\cQ(\blah)})\otimes\Ext^*_{D^b(\op{Rep}_{\bd})}(p_*\K_{\cQ(\blah)},p_*\K_{\cQ(\bnuh)})&\\
&\rightarrow\quad\Ext^*_{D^b(\op{GRep}_{\bd})}(p_*\K_{\cQ(\bmuh)},p_*\K_{\cQ(\bnuh)})&
\end{eqnarray*}
agrees with the convolution product. This
defines an associative non-unital graded algebra structure $A_{\bd}$ on
\begin{eqnarray}
\label{BM}
\bigoplus_{(\hat\mu,\hat{\la})}\Ext^*_{D^b(\op{GRep}_{\bd})}\left(p_*\K_{\cQ(\bmuh)}, p_*\K_{\cQ(\bmuh)}\right)=\bigoplus_{(\hat\mu,\hat{\la})} H_*^{BM}(\cH({\bmuh, \blah})).
\end{eqnarray}
Here  the sums are over all elements in $\Compe(\bd)\times\Compe(\bd)$. For reasons of
diagrammatic algebra, we call this product  {\bf vertical composition}
and we denote it like the usual multiplication $(f,g)\mapsto fg$. We call the algebra $A_{\bd}$ the {\bf quiver Schur algebra}.

\subsection{A faithful polynomial representation}
 The quiver Schur algebra $A_{\bd}$ acts, by \cite[Proposition 8.6.15]{CG97}, naturally on the sum (over all vector partitions of dimension vector $\bd$) of cohomologies
\begin{eqnarray} \label{V} V_{\bd}&:=&\bigoplus_{\bmuh\in \Compe(\bd)} H_*^{BM}(\Q/G)\cong \bigoplus_{\bmuh\in \Compe(\bd)}
\bLa(\bmuh).
\end{eqnarray}
Note that this is compatible with the grading convention \eqref{gradconv}.

\begin{prop}
\label{faithful}
If $\operatorname{char}(\K)=0$, the $A_{\bd}$-module $V_{\bd}$ is faithful.
\end{prop}
In fact, the hypothesis on characteristic is unnecessary, since this
is a special case of Proposition~\ref{hl-faithful}, which is
characteristic independent; however, there is
a more geometric proof in this special case, which we give here.
\begin{proof}
Recall that the ring $A_\bd$ is just the Ext-algebra of
$\oplus_{|\bmuh|=\bd}p_*\K_{\Q}$. The space  $V_{\bd}$  is the
hypercohomology of $\oplus_{|\bmuh|=\bd}p_*\K_{\Q)}$ up to
shifts of grading, since Borel-Moore homology
$H^{BM}_*(X)$ of a smooth space $X$ equals hypercohomology
$\mathbb{H}^{-*}(X,\mathbb{D})$ of its dualizing sheaf $\mathbb{D}$
which is, up to a shift, the constant sheaf.  Hence, if we let  $j$ be
the map from $\op{Rep}_{{\bd}}$ to a point, faithfulness is equivalent
to $j_*$ being injective,
\begin{equation}
j_*:\quad\Ext^\bullet(p_*\K_{\Q},
p_*\K_{\cQ(\bmuh')})\to \Ext^\bullet(j_*p_*\K_{\Q},
j_*p_*\K_{\cQ(\bmuh')}).\label{injective}
\end{equation}

We only need to check that the same property holds when
$p_*\K_{\Q} $ is replaced by a summand.  The nilpotent orbits of $G_\bd$ in $\Rep_\bd$ are equivariantly simply
connected; thus every simple  perverse sheaf supported on the nilpotent locus
is the intermediate extension of a trivial local system.
Thus, by the decomposition theorem, the
sheaves $p_*\K_{\Q}$ decompose into summands which are shifts of the
intersection cohomology sheaves of these orbits. We note that the spaces we deal with have good parity vanishing
properties. Each orbit has even equivariant cohomology, since it has a transitive action of $G$ with the stabilizer of a point given
by a connected algebraic group (\cite[Lemma~78, Theorem~79]{Libine}).  Also, the stalks of intersection cohomology sheaves
have even cohomology \cite[Theorem 5.2(3)]{Hen}.
Thus,  by \cite[Theorem 3.4.2]{BGS}, $j_*$ is faithful on semi-simple $G_d$-equi\-va\-riant perverse
sheaves, and so the map \eqref{injective} is injective.
\end{proof}

\subsection{Monoidal structure}
\label{monoidal}
The algebra $A_{\bd}$ comes along with distinguished idempotents $e_{\bmuh}$ indexed by vector compositions with dimension vector $\bd$. The $A_{\bd}$-module
\begin{equation}
\label{Ps}
P(\bmuh)=\bigoplus_{\blah\in\Compe(\bd)} e_\blah A_{\bd} e_\bmuh
\end{equation}
is a finitely generated indecomposable graded projective $A_\bd$-module. Each finitely generated indecomposable graded projective $A_{\bd}$-module is (up to a grading shift) isomorphic to one of this form. We denote by
$A_\bd\mpmod$ the category of graded finitely generated projective $A_\bd$-modules.
\begin{prop}\label{s-s-perv}
Assume $\operatorname{char}(\K)=0$.  The category $A_\bd\mpmod$ is
equivalent to the additive category of sums of shifts of semi-simple perverse
  sheaves in $D^+(\op{GRep})$ which are pure of weight $0$ with nilpotent support in
  $\GRep_\bd$.
\end{prop}
Again, the characteristic 0 hypothesis could be avoided, but at the
cost of some difficulties we prefer to avoid. The corresponding
perverse sheaves in the characteristic $p$ case will not be
semi-simple, but rather parity sheaves, in the sense of \cite{JMW}.
\begin{proof}
  There is a functor sending a finitely generated projective $A_\bd$-module $M$ to the
  perverse sheaf
  $\left(\bigoplus_{\bd(\bmuh)=\bd}p_*\K_{\cQ_{\bmuh}}\right)\otimes_{A_\bd}M$.
  This map is fully faithful, since it induces an isomorphism on the
  endomorphisms of $A_\bd$ itself.  Thus, we only need to show that every
  simple perverse sheaf on $\op{GRep}$ with nilpotent support is a
  summand of $p_*\K_{\cQ({\bmuh})}$ for some $\bmuh$.

Every such
  simple perverse sheaf is $\mathbf{IC}(\bar X)$ where $X$ is the
  locus of modules isomorphic to a fixed module $N$.  Consider the socle
  filtration on $N$. The dimension vectors of the successive
  quotients define a vector composition $\bmuh_N$. The map
  $\cQ(\bmuh_N)\to \op{GRep}$ is generically an isomorphism over $X$, and
  thus for dimension reasons, has image $\bar X$.  In
  particular, \[p_*\K_{\cQ({\bmuh_N})}\cong \mathbf{IC}(\bar X)\oplus
  \mathbf{L}\] where $\mathbf{L}$ is a finite direct sum of shifts of semi-simple perverse
  sheaves supported on $\bar X\setminus X$.   Thus, every simple
  perverse sheaf with nilpotent support is a summand of such a
  pushforward, and we are done.
\end{proof}

We are going to define a monoidal structure on $A_\bd\mpmod$ using correspondences. For $\bmuh, \bnuh\in \Compe$ the {\bf join} $\bmuh\cup\bnuh$ is the vector composition
$$\bmuh\cup\bnuh=(\mu^{(1)},\ldots,\mu^{(r(\bmuh))},\nu^{(1)},\ldots,\nu^{r(\bnuh)})$$
obtained by joining the two tuples.
% For
%instance, $\bmuh$ as in Example \ref{exs} equals
%$$((2,1,0),(1,0,0),(2,2,0))\cup ((1,0,0),(1,1,0),(1,3,0),(0,0,1)).$$
Let $\bmuh_1,\blah_1 $ and
$\bmuh_2,\blah_2$ be vector compositions with associated dimension vectors ${\bf c}$ and ${{\bd}}$
respectively. Let $$\cQ(\bmuh_1;\bmuh_2,\blah_1; \blah_2)\subseteq \cQ(\bmuh_1\cup\bmuh_2)\times_{\Rep_{\bc}}\cQ(\blah_1\cup\blah_2)$$ be the space of
representations with dimension vector ${\bf c}+{{\bd}}$ which carry
a pair of compatible flags of type $\bmuh_1\cup\bmuh_2$ and $\blah_1\cup\blah_2$ respectively, such that the subspaces of dimension
vector ${\bf c}$ in the two flags coincide. Let $\cH(\bmuh_1;\bmuh_2,\blah_1; \blah_2)$ be the quotient by the diagonal $G_\bd$-action.
\begin{definition}
The \textbf{horizontal multiplication} is the map
\begin{eqnarray}
\label{horizontal}
 A_{\bf c}\times A_{{\bd}}&\longrightarrow& A_{{\bf c} +{{\bd}}}\nonumber\\ (a,b)&\longmapsto& a|b
\end{eqnarray}
induced on equivariant Borel-Moore homology by the correspondence (i.e. by pull-and-push on the following diagram)
\begin{equation}
\label{corresp}
\cH(\bmuh_1,\blah_1)\times \cH (\bmuh_2,\blah_2) \longleftarrow \cH(\bmuh_1;\bmuh_2,\blah_1;\blah_2 )\longrightarrow \cH(\bmuh_1\cup \bmuh_2,\blah_1\cup
\blah_2).
\end{equation}
Here the rightward map is the obvious inclusion, and the leftward is induced from the map $V\mapsto
(W, V/W)$ of taking the common subrepresentation $W$ of dimension vector ${\bf c}$ and the quotient
by it.
\end{definition}
We let ${\bf A}=\oplus_{\bd}(A_\bd\mpmod)$ be the direct sum of the categories
$A_\bd\mpmod$ over all dimension vectors; that is its
objects are formal direct sums of finitely many objects from these
categories, with morphism spaces given by direct sums.
\begin{prop}
\label{monoid}
The assignment $\otimes:(P(\hat{\mu}),P(\hat{\nu}))\mapsto P(\hat{\mu}\cup\hat{\nu})$
extends to a monoidal structure $({\bf A}, \otimes, \mathbf{1})$
with unit element $\mathbf{1}=P(\emptyset)$.
\end{prop}
\begin{proof}
In the ungraded case this follows directly from Lusztig's convolution
product \cite[\S3]{Lus91}, by Proposition \ref{s-s-perv}.  More explicitly, we define \[M\otimes N={A_{{\bf c} +{{\bd}}}}\otimes_{A_{\bf c}\times A_{{\bd}}} M\boxtimes N,\] meaning one first takes the outer tensor product of the graded $A_{\bf c}$-module $M$ and the graded $A_{\bf d}$-module $N$. The resulting $A_{\bf c} \times A_{\bf d}$-module is then induced to a graded $A_{{\bf c} +{{\bd}}}$-module via the horizontal multiplication \eqref{horizontal}. This is functorial in both entries and defines the required tensor product with the asserted properties.
\end{proof}

\subsection{Categorified generic nilpotent Hall algebra}
\label{sec:Hall}
Let $K^0_q(\mathbf{A})$ be the split Gro\-then\-dieck group of the additive Krull-Schmidt category
$\mathbf{A}$, i.e. the free abelian group on isomorphism classes $[M]$ of objects in ${A}\mpmod$ modulo the relation $[M_1]+[M_2]=[M_1\oplus M_2]$. This is a free
$\mZ[q,q^{-1}]$-module where the action of $q$ is by grading shift (and has nothing to do with the parameter $q$ from the introduction).

For a graded vector space $W=\oplus_jW^j$, we define its grading shifts $W\langle d\rangle$, $d\in\mZ$, by
$(W\langle d\rangle)^j=W^{d+j}$, and let $q^d[M]=[M\langle d\rangle]$.  The module
$K^0_q(\mathbf{A})$ is of infinite rank, but is naturally a direct sum of the
Grothendieck groups $K^0_q(A_\bd\mpmod)$, each of which is finite rank.

Let $\fCompe(\bd)\subset \Compe$ be the set of vector compositions of
$\bd$ of complete flag type.  For each $\bd$, there is a subalgebra,
\begin{eqnarray}
\label{Rd}
R_\bd&=&\bigoplus_{\blah,\bmuh\in
  \fCompe(\bd)}e_{\blah}A_\bd e_{\bmuh}.
\end{eqnarray}
There is also a corresponding
monoidal subcategory $\mathbf{R}$ of $\mathbf{A}$ generated by the indecomposable
projectives indexed by the $\bmuh\in \fCompe(\bd)$ for all $\bd$. Both $A_\bd$ and
$R_\bd$ can be
defined for any quiver and the following proposition holds in general,
though in this paper we only use these categories for the affine type
$A$ quiver. The algebra $R_\bd$ appears as {\bf quiver Hecke algebra} (associated with ${\bf d})$ in the literature:

\begin{prop}[\mbox{Vasserot-Varagnolo/Rouquier \cite[3.6]{VV}}]\hfill\label{VVtheorem}\\
  As a graded algebra, $\mathbf{R}_\bd$ is isomorphic to
  the quiver Hecke algebra $R(\bd)$ associated to the quiver $\Gamma$ in
  \cite{Rou2KM}.  In particular, $K_q^0(\mathbf{R})$ is naturally
  isomorphic to the Lusztig integral form of
  $U_q^-(\slehat)$ by mapping the isomorphism classes of indecomposable projective objects to Lusztig's canonical basis.
  \label{VVR}
\end{prop}
The idempotents in $R_\bd$ get identified with those in $R(\bd)$ by
viewing the residue sequence \eqref{res} as a sequence of simple roots
$\alpha_i$.  We should note that Proposition \ref{VVtheorem} uses the
signed version of the quiver Hecke algebra appearing in \cite{VV}, \cite{BKKL} which differs from the first paper of
Khovanov and Lauda \cite{KL1}.

\begin{remark}\label{AR}
{\rm
For a Dynkin quiver, the categories $\mathbf{R}$ and
$\mathbf{A}$ are canonically equivalent.  In fact,
both are equivalent to the full category of semi-simple perverse
sheaves on $\op{GRep}$.  However, in affine type $A$ (the case of
interest in this paper), they differ.  In terms of perverse sheaves,
the IC-sheaves which appear in $p_*\K_{\cQ({\bmuh})}$ for $\bmuh$
having complete flag type are those whose Fourier transform has
nilpotent support as well; for example, the constant sheaf on the
trivial representation with dimension vector $(1,\dots,1)$ cannot
appear.  Thus, in this case there are objects in $\mathbf{A}$ which
don't lie in  $\mathbf{R}$.
}
\end{remark}
Recall that the {\bf nilpotent Hall algebra} of the quiver $\Gamma$ is an algebra structure on the set of complex valued functions on the space of (isomorphism classes of) nilpotent representations, typically considered over a finite field. The structure constants are polynomial in the cardinality $q'$ of the field. If $[M]$ denotes the constant function on the class of the representation $M$ with dimension vector $\bd(M)$ then
\begin{eqnarray*}
[M]\cdot[N]=q^{\{\bd(M),\bd(N)\}}F_{M,N}^Q [Q],
\end{eqnarray*}
where $\{\bd',\bd''\}=
\sum_{i=1}^e\bd'_i(\bd_{i}''-\bd_{i+1}'')$ denotes the Euler form, $q'=q^2$ and the $F_{M,N}^Q$ are the Hall numbers. Hence it makes sense to consider $q$ as a formal parameter and define the {\bf generic Hall algebra} over the ring of Laurent polynomials $\mC[q,q^{-1}]$. Following Vasserot and Varagnolo, \cite{VVDuke}, we denote this algebra $U_e^-$.  By work of
Schiffmann \cite[\S 2.2]{SchifCyc} it is isomorphic as an algebra to
$U_q^-(\slehat)\otimes \bLa(\infty)$, where $\bLa(\infty)$ denotes the ring of symmetric polynomials. Identifying $\bLa(\infty)$ with $U_q^-(\mathbb{H})$, the lower half of a Heisenberg algebra,
this algebra can also be described as $U_q^-(\glehat)$ as in work of
Hubery \cite{Huberycentre}, \cite{Hub}.  This generic Hall algebra
has a basis given by characteristic functions on the isomorphism classes of nilpotent representations of $\Gamma$ and is naturally generated as algebra by the characteristic functions $\mathbf{f}_\bd$ on the classes of semi-simple representations (which we label by their dimension vectors $\bd$ following Lemma~\ref{ss}). Note that for instance $\mathbf{f}_i:=\mathbf{f}_{\alpha_i}=[S_i]$ and $\mathbf{f}_{\alpha_i+\alpha_{i+1}}=\mathbf{f}_{i+1}\mathbf{f}_{i}=[S_{i+1}\oplus S_{i}]$, whereas $\mathbf{f}_{(1,\dots,1)}$ is not in the subalgebra generated by the $[S_i]$'s, cf. Remark \ref{AR}.
The integral form $U_{e,\mZ}^-$ over $\mZ[q,q^{-1}]$ is given here by the
lattice generated by all $\mathbf{f}_\bd$'s, analogous to Lusztig's integral form for quantum groups, see \cite{SchifCyc}.

\begin{prop}\label{hall-grothendieck}
If $\operatorname{char}(\K)=0$, then
there is an isomorphism $K^0_q(\mathbf{A})\cong U_{e,\mZ}^-$, $[(\bd)]\mapsto \mathbf{f}_\bd$, of
$\mZ[q,q^{-1}]$-algebras
from the graded Grothendieck ring of $\mathbf{A}$ to the integral form of the
generic nilpotent Hall algebra of the cyclic quiver.
\end{prop}
The result is in fact also true for $\K$ of positive characteristic and can be proved using the usual technique of deforming to a
characteristic 0 discrete valuation ring. Since the general result is not
needed here, we omit it, but refer to \cite{Maksimau}.
\begin{proof}
  Fixing a prime $p$, there is a natural map from $K^0_q(\mathbf{A})$
  to $U_e^-|_{q=p}$.  This is given by
 applying the equivalence of Proposition \ref{s-s-perv}, and then
  sending the class of a semi-simple perverse sheaf to the function
  given by the super-trace of Frobenius on its stalks. This is a function on
  the points of $\op{Rep}$ over the field $\mathbb{F}_p$ and hence defines
  an element of the Hall algebra.  By the definition of the Hall
  multiplication and the Grothendieck trace formula, this is an
  algebra map. Since these super-traces are polynomial in $p$ (they are the
Poincar\'e polynomials of the quiver partial flag varieties), the
  coefficients of the expansion of this function in terms of the
  characteristic functions of orbits are also polynomial, and this
  assignment can be lifted to an algebra map $K^0_q(\mathbf{A})\to U_{e,\mZ}^-$. This map is obviously surjective, since the function for each
  intersection cohomology sheaf on a nilpotent orbit, and thus the characteristic
  function on the orbit, is in its image.  It is also injective,
  since when we expand any non-zero class in the Grothendieck group in
  terms of the classes of intersection cohomology sheaves, we must
  have a nonzero value of the corresponding function on the support of
  an intersection cohomology sheaf maximal (in the closure ordering)
  amongst those with non-zero coefficient.
Since $[(\bd)]$ corresponds to the skyscraper sheaf of the semi-simple
representation of dimension $\bd$, it is sent to the characteristic
function of that point.
\end{proof}

The monoidal structure on ${\bf A}$ and the usual monoidal structure $({\mathsf{Vect}_\K},\otimes_\K,\K)$ on the category of vector spaces are compatible in the following way:

\begin{lemma}
\label{tensor}
Let $\Phi_{\bd}: A_{\bd}\rightarrow \END(V_{\bd})$ be the representation from \eqref{V}. Then
$$\Phi_{\bc+\bd}((a|b))(v)= \Phi_{\bc}(a)(v_1)\otimes
\Phi_{\bc}(b)(v_2),$$ where $v$ is the image of $v_1\otimes v_2$ under
the canonical map $V_\bmuh\otimes V_{\blah}\rightarrow
V_{\bmuh\cup\blah}$.

That is, the functor $\mathbf{V}\colon \mathbf{A}\to \mathsf{Vect}_\K$
given by $\bmuh\mapsto V_\bmuh$ is monoidal.
\end{lemma}

\begin{proof}
This follows directly from the definitions, see \cite{CG97}.
\end{proof}

Thus, we can describe elements corresponding to vector compositions with a
large number of parts by looking at (the action) of the ones with a small number of parts.

\section{Diagrams and Demazure operators}
\label{sec:splitmerge}
In this section we describe a basis of the algebras $A_{\bf d}$ and elementary morphisms, called splits and merges. We give a geometric, algebraic and diagrammatical description of these maps.

Let $\blah',\blah\in \op{Comp}_e$ be residue data.
\begin{definition}
\label{merges}
We say that $\blah'$ is a {\bf merge} of
$\blah$ (and $\blah$ a {\bf
 split} of $\blah'$) at the index $k$ if
$\blah'=(\bla^{(1)},\cdots, \blah^{(k)}+\blah^{(k+1)},
\dots,\blah^{(r)})$.
\end{definition}
If $\blah'$ is a merge of $\blah$, then there is an associated correspondence
\begin{equation}
\tikz{
\node (a) at (-5,0) {$\cQ({\blah})$};
\node (b) at (0,1) {$\cQ({\blah,k})=\left\{ (V,f, F)\in\cQ({\blah})\mid f_i(F(i)^{k+1})\subset F(i+1)^{k-1}  \right\}$};
\node (c) at (5,0) {$\cQ({\blah'})$};
\draw[thick,->] (b) -- (a);
\draw[thick,->] (b) -- (c);
}
\end{equation}
where the left map is just the obvious inclusion and the right map is forgetting $F_k(i)$ for all vertices $1\leq i\leq e$ (and reindexing all subspaces in the flags
with higher indices).
Obviously the same variety defines also a
correspondence in the opposite direction (reading from right to left) which we associate to the {\bf split}. We are interested in the equivariant version:

\begin{definition}
For $\blah'$ a merge (resp. split) of $\blah$ at $k$, we
let $\blah\stackrel{k}{\to}\blah'$ or just $\blah \to\blah'$ denote the element of $A$ given by multiplication with the
equivariant fundamental class $[\cQ({\blah,k})]$ (resp.
$[\cQ({\blah',k})]$) pushed forward to $H_*^{BM}(\cH(\blah',\blah))$.
\end{definition}

 In the most obvious choice of grading conventions, pull-back by a map is of degree~0,
  and pushforward has degree given by minus the relative (real)
  dimension of the map (i.e.\ the dimension of the target minus the
  dimension of the domain). This normalization has the disadvantage of
  breaking the symmetry between splits and merges.  It is, for
  example, carefully avoided in \cite{KL1}.
  Instead, we use, \eqref{gradconv}, the perverse normalization of the constant sheaves
  which ``averages'' the degrees of pull-back and pushforward.
  Then the degree of
  convolving with the fundamental class of a correspondence is minus
  the sum of the relative (complex) dimensions of the two projection
  maps (note that for a correspondence over two copies of the same
  space, this agrees with the most obvious normalization). In particular, we get the same answer in the split and merge
  cases.

\begin{prop}\label{split-deg}
Let $\blah'$ be a merge or split of $\blah$ at index $k$. Then $\blah\to\blah'$ is homogeneous of degree
\begin{equation}\label{eq:degree}
\sum_{i=1}^e
  \bla^{(k)}_i(\bla^{(k+1)}_{i-1}-\bla^{(k+1)}_i)=:-\{\bla^{(k+1)},\bla^{(k)}\}.
\end{equation}

\end{prop}
\begin{proof}
If $\blah'$ is a merge of $\blah$, then
  \begin{itemize}
  \item the map $\cQ({\blah,k})\to \cQ({\blah'})$ is a smooth surjection
    with fiber given by the product of Grassmannians of
   $\bla^{(k)}_i$-dimensional planes in
    $\bla^{(k)}_i+\bla^{(k+1)}_i$-dimensional space, which has
    dimension $\sum
 \bla^{(k)}_i\bla^{(k+1)}_i$, and
\item the map $\cQ({\blah,k})\to \cQ({\blah})$ is a closed inclusion of
 codimension \[\sum_{i=1}^e \dim \Hom(F(i-1)^{k+1}/F(i-1)^{k},F(i)^{k}/F(i)^{k-1})=\sum_{i=1}^e
  \bla^{(k)}_i\bla^{(k+1)}_{i-1}.\]
  \end{itemize}
The result follows for merges and hence also for splits.
\end{proof}

\subsection{Explicit formulas for merges and splits}
\label{sec:calc-basic-conv}

We give now explicit formulas for elementary merges and splits. Consider the particular
choices for the vector compositions: $\cH({(\bc,\bd)},(\bc+\bd))\cong \cF{(\bc,\bd)}/G_{\bc+\bd}$. The variety $\cQ{(\bc+\bd)}$ is just a point, but equipped with the action of $G=G_{\bc+\bd}$. We want to describe how the fundamental classes of $\cH({(\bc,\bd)},(\bc+\bd))$ or $\cH((\bc+\bd),{(\bc,\bd)})$ (which are isomorphic as varieties, but different as correspondences) act on $V$ via Proposition \ref{faithful} and determine in this way the merge and split map. They are given by pullback followed by pushforward in equivariant cohomology via the diagram
\begin{eqnarray*}
\xymatrix{
H_*^{BM}(\cQ{(\bc,\bd)}/G)\ar@/^/[r]^{\iota^*}&H_*^{BM}(\cF{(\bc,\bd)}/G) \ar@/^/[l]^{\iota_*}\ar@/^/[r]^{q_*}&\ar@/^/[l]^{q^*}H_*^{BM}( \cQ(\bc+\bd)/G),
}
\end{eqnarray*}
where $\iota:\cF{(\bc,\bd)}\rightarrow  \cQ{(\bc,\bd)}$ is the zero section of the $G$-equivariant fibre bundle $\pi:\cQ{(\bc,\bd)}\longrightarrow \cF{(\bc,\bd)}$ and $q:\cF{(\bc,\bd)}\rightarrow \cQ(\bc+\bd)$ is the proper $G$-equivariant map given by forgetting the subspaces of dimension $c_i$ for any $i$.

\begin{prop}\label{euler-int}
Let $\bc$, $\bd\in\mZ_{\geq 0}^e$ non-zero. The following diagram commutes
\begin{eqnarray}
\xymatrix{
H^*(\cQ{(\bc,\bd)}/G)\ar@/^/[rr]^{q_*\circ \iota^*}\ar@{->}[d]_{\op{Borel}}
&&\ar@/^/[ll]^{\iota_*\circ q^*}H^*( \cQ{(\bc+\bd)}/G)\ar@{->}[d]_{\op{Borel}}\\
  \bLa(\bc,\bd)\ar@/^/[rr]^{\op{int}}
&&  \bLa(\bc+\bd)
\ar@/^/[ll]^{\op{E}}
}
\end{eqnarray}
where $E$ is the inclusion map from the total invariants $\bLa(\bc+\bd)$ into the invariants $\bLa(\bc,\bd)$ followed by multiplication with the  Euler class
\begin{equation}
\label{euler}
E:=\prod_{i=1}^e \prod_{j=1}^{c_{i+1}}\prod_{k=c_i+1}^{d_i+c_i}(x_{i+1,j}-x_{i,k}),
\end{equation}
and $\op{int}$ is the integration map which sends an element $f$ to the total invariant
\begin{equation}\label{int}
\displaystyle\sum_{w\in S_{\bc+\bd}} (-1)^{l(w)} w(f)\prod_{i=1}^e\frac{1}{c_i!d_i!}
\frac{\displaystyle w\bigg(\prod_{1\leq j<k\leq c_i} (x_{i,j}-x_{i,k})\prod_{c_i<\ell<m\leq c_i+d_i}(x_{i,\ell}-x_{i,m})\bigg)}
{\displaystyle \prod_{1\leq j<k\leq c_i+d_i} (x_{i,j}-x_{i,k})}
\end{equation}
where $\ell$ denotes the usual length function on the symmetric group.
\end{prop}
Those readers who are interested in the case where $\K$ is of small
positive characteristic might be worried about \eqref{int}, since it involves division by a scalar which may not be
invertible in $\K$; However, like
divided difference operators, these operations preserve integer valued
polynomials, and so are well-defined maps modulo $p$ for any prime $p$.

\begin{remark}
\label{conventions}{\rm
By convention, we set $E=1$ if either one of the products in \eqref{euler} is empty or one of the variables $x_{i-1,k}$ or $x_{i,j}$ does not exist.
The degrees of the maps $q_*\circ
\iota^*$ and $\iota_*\circ
q^*$ are again not the degrees which one would naively guess, but rather given by
the convention $\eqref{gradconv}$, in particular they are of the same degree.
 }
\end{remark}
\begin{proof}
Since $\pi^*$ is an isomorphism and $\iota$ the inclusion of the zero section, $\iota^*$ is also an isomorphism. On the other hand $q$ is the map to a point, hence $q^*$ is just the inclusion of the total invariants.  By the usual adjunction formula, $\iota^*\iota_*(a)={\bf e}\cup a$, where ${\bf e}$ is the Euler class of the vector bundle $\pi$. To see that the map $E$ is as asserted it is enough to verify the formula $E={\bf e}$ for the Euler class. The map $i$ is the inclusion of the zero section of the vector bundle $$\bigoplus_{i\in\mathbb{V}}\Hom(\cV_{i,d_i},\cV_{i+1,c_{i+1}}),$$ where $\cV=\oplus_{i=1}^e\cV_{i,d_i}$ is the tautological vector bundle on the moduli space of quiver representation with dimension vector ${\bd}$. As equivariant vector bundles over the maximal torus of $G_{\bd}$, we have a splitting into line bundles
\begin{eqnarray*}
\cV_{i+1,\bc}\cong \bigoplus_{j=1}^{c_{i+1}}\cL_{i+1,j}, && \cV_{i,\bd}\cong \bigoplus_{k=c_i+1}^{c_i+d_i}\cL_{i,k},
\end{eqnarray*}
where $\cL_{i,k}$ is the tautological line bundle for the corresponding weight space.  Thus, \begin{eqnarray*}
\Hom(\cV_{i,\bd},\cV_{i+1,\bc})&\cong & \bigoplus_{j=1}^{c_{i+1}}\bigoplus_{k=c_i+1}^{c_i+d_i}\cL_{i+1,j}\otimes \cL_{i,k}^*
\end{eqnarray*}
 and the formula \eqref{euler} for the Euler class follows.

On the other hand, $q$ is the projection from the partial flag variety $G_{\mathbf{a}}/P_{\mathbf{c},\mathbf{d}}$ to a point, where ${\mathbf{a}}={\mathbf{c}}+{\mathbf{d}}$. The formula for equivariant integration on the full flag variety is given by
\begin{eqnarray}
\label{fullflag}
\int_{G_{\mathbf{a}}/B}f&=&\frac{\displaystyle \sum_{S_{\mathbf{a}}}(-1)^{l(w)} w\cdot f}{\displaystyle \prod_{i=1}^e\prod_{j<k\leq a_i}(x_{i,j}-x_{i,k})}.
\end{eqnarray}
Thus, we have that
\small
\begin{eqnarray*}
\int_{\cQ(\bc,\bd)} f=
\int_{\cF(\bc,\bd)}\int_{P_{\bc,\bd}/B_{\bc+\bd}}\frac{1}{\prod_{i=1}^ec_i!d_i!}fD
=\int_{G_{\bc+\bd}/B_{\bc+\bd}} \frac{1}{\prod_{i=1}^ec_i!d_i!}fD,
\end{eqnarray*}
\normalsize
where $D=\prod_{j<k\leq c_i} (x_{i,j}-x_{i,k})\prod_{c_i<\ell<m\leq c_i+d_i}(x_{i,\ell}-x_{i,m})$. The integration formula follows then from \eqref{fullflag}.
\end{proof}

\subsection{Demazure operators}
The splitting and merging maps can be described algebraically via Demazure operators acting on polynomial rings.
The $i$th {\bf Demazure operator} or {\bf difference operator} $\Delta_i=\Delta_{s_i}$ acts on $\K[x_1,\ldots x_n]$ by sending $f$ to $\frac{f-s_i(f)}{x_i-x_{i+1}}$, where $s_i=(i,i+1)$ denotes the simple transposition acting by permuting the $i$th and $(i+1)$th variable.
 If $w=s_{i_1}s_{i_2}\ldots s_{i_l}$ is a reduced expression of $w\in S_n$ we define $\Delta_w=\Delta_{i_1}\Delta_{i_2}\cdots s_{i_l}$. This is independent of the reduced expression, see \cite{Demazure}. If $G$ is a product of symmetric groups we denote by $w_0\in G$ the longest element.
Demazure operators satisfy the twisted derivation rule $$\Delta_{i}(fg)=\Delta(f)g+s_i(f)\Delta_i(g)$$ and more generally for a reduced expression $w=s_{i_1}s_{i_2}\ldots s_{i_l}$ the formula
\begin{eqnarray}
\label{Demazureprop}
\Delta_{i_1}\Delta_{i_2}\cdots\Delta_{i_l}(fg)&=&
\sum A_{i_1}A_{i_2}\cdots A_{i_l}(f) B_{i_1}B_{i_2}\cdots B_{i_l}(g),
\end{eqnarray}
where the sum runs over all possible choices of either $A_j=\Delta_j$ and  $B_j=\op{id}$ or
$A_j=s_j$ and $B_j=\Delta_j$ for each $1\leq r$.

\begin{prop} Let $(\bc,\bd)$ be a vector composition.
\label{Demazure}
\begin{enumerate}
\item \label{Dem1} Assume $\bc+\bd=(c_i+d_i)\alpha_i$ for some $i$, then $$\op{int}(f)=\Delta_{w_0^{c_i,d_i}},$$
where $w_0^{c_i,d_i}\in S_{c_i+d_i}$ denotes coset representative of $w_0$ in $S_{c_i+d_i}/(S_{c_i}\times S_{d_i})$  of minimal length. In general, $\op{int}(f)$ is a product of pairwise commuting Demazure operators $w_0^{c_i,d_i}$, one for each $i$.
\item \label{Dem2}
Merging successively from a vector composition $(\alpha_i,\alpha_i,\ldots,\alpha_i)$ of length $r$ to $r\alpha_i$ equals the Demazure operator for $w_0\in S_r$, in formulas
    \begin{equation}
    \label{Demlong}
    \Delta_{w_0}(f)=\sum_{w\in S_r}(-1)^{l(w)}w(f)\frac{1}{\prod_{1\leq i<j\leq r}(x_i-x_j)}.
    \end{equation}
\item \label{Dem3}
The split $\bc+\bd$ into $\bc$ and $\bd$, where either $c_{i+1}d_i=0$ for all $i$ or $\bc+\bd=(c_i+d_i)\alpha_i$ for some $i$, is just the inclusion from $\bLa(\bc+\bd)$ to $\bLa(\bc,\bd)$.
\end{enumerate}
\end{prop}

\begin{proof}
Part \eqref{Dem3} is obvious, since $E=1$ by Remark \ref{conventions}. The second statement is clear for $r=1$ and $r=2$. The successive merge of the first $r-1$ $\alpha_i$'s is given by induction hypothesis, hence we only have to merge with the last $\alpha_i$ and obtain
\begin{eqnarray*}
f&\mapsto& \sum_{y\in S_{r-1}\times S_1} (-1)^{l(y)}\frac{y(f)}{(r-1)!}\frac{1}{\prod_{1\leq i<j\leq r-1}(x_i-x_j)}=:P\\
&\mapsto&\sum_{z\in S_r/S_{r-1}\times S_1} (-1)^{l(z)}z(P)=\sum_{z,y}(-1)^{l(z)+l(y)}zy(f)\frac{1}{\prod_{1\leq i<j\leq r}(x_i-x_j)}.
\end{eqnarray*}
Then \eqref{Dem2} follows from  the general formula for $\Delta_{w_0}$, see \cite[10.12]{Fulton}. Associativity of the merges and formula \eqref{Dem2} gives $\op{int}(f) \Delta_{w_0(c_i,d_i)} =\Delta_{w_0(c_i+d_i)}=\Delta_{w_0(c_i,d_i)}\Delta_{w_0(c_i,d_i)}$, where $w_0(c_i,d_i)$ and $w_0(c_i+d_i)$ are the longest elements in $S_{c_i}\times S_{d_i}$ and $S_{c_i+d_i}$ respectively. We have $\op{int}(f)=\Delta_{w_0(c_i,d_i)}$, since $\Delta_{w_0(c_i,d_i)}$ surjects to the $S_{c_i}\times S_{d_i}$-invariants, and so \eqref{Dem1} follows.
\end{proof}

\subsection{The pictorial interpretation}
\label{pictorial}
As shown in \cite{VV}, the quiver Hecke algebra $R(\bd)$ is isomorphic
to the diagram algebra introduced by Khovanov and Lauda in \cite{KL1} (modulo the mentioned small differences in signs). This result allows to turn rather involved computations in the convolution algebra into a beautiful diagram calculus. Motivated by these ideas, we present now a graphical calculus for the algebra $A$ where the split and merge maps from Proposition \ref{euler-int} are  displayed as trivalent graphs. We will always read our diagrams from bottom to top. We represent

\begin{itemize}
\item the usual (vertical) algebra multiplication as vertical stacking of diagrams,
\item horizontal multiplication as horizontal stacking of diagrams,
\item the idempotent $e_{\bmuh}$ as a series of lines labeled with the parts of the vector composition or equivalently the blocks of the residue sequence \eqref{res},
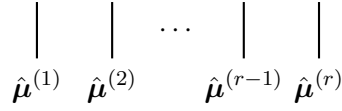
\[\tikz[thick,xscale=2.5,yscale=1.5]{
\draw (0,0) -- (0,.5) node[below,at start]{$\bmuh^{(1)}$};
\draw (.4,0) -- (.4,.5) node[below,at start]{$\bmuh^{(2)}$};
\node at (.75,.25) {$\cdots$};
\draw (1.1,0) -- (1.1,.5) node[below,at start]{$\bmuh^{(r-1)}$};
\draw (1.5,0) -- (1.5,.5) node[below,at start]{$\bmuh^{(r)}$};
}\]
\item the morphism $(\bc,\bd)\linj (\bc+\bd)$ as a joining of two strands $\bc,\bd$,
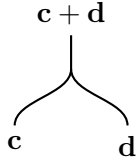
\[
\tikz[thick,xscale=2.5,yscale=1.5]{
\draw (0,0) node[below] {$\bc$} to [out=90,in=-90](.3,.5)
(.6,0) node[below] {$\bd$} to [out=90,in=-90] (.3,.5)
(.3,.5) -- (.3,.8) node[above] {$\bc+\bd$};
}\]

\item the morphism $(\bc+\bd)\linj (\bc,\bd)$ as its mirror image,
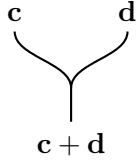
\[\tikz[thick,xscale=2.5,yscale=-1.5]{
\draw (0,0) node[above] {$\bc$} to [out=90,in=-90](.3,.5)
(.6,0) node[above] {$\bd$} to [out=90,in=-90] (.3,.5)
(.3,.5) -- (.3,.8) node[below] {$\bc+\bd$};
}\]
\item multiplication by a polynomial is displayed by putting a box containing the polynomial.
\end{itemize}

A typical element of $A$ is obtained by horizontally and vertically composing these morphisms.
The composition of a merge followed by a split of the form $(\bc,\bd)\to (\bc+\bd) \to (\bd,\bc)$ is also abbreviated as a {\bf crossing} and denoted $(\bc,\bd)\to (\bd,\bc)$.

\begin{remark}
\label{KL}
{\rm
Our calculus is an extension of the graphical calculus of \cite{KL1}:
given a crossing as in the Khovanov-Lauda picture, we interpret it as merge-split of the $k$th and $(k+1)$th strands:
\begin{eqnarray}
\label{cross}
\tikz[thick,xscale=2.5,yscale=1.5, baseline=0.8cm]{
\draw (0,0) node[below]{$\alpha_i$} to(.6,1) node[above]{$\alpha_i$}
(.6,0) node[below]{$\alpha_j$} to (0,1) node[above]{$\alpha_j$};
\pgftransformxshift{50}
\pgftransformyshift{-5}
\draw (0,0) node[below] {$\alpha_i$} to [out=90,in=-90](.3,.5)
(.6,0) node[below]{$\alpha_j$} to [out=90,in=-90] (.3,.5)
(.3,.5) -- (.3,.8) node[right, midway]{$\alpha_i+\alpha_j$}
(.3,.8) to [out=90,in=-90](.6,1.3) node[above]{$\alpha_i$}
(.3,.8) to [out=90,in=-90] (0,1.3) node[above]{$\alpha_j$};
\node at (-0.5,0.7){$\rightsquigarrow$};
}
\end{eqnarray}

Assume it involves the $a$th $\alpha_i$ and $b$th $\alpha_j$ in the residue sequence.
\begin{itemize}
\item If $j\not=i,i+1$, then our map just flips the tensor factors
  $\K[x_{i,a}]\otimes\K[x_{j,b}]\mapsto
  \K[x_{j,b}]\otimes\K[x_{i,a}]$.
\item  If $i=j$, then we associate the Demazure operator $\Delta_k$ as
  in \cite{KL1}.
\item  For $j=i+1$, we multiply by $x_{i+1,a}-x_{i,b}$ followed by flipping the tensor factors.
\end{itemize}
In each case, this agrees with the action on $\mathcal{P}o\ell$ defined
in \cite{KLII}, though Khovanov and Lauda have a single alphabet of
variables, which they index by their left-position, as opposed to
having separate alphabets for each node of the Dynkin diagram.  We
believe that when one incorporates non-unit dimension vectors, the
latter convention is more convenient. Lemma~\ref{split-deg} implies that crossings have
degree $1$, $-2$, or $0$ according to the cases  $j=i\pm 1$, $i=j$ or
$j\not=i,i\pm 1$ respectively. Given a single strand labeled with the
$a$th $\alpha_j$ we denote, following \cite{KL1} multiplication with
$x_{j,a}^R$ also by decorating the strand with $R$ dots.}
\end{remark}

To keep track of the permutation $w$ appearing in the Demazure operators $\Delta_w$, it will be useful to interpret the numbers occurring in a residue sequence as colors from the chart $\{1,\ldots,e\}$. If the sequence is of length $d$, then permutations in $S_d$ permuting only inside the colors can be viewed as elements of $S_{\bd}$, where $d_i$ is the number of $i$'s appearing. We want to associate to each idempotent, elementary split or merge such a permutation as follows:
\begin{equation}
\begin{tikzpicture}[thick,baseline=9pt]
\draw[blue] (0,0) node[below] {1} -- (0,1)
(0,1) node[above] {1} -- (0,1)
(0.3,0) node[below] {1} -- (0.3,1)
(0.3,1) node[above] {1} -- (0.3,1)
(0.6,0) node[below] {1} -- (0.6,1)
(0.6,1) node[above] {1} -- (0.6,1);
\draw[red,dashed] (0.9,0) node[below] {2} -- (0.9,1)
(0.9,1) node[above] {2} -- (0.9,1)
(1.2,0) node[below] {2} -- (1.2,1)
(1.2,1) node[above] {2} -- (1.2,1);
\draw (1.5,-0.1) -- (1.5,-0.4);
\draw (1.5,1.1) -- (1.5,1.4);
\draw[blue](1.8,0) node[below] {1} -- (1.8,1)
(1.8,1) node[above] {1} -- (1.8,1)
(2.1,0) node[below] {1} -- (2.1,1)
(2.1,1) node[above] {1} -- (2.1,1);
\draw[red,dashed] (2.4,0) node[below] {2} -- (2.4,1);
\draw[red,dashed] (2.4,1) node[above] {2} -- (2.4,1);
\end{tikzpicture}
\quad\quad
\begin{tikzpicture}[thick,baseline=9pt]
\draw[blue] (0,0) node[below] {1} -- (0,1) node[above] {1}
(0.3,0) node[below] {1} -- (0.3,1) node[above] {1}
(0.6,0) node[below] {1} -- (0.6,1) node[above] {1}
(0.9,0) node[below] {1} -- (1.8,1) node[above] {1}
(1.2,0) node[below] {1} -- (2.1,1) node[above] {1};
\draw (1.5,1.1) -- (1.5,1.5);
\draw[red,dashed] (1.5,0) node[below] {2} -- (0.9,1) node[above] {2}
(1.8,0) node[below] {2} -- (1.2,1) node[above] {2}
(2.1,0) node[below] {2} -- (2.4,1) node[above] {2};
\end{tikzpicture}
\quad\quad
\begin{tikzpicture}[thick,baseline=9pt]
\draw[blue] (0,0) node[below] {1} -- (0.6,1) node[above] {1}
(0.3,0) node[below] {1} -- (0.9,1) node[above] {1}
(0.6,0) node[below] {1} -- (1.2,1) node[above] {1};
\draw[red,dashed] (0.9,0) node[below] {2} -- (1.8,1) node[above] {2}
(1.2,0) node[below] {2} -- (2.1,1) node[above] {2};
\draw (1.5,-0.1) -- (1.5,-0.4);
\draw[blue](1.8,0) node[below] {1} -- (0,1) node[above] {1}
(2.1,0) node[below] {1} -- (0.3,1) node[above] {1};
\draw[red,dashed] (2.4,0) node[below] {2} -- (1.5,1) node[above] {2};
\end{tikzpicture}\label{colored}
\end{equation}
To an idempotent corresponding to a residue sequence of length $d$ we attach the identity element in $S_d$ which corresponds to the identity element in  $S_{\bd}$. To a split $(\bc+\bd)\linj (\bc,\bd)$
we associate  the identity element in  $S_{\bc+\bd}$. When viewed as a permutation in $S_{c+d}$ it sends  $a$ to $b$ if we find the $k$th $j$
on place $a$ in the residue sequence of $(\bc+\bd)$, and on place $b$ in the residue sequence for $(\bc,\bd)$. To an elementary merge $M:(\bc,\bd)\linj (\bc+\bd)$ we associate the product $w=w(M)=\prod_{i=1}^e w_0^{c_i,d_i}$; e.g. $(1,3,5,2,4)(6,8,7)\in S_{\bf d}=S_5\times S_3\subset S_8$ (written in cycle decomposition) to the merge in \eqref{colored}.
Note that $\Delta_{\omega(M)}$ for a merge $M$ is precisely the Demazure operator from Lemma~\ref{Demazure}. \\

Whereas we see this diagrams just as a tool of bookkeeping, they get interpreted in \cite[\S 3.3]{wKLR} as maps from the quiver Schur algebra into a {\bf weighted KLR algebra}.\\

By \eqref{euler}, a split corresponds to an inclusion followed by multiplication with some polynomial. For a composition $X=M_1M_2\cdots M_t$ of splits and merges let $\omega(X)= \omega(M_1)\omega(M_2)\cdots \omega(M_t)$, where $\omega$ of a split is, compatible with the definition above, just the identity permutation. For a finite linear combination $X'$ of such $X$'s we let $\omega(X')$ be the sum of all $\omega(X)$'s of maximal length. The following relations are basic relations in the algebra $A$ which we call {\bf braid relations}.

\begin{prop}
\label{easyrel}
Let  $\bb, \bc, \bd$ be vector compositions.
\begin{enumerate}
\item Let $E$ be the Euler class attached to the splitting $(\bc+\bd)\linj (\bc,\bd)$ and $\Delta_w=\Delta_{w_0^{c_1,d_1}}\cdots \Delta_{w_0^{c_e,d_e}}$. Then
\begin{eqnarray*}
(\bc+\bd)\linj (\bc,\bd)\linj (\bc+\bd)&=&
\Delta_w(E)  \big((\bc+\bd)\stackrel{\op{id}}{\linj} (\bc+\bd)\big)
\end{eqnarray*}
or equivalently in terms of diagrams:

\begin{center}
\tikz[thick,xscale=2.5,yscale=1.5]{
\draw (0,0) node[left] {$\bc$} to [out=90,in=-90](.3,.5)
(.6,0) node[right] {$\bd$} to [out=90,in=-90] (.3,.5)
(.3,.5) -- (.3,.8) node[above] {$\bc+\bd$}
(0,0) to [out=-90,in=90](.3,-.5)
(.6,0) to [out=-90,in=90](.3,-.5)
(.3,-.5) -- (.3,-.8) node[below] {$\bc+\bd$}
node at (1,0) {$=$}
node at (2.2,0) {$\quad\; \Delta_w(E)$}
(2.3,.8) node[above] {$\bc+\bd$} -- (2.3,0.2)
(2.3,-0.2) -- (2.3,-.8) node[below] {$\bc+\bd$}
;
\draw (2,-0.2) rectangle (2.6,0.2);
}
\end{center}
In particular, if $X:=(\bc,\bd)\rightarrow (\bd,\bc)$ is a crossing then $X^2=fX$ for some (possibly zero) polynomial $f$.
\item
Let
\begin{eqnarray*}
X_1&:=&(\bb,\bc,\bd)\linj (\bc,\bb,\bd)\linj (\bc,\bd,\bb)\linj (\bd,\bc,\bb)\\
X_2&:=&(\bb,\bc,\bd)\linj (\bb,\bd,\bc)\linj (\bd,\bb,\bc)\linj (\bd,\bc,\bb)\\
X_3&:=&(\bb,\bc,\bd)\linj (\bb+\bc+\bd)\linj (\bd,\bc,\bb)
\end{eqnarray*}
Then $X_1=X_2+R$ and $X_1=X_3+R'$,
where $\omega(R),\omega(R')<\omega(X_1)=\omega(X_2)$.
Hence, up to lower order terms $X_1\equiv X_2\equiv X_3$, in diagrams

\begin{center}
\begin{tikzpicture} [thick,xscale=1.5,yscale=0.7]
\draw
(0,0) node[above] {$\bb$} to [out=-90,in=90](.3,-.5)
(.6,0) node[above] {$\bc$} to [out=-90,in=90](.3,-.5)
(1.2,0) node[above] {$\bd$} -- (1.2,-1)
(.3,-.5) to [out=-90,in=90](.6,-1)
(.3,-.5) to [out=-90,in=90] (0,-1)
(.6,-1) to [out=-90,in=90](.9,-1.5)
(1.2,-1) to [out=-90,in=90](.9,-1.5)
(0,-1) -- (0,-2)
(.9,-1.5) to [out=-90,in=90](.6,-2)
(.9,-1.5) to [out=-90,in=90] (1.2,-2)
(0,-2) to [out=-90,in=90](.3,-2.5)
(.6,-2) to [out=-90,in=90](.3,-2.5)
(.3,-2.5) to [out=-90,in=90](0,-3) node[below] {$\bb$}
(.3,-2.5) to [out=-90,in=90] (0.6,-3) node at (0.6,-3.5) {$\bc$}
(1.2,-2) -- (1.2,-3) node[below] {$\bd$};
\node at (1.5,-1.5) {$\equiv$};
\begin{scope}[xshift=-0.5cm]
\draw
(2.5,0) node[above] {$\bb$} -- (2.5,-1)
(3.1,0) node[above] {$\bc$} to [out=-90,in=90](3.4,-.5)
(3.7,0) node[above] {$\bd$} to [out=-90,in=90](3.4,-.5)
(3.4,-.5) to [out=-90,in=90](3.1,-1)
(3.4,-.5) to [out=-90,in=90] (3.7,-1)
(2.5,-1) to [out=-90,in=90](2.8,-1.5)
(3.1,-1) to [out=-90,in=90] (2.8,-1.5)
(2.8,-1.5) to [out=-90,in=90](2.5,-2)
(2.8,-1.5) to [out=-90,in=90] (3.1,-2)
(3.7,-1) -- (3.7,-2)
(3.1,-2) to [out=-90,in=90](3.4,-2.5)
(3.7,-2) to [out=-90,in=90] (3.4,-2.5)
(2.5,-2) -- (2.5,-3) node[below] {$\bb$}
(3.4,-2.5) to [out=-90,in=90](3.7,-3) node[below] {$\bd$}
(3.4,-2.5) to [out=-90,in=90] (3.1,-3) node at (3.1,-3.5) {$\bc$};
\end{scope}
\node at (3.8,-1.5) {$\equiv$};
\begin{scope}[xshift=4cm]
\draw
(0,0) node[above] {$\bd$} to [out=-90,in=90](.3,-.5)
(.6,0) node[above] {$\bc$} to [out=-90,in=90](.3,-.5)
(1,0) node[above] {$\bb$} -- (1,-1)
(.3,-.5) to [out=-90,in=90](.3,-1)
(.3,-1) to [out=-90,in=90](.7,-1.5)
(1,-1) to [out=-90,in=90](.7,-1.5)
(.7,-1.5) to [out=-90,in=90](.3,-2)
(.7,-1.5) to [out=-90,in=90] (1,-2)
(.3,-2) to [out=-90,in=90](.3,-2.5)
(.3,-2.5) to [out=-90,in=90](0,-3) node[below] {$\bb$}
(.3,-2.5) to [out=-90,in=90] (0.6,-3) node at (0.6,-3.5) {$\bc$}
(1,-2) -- (1,-3) node[below] {$\bd$};
\end{scope}
\node at (5.7,-1.5) {$=:$};
\begin{scope}[xshift=6cm]
\draw
(0,0) node[above] {$\bd$} to [out=-90,in=90](.5,-1)
(.5,0) node[above] {$\bc$} to [out=-90,in=90](.5,-1)
(1,0) node[above] {$\bb$} to [out=-90,in=90](.5,-1)
(.5,-2) to [out=-90,in=90](0,-3) node[below] {$\bb$}
(.5,-2) to [out=-90,in=90] (0.5,-3) node at (0.6,-3.5) {$\bc$}
(.5,-2) to [out=-90,in=90] (1,-3) node[below] {$\bd$}
(.5,-1) -- (.5,-2);
\end{scope}
\end{tikzpicture}
\end{center}

\excise{\begin{center}
\tikz[thick,xscale=1.5,yscale=0.7]{
\draw
(0,0) node[above] {$\bb$} to [out=-90,in=90](.3,-.5)
(.6,0) node[above] {$\bc$} to [out=-90,in=90](.3,-.5)
(1.2,0) node[above] {$\bd$} -- (1.2,-1)
(.3,-.5) to [out=-90,in=90](.6,-1)
(.3,-.5) to [out=-90,in=90] (0,-1)
(.6,-1) to [out=-90,in=90](.9,-1.5)
(1.2,-1) to [out=-90,in=90](.9,-1.5)
(0,-1) -- (0,-2)
(.9,-1.5) to [out=-90,in=90](.6,-2)
(.9,-1.5) to [out=-90,in=90] (1.2,-2)
(0,-2) to [out=-90,in=90](.3,-2.5)
(.6,-2) to [out=-90,in=90](.3,-2.5)
(.3,-2.5) to [out=-90,in=90](0,-3) node[below] {$\bb$}
(.3,-2.5) to [out=-90,in=90] (0.6,-3) node at (0.6,-3.5) {$\bc$}
(1.2,-2) -- (1.2,-3) node[below] {$\bd$};
\node at (1.9,-1.5) {=};
\draw
(2.5,0) node[above] {$\textbf{b}$} -- (2.5,-1)
(3.1,0) node[above] {$\bc$} to [out=-90,in=90](3.4,-.5)
(3.7,0) node[above] {$\bd$} to [out=-90,in=90](3.4,-.5)
(3.4,-.5) to [out=-90,in=90](3.1,-1)
(3.4,-.5) to [out=-90,in=90] (3.7,-1)
(2.5,-1) to [out=-90,in=90](2.8,-1.5)
(3.1,-1) to [out=-90,in=90] (2.8,-1.5)
(2.8,-1.5) to [out=-90,in=90](2.5,-2)
(2.8,-1.5) to [out=-90,in=90] (3.1,-2)
(3.7,-1) -- (3.7,-2)
(3.1,-2) to [out=-90,in=90](3.4,-2.5)
(3.7,-2) to [out=-90,in=90] (3.4,-2.5)
(2.5,-2) -- (2.5,-3) node[below] {$\textbf{b}$}
(3.4,-2.5) to [out=-90,in=90](3.7,-3) node[below] {$\bd$}
(3.4,-2.5) to [out=-90,in=90] (3.1,-3) node at (3.1,-3.5) {$\bc$}
;}
\end{center}
}
\end{enumerate}
\end{prop}
At the moment,  the lower order terms are beyond any satisfactory control.

\begin{proof}
Let $x=w_0^{c_1,d_1}\cdot{w_0^{c_e,d_e}}$. Recall that $\Delta_i(gf)=\Delta_i(g)f+s_i(g)\Delta_i(f)$ for any $f,g$ and $\Delta_i(f)=0$ if $f$ is $s_i$-invariant. Hence $\Delta_x(Ef)=\Delta_x(E)f$ for any total invariant polynomial $f$ and the first claim follows.

Let now $x_1=w(X_1)$ and $x_2=w(X_2)$. Then obviously $x_1=x_2$ and $X_1=\alpha \Delta_{x_1}+R_1$ and $X_2=\beta \Delta_{x_2}+R_2$ for some polynomials $\alpha$, $\beta$ and $R_1$, $R_2$ with $\omega(R_1), \omega(R_2)<\omega(X_1)=\omega(X_2)$. We still have to verify that $\alpha=\beta$ up to terms with $\omega$ of smaller length and the definition of \eqref{euler}. Let $X_1$ be the composition $E_3\circ M_3\circ E_2\circ M_2\circ E_1\circ M_1$ where the $M_i$ denotes the $i$th merge and $E_i$ denotes the multiplication with an Euler class $E(i)$. We verify the claim by invoking the definitions and the equality $\Delta_i(gf)=s_i(g)\Delta_i(f)+\Delta(g)f$. First note that up to lower order terms
\begin{eqnarray*}
M_2\circ E_1&=&w\left(\prod_{i=1}^e \prod_{ 1\leq r\leq c_{i+1}\atop c_i+1\leq s\leq b_i+c_i}
(x_{i+1,r}-x_{i,s})\right) M_2,
\end{eqnarray*}
where $w$ is the permutation associated with $M_2$ via \eqref{colored}, hence
\begin{eqnarray*}
M_2\circ E_1&=&\prod_{1\leq r\leq c_{i+1}\atop c_i+d_i+1\leq s\leq b_i+c_i+d_i}
(x_{i+1,r}-x_{i,s})M_2.
\end{eqnarray*}
Repeating these type of calculations we obtain
\begin{eqnarray*}
X_1&=&\prod_{i=1}^e\left(
\prod_{1\leq r\leq d_{i+1}\atop {d_{i}+1\leq s\leq {c_{i}+d_i}}}
(x_{i+1,r}-x_{i,s})
\prod_{1\leq r\leq d_{i+1}\atop {c_i+d_i+1}\leq s\leq {b_{i}+c_i+d_i}}
(x_{i+1,r}-x_{i,s})\right.\\
&\times&
\left.\prod_{d_{i+1}+1\leq r\leq c_{i+1}+d_{i+1}\atop c_i+d_i+1\leq s\leq b_{i}+c_i+d_i}
(x_{i+1,r}-x_{i,s})
\right)\;M_3\circ M_2\circ M_1\\
X_2&=&\prod_{i=1}^e\left(\prod_{d_{i+1}+1\leq r\leq c_{i+1}+d_{i+1}\atop c_i+d_i+1\leq s\leq b_{i}+c_i+d_i}
(x_{i+1,r}-x_{i,s})\right. \left.\prod_{1\leq r\leq d_{i+1}\atop c_i+d_i+1\leq s\leq b_{i}+c_i+d_i}
(x_{i+1,r}-x_{i,s})
\right.\\
&\times&\left.
\prod_{1\leq r\leq d_{i+1}\atop d_i+1\leq s\leq c_{i}+d_i}
(x_{i+1,r}-x_{i,s})\right)\;M_3\circ M_2 \circ M_1
\end{eqnarray*}
and therefore $\alpha=\beta$. Note also that if $M_i=\Delta_{w_i}$ then $M_3M_2M_1$ is the Demazure operator corresponding to the product $w=w_3w_2w_1$ and the last claim follows as well.
\end{proof}

\begin{ex} Let $e=3$ and $\bb=(2,1,0), \bc=(1,1,0), \bd=(1,1,0)$. Then $X_1\equiv X_2\equiv X_3: (\bb,\bc,\bd)\linj (\bd,\bc,\bb)$
equal $$(x_{2,1}-x_{1,2})(x_{2,1}-x_{1,3})(x_{2,1}-x_{1,4})(x_{2,2}-x_{1,3})(x_{2,2}-x_{1,4})
\Delta_{s_{1}s_2s_{3}s_{1}s_{2}}\Delta_{s_{\overline{1}}s_{\overline{2}}s_{\overline{1}}}$$ up to lower order terms where $s_i$ (respectively $\overline{s}_i$) denotes the $i$th transposition for the strands colored by $1$ (respectively by $2$) in $S_\bd=S_5\times S_3$.
\end{ex}

%\begin{ex}
%Recall our previous example, where $e=2$, $\bmuh=((2,0),(1,1))$ and
%  $\blah=((2,0), (0,1),(1,0))$, $w\in S_3\times S_1$ is $\big(\left(
%  \begin{smallmatrix}
%    1 & 2 & 3 \\
%    2 & 3 & 1
%  \end{smallmatrix}\right),\left(
%  \begin{smallmatrix}
%    1\\
%    1
%  \end{smallmatrix}\right)\big)$ and $p(w)=\left(
%  \begin{smallmatrix}
%    1 & 2 & 3 & 4\\
%    2 & 3 & 4 & 1
%  \end{smallmatrix}\right)$.  Then pictorially, the element
%$\bmuh\overset{w;1}\Longrightarrow\blah$ looks like
%\[
%\tikz[thick,scale=1.5]{
%\draw (0,0) to (0,2) node[below,at start]{$(2,0)$} node[above,at end]{$(2,0)$};
%\draw (0,.2) to[out=90,in=-90] (2,2)  node[above,at end]{$(1,0)$};
%\draw (2,0)--  node[below,at start]{$(1,1)$} (2,.2) to [out=90,in=-90] (0,1.8);
%\draw (2,.2) to [out=90,in=-90] (1,1.8) -- (1,2)  node[above,at end]{$(0,1)$};
%}
%\]
%\end{ex}

\subsection{An explicit basis of $A$}
We construct and describe explicitly a basis of $A$ by geometric arguments and then interpret it diagrammatically.  

We again fix a dimension vector $\bd$. Let $\bmuh,\blah\in\Compe(\bd)$ with residue sequences
$\res\bmuh$ and $\res\blah$; and let $S^\bmuh$ and $S^{\blah}$
be the subgroups of $S_{\bd}$ preserving the blocks. Note
that $\bmuh$ is determined uniquely by $\res\bmuh$ and $S^\bmuh$.  Now,
fix a permutation $p\in S_d$ such that $p(\res\bmuh)=\res\blah$ (hence $p\in S_{\bd}$) and
let $p_-$ be its shortest double coset representative in $S^\blah\backslash S_{\bd}/S^\bmuh$.

Note that in the graphical pictures \eqref{colored} an element $p_-\in S_{\bd}$ is a shortest coset representative in
$S^\bmuh\backslash S_{\bd}/S^\blah$ if and only if strands with the same color which end or start in the same block do not cross. In particular, idempotents, splits and merges correspond to shortest double coset representatives.

To the triple $\bmuh,\blah,p_-$ we associate a composition of merges and splits from $\bmuh$ to $\blah$ as follows: First, let $\bmuh'$ be the unique vector composition with residue sequence
\begin{eqnarray*}
\res(\bmuh')&=&p_-^{-1}(\res\blah)=p_-^{-1}p(\res\bmuh),
\end{eqnarray*}
such that
$S^{\bmuh'}=S^{\bmuh}\cap p_-S^{\blah}p_-^{-1}$,
and similarly $\blah'$ the unique vector
composition so that \[\res(\blah')=p_-(\res\bmuh)=p_-p^{-1}(\res\blah)\]
and $S^{\blah'}=S^{\blah}\cap p_-^{-1}S^{\bmuh}p_-$. In other words we first refine the blocks of  $\res\bmuh$ into $\res\bmuh'$ as coarse as possible such that $p_-$ sends all elements in a block of $\res\bmuh$ to the same block of $\res\blah$. Similarly refine $\blah$ into $\blah'$ again as coarse as possible such that $p_-^{-1}$ sends all elements in a block to the same block. For instance, if $\res\bmuh=1,2|3,3|2|1,2|1$  and $\res\blah= 1,1,2,3|1,2,3|2$ then
$\res\bmuh'=1,2|3|3|2|1,2|1$ and $\blah'=1,2|3|1|3|1.2|2$ with the permutation $p_-$ displayed in Figure~\ref{Belement}.

\begin{lemma}
\label{defq}
  The vector compositions $\bmuh',\blah'$ have the same length (i.e. number of blocks), say $\ell$, and $\blah'$ can be obtained from $\bmuh'$ by some permutation $q\in S_\ell$ of its blocks inducing $p_-$ on the residue sequences.
\end{lemma}
\begin{proof}
  Each part of $\bmuh'$ is given by the numbers
  of 1's, 2's, etc. in a given block of the residue sequence for
  $\bmuh$ which are sent to a fixed block of $\blah$. It is obtained by refining
  $\bmuh$ according to the blocks of $\res\blah$ (read from left to right) they are sent to.
On the other hand, the parts $\blah'$ have the same description, just
reversing the roles of $\blah$ and $\bmuh$; thus, we just change the
order of the blocks, which gives the permutation $q\in S_\ell$ as asserted.
\end{proof}

For each $p\in S_{\blah}\backslash S_{\bd}/S_{\bmuh}$, we have the
associated permutation $q$ as in Lemma~\ref{defq} and from now on fix a reduced
decomposition $q=s_{i_1}\cdots s_{i_\ell}$ in terms of simple transpositions $s_i=(i,i+1)\in S_\ell$. We also fix the
morphism $\bmuh'\overset{q}\linj\blah'$ given by the corresponding composition of crossings and define
\begin{eqnarray}
\label{mula}
\bmuh \overset{p;1}\Longrightarrow \blah:=\bmuh\linj \bmuh'\overset{q}\linj\blah'\linj \blah.
\end{eqnarray}
Of course, this definition depends on our choice of reduced expression for $q$. For $h\in\Lambda(\bmuh')$ let $\bmuh'\overset{h}\linj\bmuh'$ be multiplication with $h$ and set
\[\bmuh \overset{p;h}\Longrightarrow \blah:=\bmuh\linj \bmuh'\overset{h}\linj\bmuh'\overset{q}\linj\blah'\linj \blah.\]

\begin{thm}\label{basis}
The morphisms \( \bmuh\overset{p;h}\Longrightarrow\blah\) generate $A$ as a $\K$-vector space.  In fact, if we let range
\begin{itemize}
\item $(\bmuh,\blah)$ over all ordered pairs of vector compositions of type $e$,
\item  $p$ over the minimal coset representatives $w$ in $S_{\blah}\backslash S_{\bd}/S_{\bmuh}$ and
\item $h$ over a basis for $\bLa(\bmuh')$,
\end{itemize}
then the morphisms $\bmuh \overset{p;h}\Longrightarrow \blah$ form a basis of $A$.
\end{thm}

\begin{remark}
\label{rkchoice}
{\rm
If we choose different reduced decompositions for $q$ as in Lemma~\ref{defq}, the resulting elements \eqref{mula} will differ by a linear combination of maps corresponding to shorter double cosets (Proposition \ref{easyrel}) which also shows that non-reduced decompositions can be replaced by reduced ones by changing $h$.}
\end{remark}

For a permutation $q$ as in Lemma \ref{defq} with its fixed reduced expression $q=s_{i_1}\cdots s_{i_\ell}$ set $q^{(j)}=s_{i_j}\cdots
s_{i_1}$.  From this we can read off a string diagram of $q$, as in
Khovanov-Lauda \cite{KL1}.  We let $\bmuh_{2j}=q^{(j)}(\bmuh')$ and
$\bmuh_{2j+1}$ be the merge of $\bmuh_{2j}$ at index $i_{j+1}$.  Then,
of course, $\bmuh_{2j}$ is a split of $\bmuh_{2j}$ at index $i_j$.  In
particular, $\bmuh_{2k}=\blah'$.

For each $w\in S_{\bmuh}\backslash S_{\bd}/S_{\blah}$, let $p(w)$ be
the unique permutation in $S_d$ of minimal length which acts as $w$ does on each
individual color, where $[1,d]$ is colored according to $\bmuh$ for
the domain and
$\blah$ for the image.  Put another way, if we think of $S_{\bd}$ as
acting on colored alphabets $1^1,\dots,d_1^1,1^2,\cdots d_2^2,\dots,1^e,\dots,d_e^e,$
and let $p(\bmuh)$ be the unique permutation of minimal length sending this sequence to one colored according to the
residue sequence of $\bmuh$, then $p(w)=p(\blah)\cdot w\cdot
p(\bmuh)^{-1}$ (where $p(\blah)$ is the unique permutation sending a
totally ordered sequence of 1's, 2's etc. to $\res\blah$).

\begin{proof}[Proof of Theorem \ref{basis}]
Our method of proof is to show that this remains a basis when we take
the associated graded of our algebra with respect to a geometrically defined filtration.

Forgetting the representation (and only keeping the flags) defines a
canonical map
\begin{eqnarray}
\label{Phi}
\Phi:&&\cH(\bmuh,\blah)\to \cF(\bmuh) \times \cF(\blah)
\end{eqnarray}
which is $G_{\bd}$-equivariant. Every fiber is the space of quiver
representations on a vector space preserving a particular pair of
flags, which is naturally an affine space (one can assume both flags
are spanned by vectors in a fixed basis, and thus the condition of
preserving the pair of flags is simply requiring certain matrix
coefficients to vanish).  Thus, over each $G_\bd$-orbit, the map is an
affine bundle, but with a different fiber over each orbit.
The $G_\bd$-orbits on $\cF(\bmuh) \times \cF(\blah)$ are in bijection
with the double coset space $S_{\bmuh}\backslash S_{\bd}/S_{\blah}$, namely
if we choose completions of the partial flags to complete flags, their
relative position is an element of $S_\bd$, and its double coset is
independent of the completion chosen.
 That is, using the Bruhat decomposition with fixed lifts of elements from $S_{\bd}$ to $G_{\bd}$ it sends an element $(xP_{\op{t}(\bmuh)}, yP_{\op{t}(\blah)})$, $x,y\in S_{\bd}$ to $x^{-1}y$, and a shortest double coset representative $w$ defines the orbit $\cF_w$ of all elements of the form $(gP_{\op{t}(\bmuh)}, gwP_{\op{t}(\blah)})$, $g\in G_{\bd}$. A pair of flags contained in $\cF_w$ is said to have relative position $w$.

The orbit $\cF_w$ is an affine bundle over $G_{\bd}/(P_{\bmuh}\cap wP_{\blah}w^{-1})$ via the bundle map \[(gP_{\op{t}(\bmuh)}, gwP_{\op{t}(\blah)})\mapsto g(P_{\bmuh}\cap wP_{\blah}w^{-1}),\] and so both the orbit and its preimage in $G_\bd$ have (equivariant) homology concentrated in even degrees.

Using the surjection $\Phi$ we can partition $\cH(\bmuh,\blah)$ into the preimages of the orbits and filter the space by letting $\cH_{k}=\bigcup_{l(w)\leq k}\Phi^{-1}(\cF_w)$ be the union of orbits with length of the shortest representative at most $k$.   The space $\cH_{k}\setminus \cH_{k-1}=\bigcup_{l(w)=k}\Phi^{-1}(\cF_w)$ is a disjoint union of affine bundles over partial flag varieties, and in particular has even homology.  The usual spectral sequence for a filtered topological space shows that $H_*^{BM}(\cH(\bmuh,\blah))$ has a filtration whose associated graded is the sum of the homologies of these orbits (the spectral sequence degenerates immediately for degree reasons because of the concentrations in even degrees).

Thus, it suffices to show that  the morphisms $\bmuh
\overset{w;h}\Longrightarrow \blah$ pass to a basis of the associated
graded.  That is, if we fix $w$, that the classes $\bmuh
\overset{w;h}\Longrightarrow \blah$ pullback to a basis of
$H^{BM}_*(\Phi^{-1}(\cF_w)/G)$.

It is enough to show that $\bmuh \overset{w;1}\Longrightarrow \blah$ goes to the fundamental class of $\Phi^{-1}(\cF_w)$, since adding the $h$ simply has the effect of multiplying by classes that range over a basis of $H^*(\Phi^{-1}(\cF_w)/G)$.
Consider the iterated fiber product
\begin{eqnarray*}
 \cH(\bmuh,w,\blah)&=&
 \cQ(\bmuh_0,\bmuh)\times_{\cQ(\bmuh_0)}\cQ(\bmuh_0,i_1)\times_{\cQ(\bmuh_1)}\cQ(\bmuh_2,i_1)
 \times_{\cQ(\bmuh_2)}\cdots\\
&& \times_{\cQ(\bmuh_{2k-2})} \cQ(\bmuh_{2k-2},i_k)\times_{\cQ(\bmuh_{2k-1})}\cQ(\bmuh_{2k},i_k)\times_{\cQ(\bmuh_{2k})}\cQ(\bmuh_{2k},\blah),
\end{eqnarray*}
where $\cQ(\bmuh_0,\bmuh)$ denotes the subset of $\cQ(\bmuh_0)$ such
that the associated graded remains semi-simple after coarsening the
flag to one of type $\bmuh$.  Recall that by the definition of
$\bmuh_{2k+1}$ we have that
\[\cQ(\bmuh_k,\bmuh_{k+1})=\cQ(\bmuh_k,i_{k+1})\qquad \cQ(\bmuh_k,\bmuh_{k-1})=\cQ(\bmuh_k,i_{k}).\]
We can think of the above fiber product as the space of
representations with compatible flags $F,F_0,\dots,F_{2k},F'$ of
dimension vectors $\bmuh,\bmuh_0,\dots,\bmuh_{2k},\blah$
such that if any two consecutive flags have subspaces of the same
size, those spaces must coincide.   We have a map
\begin{eqnarray*}
b^w:&&\cH(\bmuh,w,\blah)\to \cH(\bmuh,\blah)
\end{eqnarray*}
which remembers the flags $F$ and $F'$ and  $\bmuh
\overset{w;1}\Longrightarrow
\blah=b_*[\cH(\bmuh,w,\blah)]$ is the push-forward
of the fundamental class of $\cH(\bmuh,w,\blah)$.

More generally, we
can identify $\bLa(\bmuh_{2k})$ with the classes in
$H^*(\cH(\bmuh,w,\blah))$ generated by the Chern classes of the
tautological bundles corresponding to the successive quotients of the
flag $F_{2k}$.  Then $\bmuh
\overset{w;h}\Longrightarrow
\blah=b_*(h\cap [\cH(\bmuh,w,\blah)])$ is the pushforward of the cap
product of $h$ with the fundamental class (the Poincar\'e dual of the
class $h$).

To establish the linearly independence the following Lemma will be used.
\begin{lemma}
  The image of the map $b^w$ lies in $\cH_{\ell(w)}$ and induces an
  isomorphism between the locus in $\cH(\bmuh,w,\blah)$ where $F$ and $F'$ have relative
  position $w$ and $\Phi^{-1}(\cF_w)$.
\end{lemma}
\begin{proof}
  Let $p^{(j)}$ be the element of $S_d$ induced by acting on the
  blocks of $\res\bmuh'$ with $q^{(j)}$.  By the usual Bruhat
  decomposition, we must have that $F_{2j}$, considered as a flag on
  $\oplus_{i=1}^eV_i$, has
  relative position $\leq p^{(j)}$, i.e. smaller or equal $p^{(j)}$ in Bruhat order, with respect to $F$ and $\leq p^{(j)}p(w)^{-1}$
  with respect to $F'$.  Since $p^{(j)}\leq p(w)$ in right Bruhat order (its
  inversions are a subset of those of $p(w)$), this condition uniquely
  specifies $F_{2j}$ (which also uniquely specifies $F_{2j-1}$, since
  this is a coarsening of $F_{2j}$).

Thus, we only need to show that this flag is compatible with the
decomposition $\oplus_{i=1}^eV_i$  and strictly preserved by the
accompanying quiver representation $f$.  The former is clear, since we
can choose a basis compatible with this decomposition such that both
$F$ and $F'$ contain only coordinate subspaces; thus, the spaces of
$F_{j}$ are coordinate for this basis as well.  It follows that they
are also compatible
with the decomposition.  The latter is a well-known property of the Bruhat
decomposition:
if a nilpotent preserves two different flags $F$ and $F'$, then it
preserves any other flag $F''$ that fits into a chain where the
relative position of $F$ and $F''$ is $v$, that of $F''$ to $F'$ is
$w$, and $\ell(vw)=\ell(v)+\ell(w)$. We simply apply this to $f$
thought of as a nilpotent endomorphism of $\oplus_{i=1}^eV_i$.
\end{proof}

Thus, we have a Cartesian diagram
\[
\tikz{ \node (aa) at (0,0) {$\Phi^{-1}(\cF_w)$}; \node (ba) at (4,0)
{$\cH(\bmuh,\blah)$}; \node (ab) at (0,-2) {$b^{-1}(\Phi^{-1}(\cF_w))$}; \node (bb) at (4,-2)
{$\cH(\bmuh,w,\blah)$}; \draw[right hook->,thick] (aa) to (ba);
\draw[->,thick] (bb)--(ba) node[right,midway]{$b$}; \draw[->,thick]
(ab)--(aa) node[left,midway]{$\wr$}; \draw[right hook->,thick] (ab) to (bb); }
\]
If we push-forward and pull-back $h\cap [\cH(\bmuh,w,\blah)]$ from the bottom right to the upper left,
we obtain the class of $\bmuh
\overset{w;h}\Longrightarrow
\blah$ in the associated graded with respect to the geometric
filtration, thought of as a Borel-Moore class on $\Phi^{-1}(\cF_w)$.
We can also calculate this class by pull-back and pushforward; this
goes to $h\cap [\Phi^{-1}(\cF_w)]$, where we consider $h$ now as a
class on $\Phi^{-1}(\cF_w)$ by pull-back.  The tautological
bundles for $F_{2k}$ pulled back to $\Phi^{-1}(\cF_w)$ induce an
isomorphism $\bLa(\bmuh_{2k})\cong H^*(\Phi^{-1}(\cF_w)/G)$, and cap
product with the fundamental class induces an isomorphism
$H^*(\Phi^{-1}(\cF_w)/G)\cong H^{BM}_*(\Phi^{-1}(\cF_w)/G)$.  Thus,
the classes $\bmuh
\overset{w;h}\Longrightarrow
\blah$ pull-back to a basis of $ H^{BM}_*(\Phi^{-1}(\cF_w)/G)$. Ranging
 over all $w$, we obtain a basis of 
$H^{BM}_*(\cH(\bmuh,\blah))$, since it is a basis of the associated graded.
\excise{
Let us give a second proof based on the localization formalism.  If we
expand the class of $b_*[\cH(\bmuh,w,\blah)]$ in terms of
$\psi_{v,v'}$, we see that
\begin{itemize}
\item if $(v')^{-1}v\nleq w$ in Bruhat order, then the coefficient is 0.
\item if $(v')^{-1}v= w$ then this term can only come from the chain of
  coefficients $\cdot \psi_{w_{(1)}}$
\end{itemize}
}
\end{proof}

\begin{remark}{\rm It is worth noting that Theorem \ref{basis} could also be proved diagrammatically. The easy part is the linearly independence which follows analogously to the arguments in \cite{KL1} as follows:
Let $\bmuh,\blah$ be fixed. By construction, $\bmuh\overset{w;h}\Longrightarrow\blah$ is zero if $w\not=w_-$.
Let $S$ be the split from  $\blah$ into unit vectors and $M$ the merge from unit vectors to $\bmuh$. Then a linear combination $\sum \alpha_{w,h}(\bmuh\overset{w;h}\Longrightarrow\blah)$ is zero if and only if $\sum \alpha_{w,h}S\circ(\bmuh\overset{w;h}\Longrightarrow\blah)\circ M=0$ (because $M$ is surjective and composing with $S$ is injective). The linear independence is therefore given by \cite[Theorem 2.5]{KL1}.
}
\end{remark}
Theorem \ref{basis} provides a basis which appears natural from the diagrammatical as well as geometrical point of view. Proposition \ref{easyrel} provides some fairly obvious relations, but unfortunately we do not have a complete set of relations defining our algebras. Although we have reduced the question to one of linear algebra, it appears to be quite
difficult linear algebra and we are left with the following problem:

\begin{problem}
  Find an explicit presentation with generators and defining relations
  of the algebras $A_\bd$ (or a  Morita equivalent algebra).
\end{problem}
%Note that this problem is solved nicely in case $e=\infty$ for $\ell=1,2$, see \cite[Appendix A]{HM}, \cite{BS3}. A way to approach this problem in general might be to study our algebras in the more general context of weighted KLR-algebras by realizing $A_\bd$ as an idempotent truncation $e W^\vartheta e$ of a weighted KLR -algebra, see \cite[3.6]{wKLR}.

The following result at least describes its center (which is a Morita invariant):
\begin{lemma}
\label{center}
%\begin{center}
The ring of total invariants $R({\bd})^{S_{d}}$ is isomorphic to the center of ${\bf A}_{\bd}$.
\end{lemma}
\begin{proof}
  The proof is essentially identical to \cite[2.9]{KL1}.
  Restricting the action maps simultaneously to the total invariants, defines a map $\alpha$ from  $R({\bd})^{S_{d}}$ into the center of ${\bf A}_{\bf d}$. For each
  $\bmuh\in\Compe(\bd)$, we can obtain a new vector composition $\bmuh'\in\Compe(\bd)$ by splitting
  each dimension vector between each colors. By Proposition~\ref{euler-int}, composing with this defines an inclusion from $Ae_\bmuh$ to $Ae_{\bmuh'}$. Now, we obtain a third vector composition $\bmuh''\in\Compe(\bd)$ by reordering all the colors in the residue sequence, the smallest to the left, by permuting the blocks by applying crossings. In this case the Demazure operator from Proposition \ref{euler-int} is the identity and we get an inclusion from $Ae_{\bmuh}'$ to $Ae_{\bmuh''}$.
  It follows that  $Ae_{\bmuh}$ is a submodule (not necessarily direct summand) of  $Ae_{\bmuh''}$. In particular, every central element acts non-trivially on some $P(\bmuh'')=Ae_{\bmuh''}$. Furthermore, every $P(\bmuh'')$
is just obtained by inducing the direct sum of modules of the form $P((d_1,0,\dots)),
P((0,d_2,0,\dots)),\dots),\ldots$; we thus have an injective homomorphism $\beta\colon Z(A_\bd)\to
e_{\mu_\bd}Ae_{\mu_\bd}\cong R({\bd})^{S_{d}}$ given by
$\psi(z)=e_{\mu_\bd}$. Furthermore, the composition $\beta\circ
\alpha$ is the identity on $ R({\bd})^{S_{d}}$.  Thus, $\beta$ is
surjective as well, and gives the desired isomorphism.
\end{proof}

\section{Higher level generalizations}
\label{sec:high-levels}

In this section we extend our convolution algebras to a ``higher level version.''
Motivated by the construction of quiver varieties of Nakajima
\cite{Nak94}, \cite{Nak98}, \cite{Nak01} we consider not only quiver
representations of $\Gamma$, but enrich them with extra data. For a further development of our approach, see also \cite{wKLR}.

We equip the quiver $\Gamma$ with shadow vertices, one for each $i\in\mathbb{V}$
together with an arrow pointing from the $i$th vertex to the $i$th shadow vertex, see Figure~\ref{fig:circle} were the shadow vertices are drawn in red/grey. For given $\nu:\mathbb{V}\to \mZ_{\geq 0}$,
we extend the affine space \eqref{Rep} of representations of a fixed dimension vector $\bd$ to the affine space
\[\Rep_{\bd;\nu}:={\Rep}_\bd\times
\bigoplus_i\Hom_K\left(V_i,K^{\nu(i)}\right),\qquad \Rep_\nu=\bigsqcup_\bd
\Rep_{\bd;\nu}\] of representations $(V,f,\gamma)$ shadowed by vector spaces
$K^{\nu(i)}$. It comes endowed with the product $G_\bd$-action.
\begin{definition} Assume we are given a {\de weight data} consisting out of
\begin{itemize}
\item an $\ell$-tuple $\bnu=(\nu_1,\dots,\nu_\ell)$ of maps
$\mathbb V\to \mZ_{\geq 0}$, called an $\ell$-{\de weight};
\item an $\ell+1$-tuple \(
\grave{\mu}=(\bmuh(0),\bmuh(1),\dots,\bmuh(\ell))\) of vector compositions of no fixed length, but all of type $e$.
\end{itemize}
Then we call $\bmuh=\bmuh(0)\cup \cdots \cup \bmuh(\ell)\in\Compe(\bf d)$ the {\bf associated vector composition} and denote by
$\fQ(\grave\mu)$ the subspace of $\cQ(\bmuh)\times
\bigoplus_i\Hom_K(V_i,K^{\nu(i)})$ defined as
\[\fQ(\grave\mu)=\Big\{\left((V,f,F),\{\gamma_i\}\right) \mid \gamma_i(\grave W_i(k))\subset
K^{\nu_1(i)+\cdots +\nu_k(i)}\Big\},\] where as usual $K^a\subset K^b$ for $a\leq b$ is the subspace spanned by the first
$a$ unit vectors, and $\grave{W}_i(1)\subset \grave{W}_i(2)\subset\cdots\subset \grave{W}_i({\ell+1})=V_i$ is the partial flag at the vertex $i$ coarsening $F_i$ and obtained by picking out the largest subspace corresponding to each part of $\grave\mu$.
\end{definition}

Just as an unadorned representation carries a canonical socle
filtration (i.e. a filtration starting with the maximal semi-simple submodule and proceeding such that the successive quotients are maximal semi-simple), an extended representation $(V,f,\gamma)\in\Rep_{\bd;\nu}$ carries a slightly modified
filtration which starts with $\{0\}\subset R_1=(W,f,\gamma)$ such that $(W,f)$ is the largest
subrepresentation $(V,f)$ for which $\gamma(W_i)\subset
K^{\nu_1(i)+\cdots +\nu_k(i)}$ and proceeds inductively by considering
the corresponding first step in $(V/W,\overline{f},\overline{\gamma})$
pulled back to $(V,f)$. The dimension vectors of these subquotients
define a weight data $(\bnu,\grave{\mu})$ where $\grave{\mu}$ denotes
the type of the multi-flag $R_i/R_{i-1}$ induced by the filtration.
We call this the {\bf extended socle filtration}.

\subsection{Extended convolution algebras}
Generalizing \eqref{p}, we have a projection map
$$p\colon \fQ(\grave\mu)\to\Rep_{\bd;\nu}/G_\bd$$
where $\nu:=\nu_1+\cdots+\nu_\ell$, and we can study the convolution
algebra \[\tilde A^{\bnu}:=\Ext^*_{D(\Rep_\nu\!/G)}\Big(\bigoplus_{\grave
\mu}p_*\K_{\fQ(\grave\mu)},\bigoplus_{\grave \mu}p_*\K_{\fQ(\grave\mu)}\Big)\cong \bigoplus_{\grave\mu,\grave\nu}
H^{BM}_*(\FH(\grave\mu,\grave\nu)),\] where $\FH(\grave\mu,\grave\nu)\cong \fQ(\grave\mu)\times_{\Rep_{\bd;\nu}}
\fQ(\grave\nu)/G$ and the sum runs over all weight data $(\bnu,\grave{\mu})$ with associated vector composition $\bmuh\in\Compe(\bf d).$

No individual algebra (with vertical multiplication given by convolution) in this family has a notion of horizontal multiplication.  Instead, there is a horizontal
multiplication $\tilde{A}^\bnu\times \tilde {A}^{\bnu'}\to \tilde {A}^{\bnu\cup \bnu'}$ which concatenates the tuples of
weights.  This can, of course, be organized in a single algebra $\tilde{{A}}$, but for our purposes it is more
profitable to think of the separate algebras
$\tilde{{A}}^\bnu$. The following is a direct consequence of our definitions.
\begin{prop}
For each fixed dimension vector ${\bf d}$, horizontal composition induces a right ${\bf A}$-module
  structure on $\mathbf{A}^\bnu$.
\end{prop}

Using horizontal and vertical composition, the algebra
$\tilde{A}^\bnu$ is generated by a small number of elements (like ${\bf A}$
is) which are defined in some sense locally and also have an easy diagrammatic description. We mimic this construction now in the extended case incorporating the shadow vertices. Of
course, we still have the old merges and splits not involving the shadow vertices, but also have a new ``move'' on $\ell$-tuples of compositions.
\begin{definition}
  We call $\gla$ a {\bf left shift} of $\gmu$ (and $\gmu$ a {\bf right
    shift} of $\gla$) by $\bc$ if for some index $m$, we have that
  \[\gla(m)=\gmu(m)\cup \bc\text{ and }\gmu(m+1)=\bc\cup \gla(m+1).\] In
  words, if a vector $\bc$ has been shifted from the start of the $m+1$-st
  composition of $\gmu$ to the end of the $m$th composition for
  $\gla$.
\end{definition}

If $\gla$ is a left shift of $\gmu$,
then $\fQ(\gla)$ is naturally a subspace of $\fQ(\gmu)$.  Thus, generalizing the construction after Definition \ref{merges}, we can
think of it as a correspondence (read from right to left)
\[\tikz{
\node (a) at (-3,0) {$\fQ(\gla)$};
\node (b) at (0,1) {$\fQ(\gla)$};
\node (c) at (3,0) {$\fQ(\gmu)$};
\draw[thick,->] (b) -- (a);
\draw[thick,->] (b) -- (c);
}
\]
Similarly, for right shifts we can reverse this correspondence and read from left to right.  Thus,
the fundamental class of this correspondence gives elements of
${A}^\bnu$ corresponding to left and right shifts, which we denote by
$\gmu\to\gla$ and $\gla\to\gmu$.

  \begin{prop}\label{shift-deg}
   The degree of the map $\gla\to\gmu$ or $\gmu\leftarrow\gla$ associated to right or left shift by $\bc$ at the index $m$ equals $\sum_i c_i\nu_m(i)$.
  \end{prop}
\begin{proof}
 This follows from a direct calculation.
\end{proof}

Theorem \ref{basis} generalizes immediately to the algebra
$\tilde{A}^\bnu$.  We wish to define elements $\grave\mu
\overset{w;h}\Longrightarrow \grave\nu$ in this case, generalizing
those for $A$.  Consider a minimal length coset representative $w$ in
  $S_{\grave\mu}\backslash S_{\bd}/S_{\grave\nu}$. We let $\grave\mu
  \overset{w;h}\Longrightarrow \grave\la$ be an arbitrarily chosen
  composition of merges, splits and left and right shifts such that
  if we ignore the left and right shifts the resulting element represents $w$, cf. Remark ~\ref{rkchoice}. If we let $S_{\grave\mu}$ be the Young subgroup associated to the concatenation of the parts of $\grave\mu$, then the following holds.
\begin{prop}\label{red-basis}
The morphisms $\grave\mu \overset{w;h}\Longrightarrow \grave\la$ form a
basis of $\tilde{A}^\bnu$ where we let range
\begin{itemize}
\item $(\grave\mu,\grave\nu)$  over ordered pairs of $(\ell+1)$-tuples of
vector compositions $\grave\mu$, $\grave\la$ of type $e$,
\item  $w$  over minimal coset representatives in
  $S_{\grave\mu}\backslash S_{\bd}/S_{\grave\la}$, and \item $h$  over
a basis for $\bLa(\bmuh_j)$. \end{itemize}
\end{prop}
\begin{proof} Analogous to the proof of Theorem
  \ref{basis}; the important points are to note that
  \begin{itemize}
  \item the diagrams $\grave\mu \overset{w;h}\Longrightarrow
    \grave\la$ give Borel-Moore classes supported on pairs of flags
    with relative position $\leq w$, and
\item modulo classes supported on pairs with
    relative position $<w$, the elements $\grave\mu \overset{w;h}\Longrightarrow
    \grave\la$ give a basis of the Borel-Moore classes
    of the subset with relative position $w$ as we let $h$ range over a basis of the appropriate
    polynomial ring.
  \end{itemize}
This immediately shows that these elements are a basis.
\end{proof}

As in Section \ref{sec:calc-basic-conv}, we can also calculate how
this morphism acts on the polynomial ring $\tilde V^\bnu$ from \eqref{Vgen}.  Using horizontal composition,
it suffices to consider the case where $\bnu=(\la)$ (so $\ell=1$) and look at pairs
$(\emptyset,\bd)$ and $(\bd,\emptyset)$ of vector compositions of type $e$.  Both $\FH( (\emptyset,\bd), (\bd,\emptyset))$ and $\FH(
(\bd,\emptyset),(\emptyset,\bd))$ are simply $\cQ(\bd)$, and thus their Borel-Moore homology is a rank 1 free module over
$H^*(BG_\bd)\cong \bLa(\bd)$ generated by the corresponding fundamental class. Their actions on $\tilde{V}^\bnu$ are simply the actions
of $\iota^*$ and $\iota_*$  where $\iota\colon \fQ(\bd,\emptyset)\to \fQ(\emptyset,\bd)$ is the obvious inclusion map.

\begin{prop} \label{red-euler}
The following diagram commutes,
\[\tikz{
\node (aa) at (0,0) {$H^{BM}_*(\fQ(\bd,\emptyset))$};
\node (ba) at (6,0){$H^{BM}_*(\fQ(\emptyset,\bd))$};
\node (ab) at (0,-3) {$\bLa(\bmuh(0)\cup\cdots\cup \bmuh(\ell))$}; \node (bb) at (6,-3) {$\bLa(\bmuh(0)\cup\cdots\cup \bmuh(\ell))$};
\draw[->,thick] (aa) to[out=10,in=170] node[above,midway]{$\iota_*$} (ba);
\draw[->,thick] (ba) to[out=-170,in=-10] node[below,midway]{$\iota^*$} (aa);
\draw[->,thick] (ba) -- (bb) node[right,midway]{$\op{Borel}$};
\draw[->,thick] (aa) -- (ab) node[left,midway]{$\op{Borel}$};
\draw[->,thick] (ab) to[out=10,in=170] node[midway,above] {$\prod_{k=1}^et_k^{\la(k)}$} (bb);
\draw[->,thick] (bb) to[out=-170,in=-10] node[below,midway]{$1$} (ab);
}\]
where $t_k=x_{k,1}\cdots x_{k,d_k}$ is the highest degree elementary symmetric function in
the alphabet \eqref{alphabeth} for the $k$th vertex in $\Gamma$.
\end{prop}

\begin{proof} The map $\iota$ is the inclusion of the $0$-section of a
$G_\bd$-equivariant vector bundle, whose underlying vector bundle is trivial, but equals
$\oplus_{k=1}^e\Hom(\cV_k,K^{\la(k)})$ equivariantly.  As the second vector space is trivial, the Euler class of this bundle is
$\prod_{k=1}^e c_{top}(\cV_k^*)^{\la(k)}$.  Since the Borel isomorphism sends the Chern class of maximal possible degree $c_{top}(\cV_k^*)$ to $t_k$, the result
follows.
\end{proof}

To make the algebra more concrete we first generalize Proposition \ref{faithful}:
\begin{prop}\label{hl-faithful}
The algebra $\tilde{A}^\bnu$ has a natural faithful action on
  \begin{equation}
    \label{Vgen}
    \tilde{V}^\bnu=
    \bigoplus_{\grave\mu}H^*(\fQ(\grave\mu))\cong \bLa(\bmuh(0)\cup\cdots\cup \bmuh(\ell)).
  \end{equation}
\end{prop}
\begin{proof}
  Let $\tilde {\fQ}$ and $\tilde \FH$ denote the varieties
  corresponding to ${\fQ}$ and ${\FH}$ respectively where we do not take the
  quotient by $G_\bd$.  Consider the inclusion maps of $T_\bd$-fixed
  points
\[ \iota_{\gmu}\colon \tilde {\fQ}(\gmu)^{T_\bd}/T_\bd \to
  \tilde {\fQ}(\gmu)/T_\bd,\qquad \iota_{\gmu;\grave\nu}\colon
  \FH(\grave\mu, \grave\nu)^{T_\bd\times T_\bd}/T_\bd\times T_\bd\to
  \FH(\grave\mu, \grave\nu)/T_\bd\times T_\bd,\]
and the natural map of
  quotients
\[
\beta_{\gmu}\colon\tilde {\fQ}(\gmu)/T_\bd\to \tilde
  {\fQ}(\gmu)/G_\bd =\fQ(\gmu),\;\;
    \beta_{\gmu;\grave\nu}\colon\tilde {\FH}(\gmu,\grave\nu)/T_\bd\times
  T_\bd\to \tilde {\FH}(\gmu,\grave\nu)/G_\bd \times
  G_\bd=\FH(\gmu,\grave\nu).\]
  We also have the proper projection maps
  \[p_1:\tilde {\FH}(\gmu,\grave\nu)/T_\bd\times
  T_\bd\rightarrow \tilde {\fQ}(\gmu)/T_\bd, \quad p_2: \tilde {\FH}(\gmu,\grave\nu)/T_\bd\times
  T_\bd\rightarrow \tilde {\fQ}(\grave\nu)/T_\bd\]
  and similarly \[\overline{p}_1:\tilde {\FH}(\gmu,\grave\nu)/G_\bd\times
  G_\bd\rightarrow \tilde {\fQ}(\gmu)/G_\bd,\quad\overline{p}_2: \tilde {\FH}(\gmu,\grave\nu)/G_\bd\times
  G_\bd\rightarrow \tilde {\fQ}(\grave\nu)/G_\bd,\]
and also the maps $p_1^T$, $p_2^T$ for the $T$-fixed points. We abbreviate $U=H^*(BG_\bd)$ and $V=H^*(BT_\bd)$.  Since $G_\bd$ is a
product of general linear groups, the pull-back by $\beta:BT_\bd \to
BG_\bd$ is an isomorphism between $U$ and invariants of the Weyl group
$S_\bd$ in
$V$.  In particular, $V$ is a free $U$-algebra of finite rank. (Note
that the maps  $\beta$ are smooth, hence pullbacks exist.)

Now, consider the Leray-Serre
  spectral sequence applied to the map $\beta_{\gmu}$. It degenerates at
  the $E^2$-page for parity reasons, so as a $V$-module, we have that $H^*( \tilde  {\fQ}(\gmu)/T_\bd)$ is isomorphic to $V\otimes_U H^*( \tilde{\fQ}(\gmu)/G_\bd)=V\otimes_U H^*( {\fQ}(\gmu))$; similarly for $\FH$.

  We claim there is a (up to signs) commutative  diagram
 % \begin{figure}
\[\tikz[->,thick]{\matrix[row sep=3mm,column sep=10mm,ampersand
  replacement=\&]{
 \& \node (b) {$
  \Hom_{U}(H^*(\fQ(\grave\mu) ),H^*(\fQ(\grave\nu) ))$};\\
\node(a) {$H^{BM}_*(\FH(\grave\mu, \grave\nu))$}; \&\\
 \& \node (g) {$
  \Hom_{V}(V\otimes_{U} H^*(\fQ(\grave\mu) ),V\otimes_{U} H^*(\fQ(\grave\nu) ))$};\\
\node(c) {$H^{BM}_*(\tilde \FH(\grave\mu, \grave\nu)/T_\bd\times T_\bd)$}; \&\\ \& \node(d){$
  \Hom_{V}(H^*(\tilde \fQ(\grave\mu) /T_\bd
  ),H^*(\tilde \fQ(\grave\nu) /T_\bd))$};\\
\node(e) {$H^{BM}_*(\tilde \FH(\grave\mu, \grave\nu)^{T_\bd\times T_\bd}/T_\bd\times T_\bd)$}; \&\\ \& \node(f){$
  \Hom_{V}(H^*(\tilde \fQ(\grave\mu) ^{T_\bd} /T_\bd
  ),H^*(\tilde \fQ(\grave\nu) ^{T_\bd} /T_\bd)) $};\\
};
\draw (a)-- node [above,midway]{$\star-$} (b);
\draw (c)--  node [above,midway]{$\star-$} (d);
\draw (e)--  node [above,midway]{$\star-$} (f);
\draw (a)-- node [above,left]{$\beta_{\gmu;\grave\nu}^*$} (c);
\draw (c)--node [left,midway]{$\iota^*_{\gmu;\grave\nu} (\iota_{\gmu;\grave\nu})_*\iota^*_{\gmu;\grave\nu}$} (e);
\draw (b)-- node [right,midway]{$\operatorname{id}_V\otimes -$} (g);
\draw (g)--  node [above,midway,rotate=-90]{$\sim$}  (d);
\draw (d)-- node [right,midway]{$\iota^*_{\grave\nu}\circ -\circ (\iota_\gmu)_*$}(f);
}\]
%\label{fig:huge}
%\caption{}
%\end{figure}

where $\star$ denotes the convolution in equivariant Borel-Moore homology. (Note that the $\iota^*$'s exist by \cite[2.6.21]{CG97}, since the variety  $\iota_{\gmu}$ can be embedded into $\cQ(\bmuh)\times \bigoplus_i\Hom_K(V_i,K^{\nu(i)})$ with smooth fixed point set, similarly for the other $\iota$'s.)
We first show that the top square is commutative. For $a\in H^*(\fQ(\grave\mu)/T_\bd)$
and $b\in H^{BM}_*(\FH(\grave\mu, \grave\nu))$ we can find $v\in V$ and $x\in H^*(\fQ(\grave\mu))$ such that
\begin{multline*}
  a\star \beta^*_{\gmu;\grave\nu}(b)= {p_2}_*
  (p_1^*(a)\cap\beta^*_{\gmu;\grave\nu} (b))\\ ={p_2}_*
  (p_1^*(v\cdot\beta^*_\gmu(x))\cap\beta^*_{\gmu;\grave\nu} (b))
  ={p_2}_* (v\cdot p_1^*(\beta^*_\gmu(x))\cap\beta^*_{\gmu;\grave\nu}
  (b)),
\end{multline*}
where we used that ${p_1}_*$ is $V$-equivariant. Hence
\begin{align*}
  a\star \beta^*_{\gmu;\grave\nu}(b)&={p_2}_* (v\cdot (\beta_\gmu\circ
  p_1)^*(x)\cap\beta^*_{\gmu;\grave\nu} (b))= {p_2}_* (v\cdot
  (p_1\circ\beta_{\gmu,\grave\nu})^*(x)\cap\beta^*_{\gmu;\grave\nu}
  (b))\\
 & = {p_2}_* (v\cdot
  \beta_{\gmu,\grave\nu}^*(\overline{p}_1^*(x))\cap\beta^*_{\gmu;\grave\nu}
  (b)) = v\cdot{p_2}_*
  (\beta_{\gmu,\grave\nu}^*(\overline{p}_1^*(x)\cap b)\\ &=
  v\beta_{\grave\nu}^*{\overline{p}_2}_*(\overline{p}_1^*(x)\cap b) =
  v\beta_{\grave\nu}^*(x\star b),
\end{align*}
where we used $\beta_\mu\circ p_1=\overline{p}_1\circ\beta^*_{\gmu;\grave\nu}$, base change and the fact that ${p_2}_*$ is $V$-equivariant.
Altogether
\[a\star \beta^*_{\gmu;\grave\nu}(b)=\beta_{\grave\nu}^*(v\cdot x\star b)\]
and hence the top square is commutative. To see that the bottom square is commutative up to sign, consider $a\in  H^*(\fQ(\grave\mu)/T_\bd)$ and $b\in H^{BM}_*(\tilde \FH(\grave\mu, \grave\nu)/T_\bd\times T_\bd)$. Then
\begin{align*}
  \iota^*_{\grave\nu}((\iota_{\gmu})_*a\star b)
  &= \iota^*_{\grave\nu} ({p_2}_*(p_1^*({\iota_{\gmu}}_*(a))\cap b)) &
  \text{(def. of conv.)}\\
  &={p_2^T}_*(\iota^*_{\gmu;\grave\nu} (p_1^*({\iota _{\gmu}}_*(a))\cap b)) &
  \text{(base change)}\\
  &={p_2^T}_* (\iota^*_{\gmu;\grave\nu}(\iota_{\gmu;\grave\nu})_*
  ({p_1^T}^*(a)) \cap b)  & \text{(base change)}\\
  &=\pm {p_2}_*^T  (\iota^*_{\gmu;\grave\nu} (\iota_{\gmu;\grave\nu})_*{p_1^T}^*(a))\cap \iota_{\gmu;\grave\nu}^*(b))  & \text{(projection)}\\
  &=\pm {p_2}_*^T (q\cdot( {p_1^T}^*(a))\cap
  \iota_{\gmu;\grave\nu}^*(b))  & (\iota^*_{\gmu;\grave\nu}(\iota_{\gmu;\grave\nu})_*=q\cdot)\\
  &=\pm {p_2}_*^T ( {p_1^T}^*(a)\cap
  \iota_{\gmu;\grave\nu}^*{\iota_{\gmu;\grave\nu}}_*\iota_{\gmu;\grave\nu}^*(b))
  &(\iota^*_{\gmu;\grave\nu}(\iota_{\gmu;\grave\nu})_*=q\cdot)\\
  &=\pm a\star  \iota_{\gmu;\grave\nu}^*
  {\iota_{\gmu;\grave\nu}}_*\iota_{\gmu;\grave\nu}^*(b). &\text{(def. of conv.)}
\end{align*}
where in addition to the definition of convolution and the base change
formula, we used the
projection formula \cite[Remark 5.4]{FultonAnderson} and the property from \cite[Corollary~2.6.44]{CG97} that $\iota^*_{\gmu;\grave\nu}(\iota_{\gmu;\grave\nu})_*$ is multiplication with some Euler class $q$.

The map  $\beta_{\gmu,\grave\nu}^*$ is injective, since $\operatorname{id}_V\otimes_-$ and the map $M\to V\otimes_UM, m\mapsto 1\otimes m$ is injective for any free $U$ module.
Furthermore, the $\iota^*$'s and $\iota_*$'s are injective, because all
varieties that appear are equivariantly formal and the local Euler classes are  not zerodivisors, see \cite[Theorem 3 and Lemma 4]{Brionlectures}. The bottom horizontal arrow is injective, since the torus fixed points are isolated.

Thus any element in the kernel of the top horizontal map is in the kernel of the
map from the top left to the bottom right. We have  
seen already that going first vertically and then horizontally is injective, so the kernel is trivial.  This completes the proof.
\end{proof}

Just as $A$ contains an idempotent which projects down to the quiver Hecke algebra, the algebra $\tAb$ contains an
analogous idempotent $\eT$. Let
\begin{eqnarray}
\label{eT}
\eT&=&\sum e_{\grave\mu}
\end{eqnarray}
denote the idempotent in $\tAb$ defined as the sum over all primitive
idempotents $e_{\grave\mu}$ indexed by all vector multi-compositions
${\grave\mu}$ where each composition appearing is of the form $\al_i$,
that is, the corresponding flag is complete. We call $\eT\tAb\eT$ the {\bf geometric
tensor algebra}. The name and construction is based on a diagrammatically defined { tensor algebra} $\tilde{T}^\bnu$ defined in  \cite[Def. 4.5]{Webmerged} to categorify tensor products of representations of quantum groups. To connect the two algebras we first extend our graphical calculus from Section \ref{pictorial} to the algebras $\tilde{
  A}^{\bnu}$.

\subsection{Extended graphical calculus}
\label{KLRalgebra}
To extend our graphical calculus to the algebras $\tilde{
  A}^{\bnu}$ we have to encode the data attached to the shadow vertices. To do so we add a {\bf red band} labeled
with $\nu_i$ between the black lines representing $\bmuh(i-1)$ and
$\bmuh(i)$. For instance, for fixed $\bnu$, the idempotent
$e_{(\bnu,\grave{\mu})}$ corresponding to the weight data
$(\bnu,\grave{\mu})$ and another element of the algebra is graphically
described in Figure~\ref{idempotent} below.
\begin{figure}[h]
\begin{tikzpicture} [thick,scale=3]
\draw[wei] (0,.3) -- (0,0) node [below] {\small $\nu_1$};
\foreach \x in {1,2,...,3} {
\draw (\x/5,.3) -- (\x/5,0) node [below] {\small$\bmuh^{(\x)}$};}
\begin{scope}[xshift=0.8cm]
\draw[wei] (0,.3) -- (0,0) node [below] {\small $\nu_2$};
\foreach \x in {1,2,...,3} {
\draw (\x/5,.3) -- (\x/5,0) node [below] {\small$\bmuh^{(\x)}$};}
\end{scope}
\begin{scope}[xshift=2cm]
\draw[wei] (0,.3) -- (0,0) node [below] {\small $\omega_j$};
\foreach \x in {0,0.6,...,1.8} {
\draw (\x+.6,.3) -- (\x+.6,0) node[midway,circle,fill,inner sep=2pt] {}
node [below] {\tiny $\alpha_{j+2}$};
\draw (\x+.2,.3) -- (\x+.2,0) node [below] {\tiny $\alpha_j$};
\draw (\x+.4,.3) -- (\x+.4,0) node [below] {\tiny $\alpha_{j+1}$};
}
\end{scope}
\end{tikzpicture}
\caption{An idempotent and $t^{(j;9)}$ from Lemma \ref{id-split} for $e=3$.}
\label{idempotent}
\end{figure}
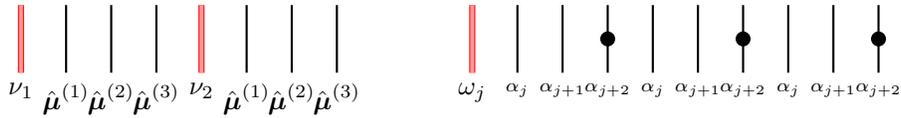

The splitting and merging morphisms defined earlier are displayed as in Section~\ref{pictorial}, except of additional red lines separating the multi-compositions. We have also diagrams associated to the new morphisms, that is to left and right shifts. They are denoted by a (left resp. right) crossing of red and black strands, as shown below:
  \begin{equation}
  \label{redblack}
    \begin{tikzpicture}[baseline=-0.3cm]
      \node at (-3,0)
[label=below:{a left shift}]
      {\begin{tikzpicture} [thick,scale=2.3] \draw[wei] (0,0) --
          (.3,.3); \draw (.3,0) -- (0,.3);
        \end{tikzpicture}}; \node at (3,0)
[label=below:{a right shift}] {\begin{tikzpicture} [thick,scale=2.3]
          \draw[wei] (.3,0) -- (0,.3); \draw
          (0,0) -- (.3,.3);
        \end{tikzpicture}};
    \end{tikzpicture}
  \end{equation}
For instance Figure~\ref{violating} (read as usual from bottom to top) shows a merge and split followed by a left shift. Proposition~\ref{shift-deg} assigns a degree to each crossing of a red band with a black strand, see Remark~\ref{degrees} for explicit formulas.

Following \cite{KL1} and \cite{Webmerged} we define now a diagram algebra, much like that defined above, but only allowing lines labeled by simple roots, and only allowing crossings, rather than splits and merges. Let $m$ be a fixed integer with its integral points $\cP=\left\{(x,y)\mid x\in \{1,2,\ldots, m\}, y\in\{1,2\}\right\}$. A {\bf strand diagram} is a collections of labeled decorated arcs (or strands) on the plane connecting points from $\cP$ with $y=0$ with those with $y=1$ satisfying the following properties. Arcs are assumed to have no critical points and each of them is colored black or red. Black arcs are labeled by elements from $\mV$ (or alternatively simple roots of $\mg=\hat{\mathfrak{sl}}_e$) and red arcs are labeled by weights of $\mg$. Moreover, black arcs can carry a finite number of dots. Arcs can intersect but triple intersections are not allowed. We consider strand diagrams up to isotopies that do not change the combinatorial type (labels and decorations) of the diagram and do not create critical points.

Locally around each point of an arc a strand  diagram is either a single line or a crossing of two black strands as in \eqref{cross} or a crossing of a black with a red line as in \eqref{redblack}. In particular, two red arcs never cross and no pair of two lines is allowed to end at the same point. We consider the vector space $\mathbb{D}$ spanned by isotopy classes of such diagrams for varying $m$. It is an algebra, called the {\bf free strand algebra}, by vertical concatenation of diagrams. for instance, the first element in Figure~\ref{idempotent} is an idempotent.

To define an interesting quotient of this free strand algebra we need to introduce some extra data, similar to the quiver Hecke algebras from \cite{KL1}, \cite{Rou2KM},  \cite{BKKL}.

For a weight $\la$ let $\la^i=\al_i^\vee(\la)$ be its Dynkin labels.
For $i\not=j\in\mV$ we set
\begin{eqnarray}
\label{Qij}
Q_{ij}(u,v)&=&\begin{cases} 1  & \text{if } i\neq j\pm 1\\
u-v & \text{if }i=j-1\quad (e\neq 2)\\
v-u & \text{if }i=j+1\quad (e\neq 2)\\
(u-v)(v-u) &\text{if }i=j-1=j+1 \quad(e=2).
\end{cases}
\end{eqnarray}

\begin{definition}
For an $\ell$-tuple of weights, $\bnu$,  the {\bf tensor algebra} $\tilde{T}^\bnu$ is the subquotient algebra of the free strand algebra spanned by all diagrams which have exactly $\ell$ red arcs labeled by $\nu_1,\ldots, \nu_\ell$ in this order  modulo the following relations (R1)-(R3) with all its possible mirror images:
\begin{enumerate}[(R1)]
\item The usual KLR relations from Figure~\ref{quiver-hecke}.\footnote{The polynomial in the box means that we have a linear combination of diagrams, one for each monomial appearing. The diagram attached to the monomial is the identity diagram as indicated equipped with $a_i$ dots on the $i$th strand, where $a_i$ is the exponent of $y_i$. To match it with \cite{BKKL} note that our $y$'s are the negatives of the $y$'s there. }
\item  All black crossings and dots pass through red lines, with possibly a correction term:
  \begin{equation*}
    \begin{tikzpicture}[thick]
      \draw (-3,0)  +(1,-1) -- +(-1,1) node[at start,below]{$i$};
      \draw (-3,0) +(-1,-1) -- +(1,1)node [at start,below]{$j$};
      \draw[wei] (-3,0)  +(0,-1) .. controls +(-1,0) ..  +(0,1);
      \node at (-1,0) {=};
      \draw (1,0)  +(1,-1) -- +(-1,1) node[at start,below]{$i$};
      \draw (1,0) +(-1,-1) -- +(1,1) node [at start,below]{$j$};
      \draw[wei] (1,0) +(0,-1) .. controls +(1,0) ..  +(0,1);
\node at (2.8,0) {$+ $};
      \draw (6.5,0)  +(1,-1) -- +(1,1) node[midway,circle,fill,inner sep=2.5pt,label=right:{$a$}]{} node[at start,below]{$i$};
      \draw (6.5,0) +(-1,-1) -- +(-1,1) node[midway,circle,fill,inner sep=2.5pt,label=left:{$b$}]{} node [at start,below]{$j$};
      \draw[wei] (6.5,0) +(0,-1) -- +(0,1);
\node at (3.8,-.2){$\displaystyle \sum_{a+b+1=\la^i} \delta_{i,j} $}  ;
 \end{tikzpicture}
  \end{equation*}
\begin{equation}\label{dumb}
    \begin{tikzpicture}[thick,baseline=2.85cm]
      \draw[wei] (-3,3)  +(1,-1) -- +(-1,1);
      \draw (-3,3)  +(0,-1) .. controls +(-1,0) ..  +(0,1);
      \draw (-3,3) +(-1,-1) -- +(1,1);
      \node at (-1,3) {=};
      \draw[wei] (1,3)  +(1,-1) -- +(-1,1);
  \draw (1,3)  +(0,-1) .. controls +(1,0) ..  +(0,1);
      \draw (1,3) +(-1,-1) -- +(1,1);    \end{tikzpicture}
  \end{equation}
\begin{equation*}
    \begin{tikzpicture}[thick]
  \draw(-3,6) +(-1,-1) -- +(1,1);
  \draw[wei](-3,6) +(1,-1) -- +(-1,1);
\fill (-3.5,5.5) circle (3pt);
\node at (-1,6) {=};
 \draw(1,6) +(-1,-1) -- +(1,1);
  \draw[wei](1,6) +(1,-1) -- +(-1,1);
\fill (1.5,6.5) circle (3pt);
    \end{tikzpicture}
  \end{equation*}
\item A red line labeled $\la_i$ and a black line labeled $\al_i$ can be separated by adding $\la^i=\al_i^\vee(\la)$ dots to the black strand:
  \begin{equation}\label{cost}
  \begin{tikzpicture}[thick,baseline=1.6cm]
    \draw (-2.8,0)  +(0,-1) .. controls +(1.6,0) ..  +(0,1) node[below,at start]{$i$};
       \draw[wei] (-1.2,0)  +(0,-1) .. controls +(-1.6,0) ..  +(0,1) node[below,at start]{$\la$};
           \node at (-.3,0) {=};
    \draw[wei] (2.8,0)  +(0,-1) -- +(0,1) node[below,at start]{$\la$};
       \draw (1.2,0)  +(0,-1) -- +(0,1) node[below,at start]{$i$};
       \fill (1.2,0) circle (3pt) node[left=3pt]{$\la^i$};
          \draw[wei] (-2.8,3)  +(0,-1) .. controls +(1.6,0) ..  +(0,1) node[below,at start]{$\la$};
  \draw (-1.2,3)  +(0,-1) .. controls +(-1.6,0) ..  +(0,1) node[below,at start]{$i$};
           \node at (-.3,3) {=};
    \draw (2.8,3)  +(0,-1) -- +(0,1) node[below,at start]{$i$};
       \draw[wei] (1.2,3)  +(0,-1) -- +(0,1) node[below,at start]{$\la$};
       \fill (2.8,3) circle (3pt) node[right=3pt]{$\la^i$};
  \end{tikzpicture}
\end{equation}
\end{enumerate}

\end{definition}
This algebra is graded by setting: each black crossing with label $i$ and $j$ is of degree $-2$ if $i=j$, of degree $1$ if $j=\mp 1$ and of degree zero otherwise.  A crossing of a red line labeled $\la$ and a black line labeled $\al_i$ is of degree $\la^i=\al_i^\vee(\la)$,  and a dot is of degree $2$.\\

The geometrical tensor algebra agrees with the diagrammatical tensor algebra:

\begin{prop}
\label{eAeT}
We have an isomorphism of graded rings
$$\tilde{\varepsilon}\colon \tilde{T}^\bnu\cong
\eT\tAb\eT$$ sending each crossings \eqref{cross} of black strands to the corresponding composition of merge and split and the crossings of a red with a black strand to the corresponding shift map in $\eT\tAb\eT$, and finally matching a dot on the
$k$th strand from the left that is labeled $i$ with the
polynomial generators $x_{i,k}$.
\end{prop}
\begin{proof}
In order to check that this map is well-defined, we must check the
relations of $\tilde{T}^\bnu$ in $\eT\tAb\eT$. By Proposition \ref{hl-faithful}, we have a faithful action of
$\eT\tAb\eT$ on 
\begin{eqnarray}
\label{eV}
\eT \tilde{V}^\bnu&\cong &\bigoplus_{\bmuh(i)\in
  \fCompe(\bd)}H^*(\fQ(\grave\mu))\cong \bLa(\bmuh(0)\cup\cdots\cup
\bmuh(\ell)).
\end{eqnarray}
This is a sum of polynomial rings, corresponding to $\ell+1$-tuples of
sequences of simple roots. The explicit formulas for this action
from Propositions \ref{euler-int} and \ref{red-euler} are shown in Figure~\ref{poly-rep}. (There we always display the relevant part of the diagrams indicating the geometric convolution operation. We depict the $k$th strand and, apart from the last picture, the $k+1$th strand and write $y_j$ for the variable attached to the $j$th strand.)
\begin{figure}[h]
\begin{center}
\begin{tikzpicture}
\node at (-4,3.5){
\begin{tikzpicture}[scale=.9]
\draw[wei] (-1,-1) -- (1,1) node[at start,below]{
$\la$} node[at end,above]{
$\la$};
\draw[thick] (1,-1) -- (-1,1) node[at start,below]{
$i$} node[at end,above]{
$i$};
\draw[thick,|->] (1.5,0) -- (2.5,0);
\node at (4,0) {$f\mapsto 1\cdot f$ } ;
\end{tikzpicture}};

\node at (4,3.5){
\begin{tikzpicture}[scale=.9]
\draw[wei] (1,-1) -- (-1,1) node[at start,below]{
$\la$} node[at end,above]{
$\la$};
\draw[thick] (-1,-1) -- (1,1) node[at start,below]{
$i$} node[at end,above]{
$i$};
\draw[thick,|->] (1.5,0) -- (2.5,0);
\node at (4,0) {$f\mapsto y_k^{\la^i}\cdot f$ } ;
\end{tikzpicture}
};

\node at (-2.5,0){
\begin{tikzpicture}[scale=.9]
\draw[thick] (1,-1) -- (-1,1) node[at start,below]{
$j$} node[at end,above]{
$j$};
\draw[thick] (-1,-1) -- (1,1) node[at start,below]{
$i$} node[at end,above]{
$i$};
\draw[thick,|->] (1.5,0) -- (2.5,0);
\node at (5.5,0) {$\displaystyle
\begin{cases}
f\mapsto f^{s_k} & \text{if }i \nrightarrow j\\
f\mapsto (y_{k+1}-y_{k}) f^{s_k} &\text{if } i\to j\\
\displaystyle  f\mapsto \frac{f-f^{s_k}}{y_k-y_{k+1}} &\text{if } i=j \\
\end{cases}$ } ;
\end{tikzpicture}
};

\node at (5,0){
\begin{tikzpicture}[scale=.9]
\draw[thick] (1,-1) -- (1,1) node[at start,below]{
$i$} node[at end,above]{
$i$} node[circle,midway,fill,inner sep=2pt]{};
\draw[thick,|->] (1.5,0) -- (2.5,0);
\node at (4,0) {$f\mapsto y_k\cdot f$ } ;
\end{tikzpicture}
};
\end{tikzpicture}

\end{center}
\caption{The polynomial representation}
\label{poly-rep}
\end{figure}
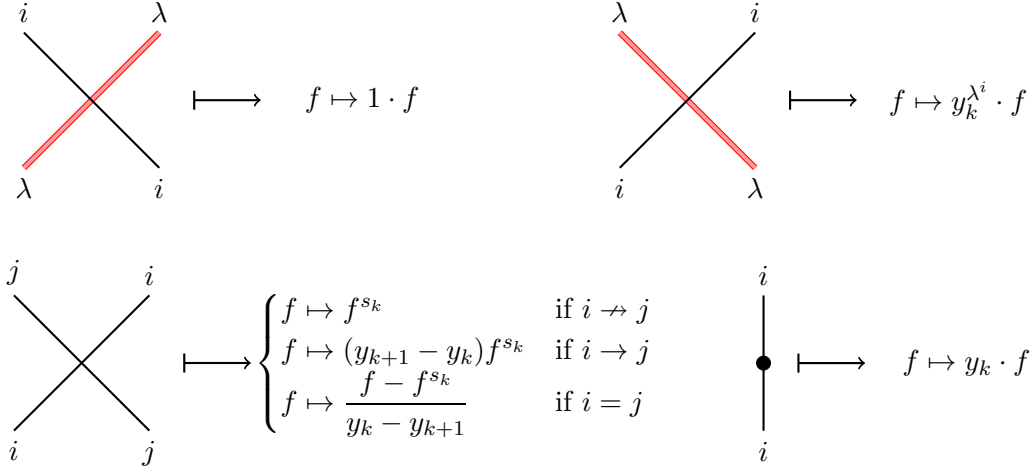
An easy direct calculation shows that the assignments satisfy the relations (R1)-(R3), hence $\tilde{T}^\bnu$ acts on \eqref{eV}, too,  by the assignments from Figure \ref{poly-rep}. By \cite[Lem. 4.17]{Webmerged} this action of $\tilde{T}^\bnu$ is faithful. (We can obtain the indexing used in \cite{Webmerged} by taking $\Bi$ to be the
concatenation of the sequences of roots in the description in this
paper, and the function $\kappa$ to send $j$ to
$\sum_{k<j}|\bmuh(k)|$.) This shows we have an injection $\tilde{\varepsilon}:\tilde{T}^\bnu\hookrightarrow
\eT\tAb\eT$ since we matched faithful representations of these
algebras.  This map satisfies by construction all the assertions of the proposition. Moreover, it sends the basis of \cite[4.17]{Webmerged} for
$\tilde{T}^\bnu$ to a basis of the type given by  $\eT\tAb\eT$ in
Proposition \ref{red-basis}.  Thus, this map is an isomorphism.
\end{proof}

\begin{figure}[h]
\begin{equation*}
    \begin{tikzpicture}[scale=1]
      \draw[thick](-4,0) +(-1,-1) -- +(1,1) node[below,at start]
      {$i$}; \draw[thick](-4,0) +(1,-1) -- +(-1,1) node[below,at
      start] {$j$}; \fill (-4.5,.5) circle (3pt);
      % \draw[thick] (0,0) +(0,-1) -- +(0,1) node[below, at
      % start]{$i$}; \fill (0,0) circle (5pt);
      \node at (-2,0){=}; \draw[thick](0,0) +(-1,-1) -- +(1,1)
      node[below,at start] {$i$}; \draw[thick](0,0) +(1,-1) --
      +(-1,1) node[below,at start] {$j$}; \fill (.5,-.5) circle (3pt);
      \node at (4,0){unless $i=j$};
    \end{tikzpicture}
  \end{equation*}
  \begin{equation*}
    \begin{tikzpicture}[scale=1]
      \draw[thick](-4,0) +(-1,-1) -- +(1,1) node[below,at start]
      {$i$}; \draw[thick](-4,0) +(1,-1) -- +(-1,1) node[below,at
      start] {$i$}; \fill (-4.5,.5) circle (3pt);
      % \draw[thick] (0,0) +(0,-1) -- +(0,1) node[below, at
      % start]{$i$}; \fill (0,0) circle (5pt);
      \node at (-2,0){=}; \draw[thick](0,0) +(-1,-1) -- +(1,1)
      node[below,at start] {$i$}; \draw[thick](0,0) +(1,-1) --
      +(-1,1) node[below,at start] {$i$}; \fill (.5,-.5) circle (3pt);
      \node at (2,0){+}; \draw[thick](4,0) +(-1,-1) -- +(-1,1)
      node[below,at start] {$i$}; \draw[thick](4,0) +(0,-1) --
      +(0,1) node[below,at start] {$i$};
    \end{tikzpicture}
  \end{equation*}
 \begin{equation*}
    \begin{tikzpicture}[scale=1]
      \draw[thick](-4,0) +(-1,-1) -- +(1,1) node[below,at start]
      {$i$}; \draw[thick](-4,0) +(1,-1) -- +(-1,1) node[below,at
      start] {$i$}; \fill (-4.5,-.5) circle (3pt);
      % \draw[thick] (0,0) +(0,-1) -- +(0,1) node[below, at
      % start]{$i$}; \fill (0,0) circle (5pt);
      \node at (-2,0){=}; \draw[thick](0,0) +(-1,-1) -- +(1,1)
      node[below,at start] {$i$}; \draw[thick](0,0) +(1,-1) --
      +(-1,1) node[below,at start] {$i$}; \fill (.5,.5) circle (3pt);
      \node at (2,0){+}; \draw[thick](4,0) +(-1,-1) -- +(-1,1)
      node[below,at start] {$i$}; \draw[thick](4,0) +(0,-1) --
      +(0,1) node[below,at start] {$i$};
    \end{tikzpicture}
  \end{equation*}
  \begin{equation*}
    \begin{tikzpicture}[thick,scale=1]
      \draw (-2.8,0) +(0,-1) .. controls +(1.6,0) ..  +(0,1)
      node[below,at start]{$i$}; \draw (-1.2,0) +(0,-1) .. controls
      +(-1.6,0) ..  +(0,1) node[below,at start]{$i$}; \node at (-.5,0)
      {=}; \node at (0.4,0) {$0$};
\node at (1.5,.05) {and};
    \end{tikzpicture}
\hspace{.4cm}
    \begin{tikzpicture}[thick,scale=1]
      \draw (-2.8,0) +(0,-1) .. controls +(1.6,0) ..  +(0,1)
      node[below,at start]{$i$}; \draw (-1.2,0) +(0,-1) .. controls
+(-1.6,0) ..  +(0,1) node[below,at start]{$j$}; \node at (-.5,0)
      {=};
\draw (1.8,0) +(0,-1) -- +(0,1) node[below,at start]{$j$};
      \draw (1,0) +(0,-1) -- +(0,1) node[below,at start]{$i$};
\node[inner xsep=10pt,fill=white,draw,inner ysep=8pt] at (1.4,0) {$Q_{ij}(y_1,y_2)$};
    \end{tikzpicture}
  \end{equation*}
  \begin{equation*}
    \begin{tikzpicture}[thick,scale=1]
      \draw (-3,0) +(1,-1) -- +(-1,1) node[below,at start]{$k$}; \draw
      (-3,0) +(-1,-1) -- +(1,1) node[below,at start]{$i$}; \draw
      (-3,0) +(0,-1) .. controls +(-1,0) ..  +(0,1) node[below,at
      start]{$j$}; \node at (-1,0) {=}; \draw (1,0) +(1,-1) -- +(-1,1)
      node[below,at start]{$k$}; \draw (1,0) +(-1,-1) -- +(1,1)
      node[below,at start]{$i$}; \draw (1,0) +(0,-1) .. controls
      +(1,0) ..  +(0,1) node[below,at start]{$j$}; \node at (5,0)
      {unless $i=k\neq j$};
    \end{tikzpicture}
  \end{equation*}
  \begin{equation*}
    \begin{tikzpicture}[thick,scale=1]
      \draw (-3,0) +(1,-1) -- +(-1,1) node[below,at start]{$i$}; \draw
      (-3,0) +(-1,-1) -- +(1,1) node[below,at start]{$i$}; \draw
      (-3,0) +(0,-1) .. controls +(-1,0) ..  +(0,1) node[below,at
      start]{$j$}; \node at (-1,0) {=}; \draw (1,0) +(1,-1) -- +(-1,1)
      node[below,at start]{$i$}; \draw (1,0) +(-1,-1) -- +(1,1)
      node[below,at start]{$i$}; \draw (1,0) +(0,-1) .. controls
      +(1,0) ..  +(0,1) node[below,at start]{$j$}; \node at (2.8,0)
      {$+$};        \draw (6.2,0)
      +(1,-1) -- +(1,1) node[below,at start]{$i$}; \draw (6.2,0)
      +(-1,-1) -- +(-1,1) node[below,at start]{$i$}; \draw (6.2,0)
      +(0,-1) -- +(0,1) node[below,at start]{$j$};
\node[inner ysep=8pt,inner xsep=5pt,fill=white,draw,scale=.8] at (6.2,0){$\displaystyle \frac{Q_{ij}(y_3,y_2)-Q_{ij}(y_1,y_2)}{y_3-y_1}$};
    \end{tikzpicture}
  \end{equation*}
\caption{The relations of the quiver Hecke algebra.}
\label{quiver-hecke}
\end{figure}
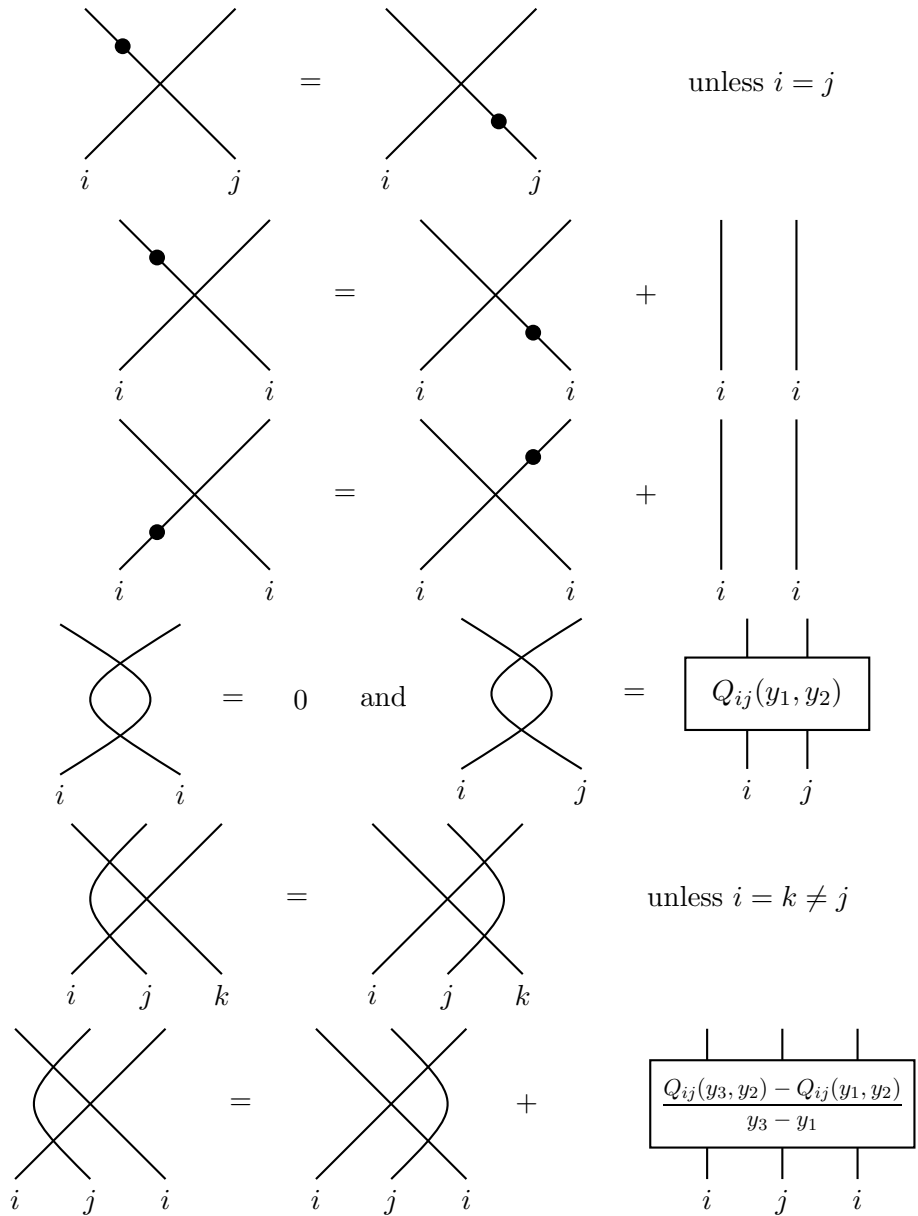

\subsection{Cyclotomic quotients}
The quiver Hecke algebra $R(\bd)$ has a very interesting family of finite
dimensional quotients, the cyclotomic quiver Hecke algebras $T^\la$ (implicitly depending on our choice of $\bd$) where
$\la$ is an integer valued function on $\Gamma$.  We will usually
think of this function as a weight of the affine Lie algebra
$\widehat{\mathfrak{sl}}_e$.  As shown in \cite{BKKL}, blocks of the diagrammatically defined cyclotomic quotients of the quiver Hecke algebra are
isomorphic to blocks of cyclotomic Hecke
algebras for symmetric groups. The latter are Hecke algebra which behave as though they are of ``characteristic $e$''
(i.e. the cyclotomic quotient is defined using powers of an element
$q$ from the ground field and $e$ is the smallest number such that
$1+q+q^2+\cdots +q^{e-1}=0$).

We want now to define similar quotients ${A}^\bnu$ of $\tAb$:

\begin{definition}
Let $I=I^\bnu$ be the ideal of ${A}^\bnu$ generated by all elements of the form $e_{\alpha}|a$, where $\alpha$ is any root.  We call this ideal the {\bf violating ideal} and the quotient ${A}^\bnu=\tAb/I$ the {\bf cyclotomic quotient} of  ${A}^\bnu$ and call ${A}^\bnu$ the {\bf cyclotomic quiver Schur algebra}.
\end{definition}
Pictorially, $I$ is generated by all diagrams where at some point the left-most strand is black and labeled with some root, see Figure \ref{violating} for an example.
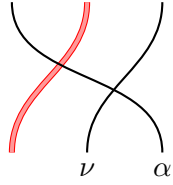
\begin{figure}[h]
\begin{tikzpicture}
\draw[wei] (0,0) to [in=-90,out=90] (1,2);
\draw[thick] (1,0) node (a) [below]{$\nu$} to [in=-90,out=90] (2,2);
\draw[thick] (2,0) node (b) [below]{$\al$} to [in=-90,out=90] (0,2);
\end{tikzpicture}
\caption{An example of an element of the violating ideal.}
\label{violating}
\end{figure}
Obviously, the idempotent \eqref{idempotent} associated to an $\ell+1$-tuple will be $0$ in this quotient unless it begins
with the empty set (hence a red strand to the left).  To avoid much tedious writing of $\emptyset$, we shall henceforth just deal with $\ell$-tuples, and leave the initial $\emptyset$ as given.

Note that if $\bnu=(\la)$, and $r$ is the element which is just a red
strand labeled with $\la$, then $a\mapsto r|a$ followed by projection
defines an algebra map $\tilde{A}^\la\to {A}^\la$ which is surjective. The kernel
of this map is called the {\bf cyclotomic ideal} of $A$.  An argument
analogous to \cite[Th. 4.20]{Webmerged} shows that this ideal has a
definition looking more like a traditional cyclotomic ideal.

Let $T^\bnu=\tilde {T}^\bnu/J$ be the cyclotomic quotient of the tensor algebra as defined in \cite[Def. 4.6]{Webmerged}.
\begin{prop} \label{Bencycliso}
\(\eT {A}^\bnu \eT\cong T^\bnu \) \end{prop}

\begin{proof}
The surjection $\Psi\colon \eT {\tilde{A}}^\bnu \eT\to \eT {A}^\bnu \eT$ induces a
surjective map $\Psi\circ \tilde{\varepsilon}\colon \tilde{T}^\bnu\to \eT {A}^\bnu \eT$.  This map
obviously kills the violating ideal $J$, and thus induces a surjective map
$\varepsilon\colon T^\bnu\to \eT {A}^\bnu \eT$.  Thus in order to
check that this is an isomorphism, we need to check that $\ker
(\Psi\circ \tilde{\varepsilon})\subset J$.

By definition, any element $x\in \ker
(\Psi\circ \tilde{\varepsilon})$ is a sum of elements of the form $x'=abc$ where $a\in
\eT \tAb, c\in \tAb \eT$ and $b=e_{\al}|b'$ is a generator of $I$ as above. First, we apply Proposition \ref{red-basis} to $a$ and $c$, so we may
assume that $a=a_1a_2$ and $c=c_2c_1$ where  $a_1,c_1\in \eT \tAb \eT$, and $a_2$ just joins and $c_2$ just splits
strands.  The element $a_2$ does not move any black strands to the
right of reds, so at its top, there is still at least one black strand still left of all reds.  This
shows immediately that
the element $a_1$ is in the violating ideal.  Thus, the factorization $x'=a_1(a_2bc)$ shows that $x'\in
J$.  This completes the proof.
\end{proof}

\begin{prop}\label{T-morita}
  When $e=\infty$, this idempotent induces a Morita equivalence
  between ${A}^\bnu $ and $T^\bnu$
\end{prop}
\begin{proof}
 We need only note that the idempotent $e_{\bd}$ factors through the horizontal
  product $\cdots |e_{d_{1}\al_{1}} |e_{d_0\al_0}|
  e_{d_{-1}\al_{-1}}|\cdots$ of the idempotents $e_{d_i\al_i}$ in
  order (this is a finite product, since $|\bd|<\infty$); the split
  and merge just gives the identity.
\end{proof}

\section{A graded cellular basis}
\label{sec:cyclotomic-quotients}

\subsection{Cellular basis of cyclotomic quotients}
\label{sec:basis-defined}
We specialize in the remaining sections to the case
where $\bnu$ is a sequence of fundamental weights, so
$\bnu=(\omega_{\charge_1},\dots,\omega_{\charge_\ell})$, where the $j$th fundamental weight is of the form $\omega_j(\alpha_k)=\delta_{j,k}$. {\it We also assume for simplicity from now on $e>2$.}
We wish to show that ${A}^\bnu$ has a cellular basis, turning it into a
graded cellular algebra. Let us first recall the relevant definitions from \cite{GL}.
\begin{definition}
  A {\bf cellular algebra} is an associative unital algebra $H$
  together with a {\bf cell datum} $(\La, M, C, *,>)$ such
  that 
  \begin{enumerate}[(C1)] 
  \item $(\La,>)$ is a partially ordered set and
    $M(\la)$ is a finite set for each $\la \in \La$; 
    \item
    $C:\dot\bigcup_{\la \in \La} M(\la) \times M(\la) \rightarrow H,
    (\sS,\sT) \mapsto C^\la_{\sS,\sT}$ is an injective map whose image is a
    basis for $H$;
  \item the map $*:H \rightarrow H$ is an algebra
    anti-automorphism such that $(C^\la_{\sS,\sT})^* = C_{\sT,\sS}^\la$ for
    all $\la \in \La$ and $\sS, \sT \in M(\la)$; 
    \item if
    $\la \in \La$ and $\sS, \sT \in M(\la)$ then for any $x \in
    H$ we have that $$ x C_{\sS,\sT}^\la \equiv \sum_{\sS'
      \in M(\xi)} r_x(\sS',\sS) C_{\sS',\sT}^\la \pmod{H(>
      \la)}$$ where the scalar $r_x(\sS',\sS)$ is independent of
    $\sT$ and $H(> \la)$ denotes the subspace of $H$ generated by
    $\{C_{\sS'',\sT''}^\nu\:|\:\nu > \la, \sS'',\sT'' \in
    M(\nu)\}$. 
    \end{enumerate} 
    The basis consisting of the $C^\la_{T,S}$
  is then called a {\bf cellular basis} of $H$.
\end{definition}

If there is  additionally a degree function $\op{deg}:\bigcup_{\la\in\la} M(\la)\rightarrow\mathbb{Z},$ $S\mapsto \op{deg}_S^\la$
with the property $\op{deg}(C_{S,T})=\op{deg}_S^\la+\op{deg}_T^\la$ (for $(S,T)\in M(\la)\times M(\la)$) such that $H$ turns into a graded algebra, then $H$ is called a {\bf graded cellular algebra} with graded cell datum $(\La, M, C, *,<,\op{deg})$, see \cite{HM}.

We want to construct now such a basis diagrammatically, indexed by pairs $(T,S)$ of semi-standard Young tableaux of the same shape $\la$
(but, not necessarily of the same type) with the shape $\la$ ranging over the set $\Lambda$ of all
$\ell$-multi-partitions (i.e. an $\ell$-multi-composition where the parts are partitions) ordered lexicographically. The involution $*$ will be given by a vertical reflection of the diagram in the diagrammatical
realization.

\begin{definition}
\label{sstableau}
Given an $\ell$-multi-partition $\hat{\la}$, a {\de semi-standard} $\hat{\la}$-{\bf tableau} is a filling of the boxes of its $\ell$-tuple of Young diagrams with numbers from the $\ell$-fold disjoint union of $\mZ_{\geq 0}$, usually denoted with a subscript to show which of the $\ell$ copies it comes from, subject to the following rules (when the partitions are drawn in the English style):

\begin{itemize}
\item the entries in each component are in each row weakly increasing from left to right and in each column strictly increasing from top to bottom, (with respect to the lexicographic order first in the subscripts and then the numbers themselves);
\item the $i$th copy of $\mZ$ can only be used in the first $i$ partitions of the multi-partition;
\item We will also always assume that our tableaux have no gaps, i.e. if $j_k$ appears, then the $i_k$ for all $i\leq j$ also  appear in the tableau.
\end{itemize}
\end{definition}

A semi-standard $\hat{\la}$-tableau $\sS$ is called {\de standard} or
a {\de standard multitableau}, if each of the entries only appears
once and {\de super-standard} if it contains only the numbers $1,\dots,
n$ from the last alphabet.

The {\de type} of $\sS$ is $\bmu=(\bmu^{(1)},\bmu^{(2)},\ldots, \bmu^{(\ell)})$ where $\bmu_i^{(k)}$ denotes the multiplicity of the number $i$ coming from the $k$th copy of the alphabet.

\begin{ex} The example from \cite[4.9]{DJM} of shape $((4,3),(2,1),(2,1))$ ,
$$\sS=\left(\young(\oneone \oneone \oneone \twoone,\twoone\twoone\threeone)\;,\; \young(\onetwo\threethree,\twotwo)\;,\;\young(\onethree\onethree,\twothree)\right)$$
is a semistandard multitableau of type $((3,3,1),(1,1,0),(2,1,1))$.
\end{ex}

Note that the no-gap condition in Definition \ref{sstableau} restricts the possible types of semistandard tableaux to $\ell$-multi-compositions, where $\bmu_i^{(k)}=0$ implies $\bmu_j^{(k)}=0$ for all $j>i$. We call such multi-compositions {\bf admissible}. Note that multi-partitions are automatically admissible and we have the following obvious fact:

\begin{lemma}
\label{admiss}
Let $\bmu$ be an admissible $\ell$-multi-composition of $n$ and $\bla \triangleright\bmu$ be an $\ell$-multi-partition of $n$ larger than $\bmu$ in the dominance ordering. Then $\bla$ is (viewed as a multi-composition) also admissible.
\end{lemma}

In particular, following \cite[\S6]{DJM}, we have the cyclotomic $q$-Schur algebra associated to the set $\Lambda$ of admissible $\ell$-multi-compositions of $n$ defined in \eqref{DefqSchur} as the endomorphism ring of the direct sum of signed permutation modules associated with $\mu\in\Lambda$.
%  It is then called {\de super standard} if the row reading word (i.e. the sequence obtained by reading the rows from top to bottom in the first partition followed by the second etc.) is increasing, see \eqref{groundstate} for an example of a super standard standard multitableau.\\
As shown in \cite{DJM},  the Ariki-Koike algebra
$\mathfrak{H}(q,Q_1,\ldots, Q_\ell)$ from the introduction is a cellular algebra with basis labeled by pairs of standard $\hat{\la}$-tableaux of the same shape, but varying over all $\hat{\la}$.\\

We fix an $\ell$-tuple of vertices $\charge=(\charge_1,\dots,\charge_\ell)$ in $\mZ/e\mZ$, which is called the {\bf charge}.  Given a Young diagram, the {\bf content} of a box in the $i$th row and $j$th column of is $j-i$. If it appears in the $k$th Young diagram of an $\ell$-multi-partition, then its {\bf residue} is $\charge_k+j-i\pmod e$.

Let ${\sS}$ be a  semi-standard $\hat{\la}$-tableau. The {\bf residue row-sequence} of ${\sS}$ is  the sequence obtained by reading the residues along the rows from top to bottom starting in the first partition followed by the second partition etc. Hereby the rows are separated by vertical black lines $\mid$ and the partitions themselves with vertical red lines (see e.g. Figure~\ref{idempotent}).
The {\bf residue entry-sequence} is the sequence obtained by reading the the residues along the entries starting from $1_1$, $1_2$, etc. ending with the residue of the box with the largest entry. In case an entry appears more than once we order them according to their appearance in the residue sequence.

Then $w_{\sS}$ denotes the permutation of minimal length such that $w_{\sS}$ applied to the row reading word of $\sS$ is increasing.
In particular, $w_\sS$ is a shortest coset representative for the
Young subgroup $S_{\la}$ attached to the rows (since the entries of
each row are already weakly ordered) acting from the right, and the Young subgroup $S_{\mu}$ associated to the
multiplicities of the entries in $\sS$ (that is, the Young subgroup
that fixes $w_{\sS}$ times the row reading word) acting from the left.

The compositions
$\la$ and $\mu$ get refined by the following two vector
$\ell$-multi\-com\-po\-sitions
$\grave{\la}_\sS=(\blah_\sS(1),\dots,\blah_\sS(\ell))$ and
$\grave{\mu}_\sS=(\bmuh_\sS(1),\dots,\bmuh_\sS(\ell))$: \small
\begin{eqnarray*}
\grave{\la}_\sS(k)[g,h]&=&\#\{\text{boxes of residue $h$ in the $g$th
  row of the $k$th Young diagram}\}\\ \nopagebreak
\grave{\mu}_\sS(k)[g,h]&=&\#\{\text{boxes of residue $h$ and entry $g_{k}$}\}
\end{eqnarray*}
\normalsize
Note that the first does not depend on the entries of $\sS$, but only
on the shape. In particular, if $\xi$ is a shape, then we will use
$\gla_\xi$ to denote $\gla_\sS$ for a tableau of shape $\xi$.

\begin{ex}
\label{runningex}
Let $e=3$, $l=1$ and $\hat{\la}=\yng(4,3,2)$. Then $\sS=\young(1125,244,3)$ is a semistandard $\hat{\la}$-tableau of type $\mu=(2,2,1,2,1)$. The row reading word is $1,1,2,5,2,4,4,3$ and therefore $w_\sS=\begin{pmatrix}1&2&3&4&5&6&7&8\\ 1&2&3&8&4&6&7&5\end{pmatrix}\in S_8$, a shortest coset representative for $S_\mu \backslash S_8/S_\la$, where $S_\la=S_4\times S_3\times S_2$ and $S_\mu=S_2\times S_2\times S_1\times S_2\times S_1$. If $\charge=(1)$, then the residues are $\young(1231,312,2)$, the residue sequence is $1,2,3,1|3,1,2|2$ and we have
\begin{eqnarray*}
\grave{\mu}_\sS&=&(((1,1,0),(0,0,2),(0,1,0),(1,2,0),(1,0,0))),\\
\grave{\la}_\sS&=&(((2,1,1),(1,1,1),(0,1,0)).
\end{eqnarray*}
\end{ex}
%\begin{ex}
%\label{ex:notmulti}
%Let
%$\charge=(1)$ and $e=3$. Consider the semi-standard tableau \[\sS=\young(1125,244,3)\qquad  \text{
%  with the residues given by} \quad \young(1231,312,2).\] Then the row sequence is $11252443$ and
%$w_{\sS}=s_7s_6s_5s_6s_7s_4\in S_8$ is a shortest right coset representative
%modulo the subgroup $S_\mu=S_2\times S_2\times S_1 \times S_2 \times
%S_1=\langle s_1,s_3,s_6\rangle$ corresponding to the type $\mu$ and a
%shortest left coset representative modulo the subgroup
%$S_4\times S_3\times S_1=\langle s_1,s_2,s_3,s_5,s_6\rangle$ corresponding to the shape.
%Furthermore, we have
%\begin{eqnarray*}
%\grave{\la}_\sS&=&((1,1,0),(0,0,1),,(1,0,0), (0,0,1),(1,1,0))\quad \text{ and }\quad\\
%\grave{\mu}_\sS&=&((1,1,0),(0,0,1),(0,0,1),(0,1,0),(1,1,0),(1,0,0),(1,0,1))
%\end{eqnarray*}
%with the corresponding (equal) Young  subgroups $S_2\times S_1\times S_1\times S_1\times S_1\times S_2\times S_1$.
%}
%\end{ex}

By construction $w_\sS$ is a shortest length representative in
$S_{\grave{\la}_\sS}\backslash S_d /S_{\grave{\mu}_\sS}$, hence we can invoke Proposition \ref{red-basis} and associate a basis vector in ${A}^\bnu$:
\begin{definition}%[The basis vector associated with a semi-standard tableaux]\hfill\\
\label{CST}
Given ${\sS}$, a  semi-standard $\hat{\mu}$-tableau, denote by $B_{\sS}$ the element $\grave{\mu}_\sS\overset{w_\sS;1}\Longrightarrow\grave{\la}_\sS$ and by $B_{\sS}^*$ the element obtained by flipping the associated diagram, cf. Figure \ref{Belement},  vertically. Set $C_{\sS,\sT}=B_\sS^* B_\sT$ for $\sS$ and $\sT$ semi-standard tableaux.
\end{definition}
Since the shape of $\sS$ is determined by the idempotent attached to $\grave{\la}_\sS$, applying the definition of $C_{\sS,\sT}=B_\sS^* B_\sT$ to tableaux which are not the same shape gives $0$. Also note that
$C_{\sS,\sT}^\ast=(B_\sS^* B_\sT)^*=(B_\sT B_\sS^*)=C_{\sT,\sS}$.

\begin{figure}[h]
\usetikzlibrary{positioning}
\begin{tikzpicture}[level 1/.style={sibling distance=2em},level
distance=.8cm]
%\hspace{-.8cm}
\begin{scope}[node distance=.1mm]
%\node (A) at (-.9,0){1 , 2 , 3 , 1};
%\node (A) {1 , 2 , 3 , 1};
\node (O) at (-2.2,0){};
\node (Aleft) at (-1.2,0){};
\node (Ax) [base left=of O] {residue row sequence:};
\node (AA) at (1.4,0){$\mid$};
\node (A) [base left=of AA]{1 , 2 , 3 , 1$\quad$};

%\node (AA) [base right=of A] {$\mid$};
%\draw [thick] (1.65,.2) -- +(0,-1.9);
%\node (AA) at (1.4,-1.5){$\mid$};
%\node (B) at (1.5,0){3 , 1 , 2};
\node (B) [base right=of AA] {3 , 1 , 2};
%\node (BB) [base right=of B] {$\mid$};
%\draw [thick] (4.5,.2) -- +(0,-1.9);
\node (Ba) at (3.4,0){$\mid$};
\node (C) [base right=of Ba] {2};
\node(Cup) at (3.8,0) {};
\end{scope}

\begin{scope}[node distance=.1mm]
\node (A1) at (-.9,-1.5) {1 , 2};
\node (A1down) at (-1.2,-1.6){};
\node (A1up) at (-1.2,-1.4){};
\node (A1b) at (-.3,-1.5){$\mid$};
%\node (A1b) [base right=of A1] {$\mid$};
\node (A2) [base right=of A1b] {3};
\node (A2b) [base right=of A2] {$\mid$};
\node (A3) [base right=of A2b] {1};
%\node (A3b) [base right=of A3] {$\mid$};
\node (B1a) at (1.4,-1.5) {$\mid$};
\node (B1) [base right=of B1a] {3};
\node (B1b) at (2.2,-1.5){$\mid$};
%\node (B1b) [base right=of B1] {$\mid$};
\node (B2) [base right=of B1b] {1 , 2};
\node (B2b) at (3.4,-1.5){$\mid$};
\node (C1) [base right=of B2b] {2};
\node (C1up) at (3.8,-1.4){};
\node (C1down) at (3.8,-1.6){};
\end{scope}

\draw[thick] (A) .. controls +(down:10mm) and +(up:10mm) .. (A1);
\draw[thick] (A) .. controls +(down:10mm) and +(up:10mm) .. (A2);
\draw[thick] (A) .. controls +(down:10mm) and +(up:10mm) .. (A3);
\draw[thick] (B) .. controls +(down:10mm) and +(up:10mm) .. (B1);
\draw[thick] (B) .. controls +(down:10mm) and +(up:10mm) .. (B2);
\draw[thick] (C) -- (C1);
\draw [thick, dotted, text width=3cm][->] (Aleft) .. controls +(left:20mm) and
+(left:10mm) .. node [left] {refine according to entries} (A1up);
\draw [thick,dotted, text width=3cm][<-] (C) .. controls +(right:15mm) and
+(right:15mm) .. node [right] {group according to rows} (C1up);
\begin{scope}[node distance=.5mm]
\node (Y1) at (-.9,-3.3) {1 , 2};
\node (Y1up) at (-1.2,-3.2) {};
\node (Y1down) at (-1.2,-3.4) {};
\node (Y1b) at (-.3,-3.3){$\mid$};
%\node (Y1b) [base right=of Y1] {$\mid$};
\node (Y2) [base right=of Y1b] {3};
\node (Y2b) [base right=of Y2] {$\mid$};
\node (Y3) [base right=of Y2b] {3};
\node (Y3b) at (1.4,-3.3){$\mid$};
%\node (Y3b) [base right=of Y3] {$\mid$};
\node (Y4) [base right=of Y3b] {2};
\node (Y4b) at (2.2,-3.3){$\mid$};
%\node (Y4b) [base right=of Y4] {$\mid$};
\node (Y5) [base right=of Y4b] {1 , 2};
\node (Y5b) at (3.4,-3.3){$\mid$};
%\node (Y5b) [base right=of Y5] {$\mid$};
\node (Y6) [base right=of Y5b] {1};
\node (Y6up) at (3.8,-3.2){};
\node (Y6down) at (3.8,-3.4){};
\end{scope}

\draw [thick,dotted,text width=3cm][->] (A1down) .. controls +(left:15mm) and
+(left:15mm) .. node [left] {reorder according to entries} (Y1up);

\draw [thick,dotted, text width=3cm][<-] (C1down) .. controls +(right:15mm) and
+(right:15mm) .. node [right] {reorder according to rows} (Y6up);
\draw[thick] (A1) -- (Y1);
\draw[thick] (A2) -- (Y2);
\draw (A3) .. controls +(down:10mm) and +(up:10mm) .. (Y6);
\draw (B1) -- (Y3);
\draw[thick] (B2) -- (Y5);
\draw (C1) .. controls +(down:10mm) and +(up:10mm) .. (Y4);

\begin{scope}[node distance=.1mm]
\node (Z1) at (-.9,-4.5) {1 , 2};
\node (Z1b) at (-.3,-4.5){$\mid$};
\node (Z2) [base right=of Z1b] {$\quad$3 , 3};
\node (Z2b) at (1.4,-4.5){$\mid$};
\node (Z3) [base right=of Z2b] {2};
\node (Z3b) at (2.2,-4.5){$\mid$};
%\node (Z3b) [base right=of Z3] {$\mid$};
\node (Z4) [base right=of Z3b] {1 , 2};
\node (Z4b) at (3.4,-4.5){$\mid$};
%\node (Z4b) [base right=of Z4] {$\mid$};
\node (Z5) [base right=of Z4b] {1};
\end{scope}
\draw[thick] (Y1) -- (Z1);
\draw[thick] (Y2) .. controls +(down:10mm) and +(up:10mm) .. (Z2);
\draw[thick] (Y3) .. controls +(down:10mm) and +(up:10mm) .. (Z2);
\draw (Y4) -- (Z3);
\draw[thick] (Y5) -- (Z4);
\draw  (Y6) -- (Z5);
\draw [thick, dotted, text width=3cm][->] (Y1down) .. controls +(left:15mm) and
+(left:15mm) .. node [left] {group according to entries} (Z1);
\draw [thick,dotted, text width=3cm][<-] (Y6down) .. controls +(right:15mm) and
+(right:15mm) .. node [right] {refine according to rows} (Z5);
\end{tikzpicture}
\\
The bottom line represents $\grave{\mu}_\sS$, the top line, $\grave{\la}_\sS$.
\caption{Construction of the element $B_{\sS}$}
\label{Belement}
\end{figure}

Now, let $\Lambda$ be the set of $\ell$-multi-partitions ordered lexicographically, and for $\xi\in\Lambda$ let $M(\xi)$ be the set of $\xi$-semi-standard tableaux (which is finite thanks to the last condition in Definition \ref{sstableau}).  Let $C:M(\xi)\times M(\xi)\to {A}^\bnu$ be the map $(\sS,\sT)\mapsto C_{\sS,\sT}$ and the involution $*$ from Definition \ref{CST}.
Let $\op{deg}(B_{\sS})$ be the degree of the homogeneous element $B_{\sS}$.\\

\begin{ex}
\label{ourex}
For $e=3$ consider the semistandard
$2$-tableaux \[
\sS=\left(\young(\oneone\twoone\threeone,\fourone\onetwo,\threetwo),\young(\twotwo\fourtwo,\fivetwo)\right) \qquad \sT=\left(\young(\oneone\oneone\twoone,\twoone\onetwo,\twotwo),\young(\onetwo\twotwo,\twotwo)\right)\]
Denoting the residue and row sequences using simple roots we have 
\[B_{\sS}=\tikz[thick,baseline]{ \draw (.5,.5)
  to[out=-90,in=90]  node [above, at start]{$\delta$} node [below, at end]{$\al_2$} (.5,-.5); \draw
  (.5,.5) to[out=-90,in=90]  node [below, at end]{$\al_3$}
  (1,-.5);\draw (.5,.5) to[out=-90,in=90] node [below, at end]{$\al_1$}
  (1.5,-.5); \draw (1.5,.5) to[out=-90,in=90]  node [above, at start]{$\al_{1,2}$} node [below, at end]{$\al_1$} (2,-.5);  \draw (1.5,.5)
  to[out=-90,in=90] node [below, at end]{$\al_2$}  (3,-.5); \draw (2.5,.5) to[out=-90,in=90] node [below, at end]{$\al_3$}  node [above, at start]{$\al_3$} (4,-.5); 
\draw[wei] (3,.5) to[out=-90,in=90] node [below,
at end]{$\omega_3$}  node [above, at start]{$\omega_3$} (2.5,-.5); \draw[wei] (0,.5)
to[out=-90,in=90] node [below,
at end]{$\omega_2$}  node [above, at start]{$\omega_2$} (0,-.5); \draw (4,.5) to[out=-90,in=90] node [below,
at end]{$\al_3$}  node [above, at start]{$\al_{3,1}$} (3.5,-.5); \draw
(4,.5) to[out=-90,in=90] node [below, at end]{$\al_1$}  (4.5,-.5);
\draw (4.8,.5) to[out=-90,in=90] node [below, at end]{$\al_2$}  node [above, at start]{$\al_2$} (5,-.5);
}\qquad B_{\sT}=\tikz[thick,baseline]{ \draw (.7,.5)
  to[out=-90,in=90]  node [above, at start]{$\delta$} node [below, at end]{$\al_{2,3}$} (.6,-.5); \draw (.7,.5) to[out=-90,in=90] node [below, at end]{$2\al_1$}
  (1.4,-.5); \draw (1.5,.5) to[out=-90,in=90]  node [above, at start]{$\al_{1,2}$} (1.4,-.5);  \draw (1.5,.5)
  to[out=-90,in=90] node [below, at end]{$\al_{2,3}$}  (3.2,-.5); \draw (2.2,.5) to[out=-90,in=90]  node [above, at start]{$\al_3$} (4.5,-.5); 
\draw[wei] (2.7,.5) to[out=-90,in=90] node [below,
at end]{$\omega_3$}  node [above, at start]{$\omega_3$} (2.2,-.5); \draw[wei] (0,.5)
to[out=-90,in=90] node [below,
at end]{$\omega_2$}  node [above, at start]{$\omega_2$} (0,-.5); \draw (4,.5) to[out=-90,in=90] node [above, at start]{$\al_{3,1}$} (3.2,-.5); \draw
(4,.5) to[out=-90,in=90] node [below, at end]{$\delta$}  (4.5,-.5);
\draw (4.8,.5) to[out=-90,in=90]node [above, at start]{$\al_{2}$}  (4.5,-.5);
},\] 
abbreviating $\delta=\alpha_1+\alpha_2+\alpha_3$,  in case $\charge=(2,3)$. Then
 \[\op{deg}(B_{\sS}) =0+0+(-2+1)+0+0+0+1+0+0=0,\qquad
\op{deg}(B_{\sT})=-1+1+0+0=0\] where we listed the contributions for each strand starting with the one corresponding to the largest entry. Changing the charge for instance to  $\charge=(0,0)$ gives the degrees
 \[\op{deg}(B_{\sS}) =0+0+1+0+1+0+1+0+0=3,\qquad
\op{deg}(B_{\sT})=0+0+0+0=0.\] 
\end{ex}

\begin{theorem}\label{A-is-cellular}
 The collection $(\Lambda,M,C,*,\op{deg})$ is a graded cell datum for ${A}^\bnu$.
\end{theorem}

\subsection{Cellular basis of the quiver Schur algebra}
Before we give the proof of this theorem we deduce a couple of easy extensions to related
algebras and then introduce the combinatorics of the degrees. First, we can remove the restriction that the weights
$\nu_i$ must be fundamental.  Consider an
arbitrary list of weights $\bnu'$ and a list of fundamental weights $\bnu$ such that $\nu_k'=\nu_{j_{k-1}+1}+\cdots +
\nu_{j_k}$ for some
$j_1< \cdots < j_g$.  In this case, we can find an idempotent $e_{\bnu'}$ such that
$A^{\bnu'}=e_{\bnu'}A^{\bnu}e_{\bnu'}$, namely the idempotent defined by
\[ e_{\bnu'}e_{\gmu}x=
\begin{cases}
  0 & \text{if $\hat{\bmu}(j_m)\neq 0$
    for some $m$},\\
  e_{\gmu}x& \text{otherwise}.
\end{cases}\]
In terms of diagrams, it kills all diagrams in which a black strand
ends between the $k$ and $(k+1)$st strands for some $k\neq j_m$.

This idempotent is compatible with the basis of Theorem
\ref{A-is-cellular}.  We let $M(\la,\bnu')$ be the set of standard
tableaux of shape $\la$ which only use entries from the union of the
alphabets indexed by $j_m$ for all $m$.  We let $\becircled M(\la,\bnu')$ be the
complement of this set in standard tableaux of shape $\la$, that is,
those which contain at least one entry from the
alphabet $1_{k},2_k,\dots$ for $k\neq j_m$.

Immediately from the definition, we have that  \[e_{\bnu'}\CST e_{\bnu'}=
\begin{cases}
  \CST & \text{ if } \sS,\sT\in M(\la,\bnu'),\\
0 & \text{ if }\sS\text{ or } \sT\in \becircled M(\la,\bnu').
\end{cases}
\]
We can thus define a graded cell datum $(\Lambda,
M(-,\bnu'),C,*,\deg)$ by restricting the maps $C$ and $\deg$ to
$M(-,\bnu')$ and $*$ to $A^{\bnu'}$.
\begin{corollary}\label{co:general-nu}
 The collection $(\Lambda, M(-,\bnu'),C,*,\deg)$ is a graded cell datum for
 $A^{\bnu'}$.
\end{corollary}
The cellular bases from Corollary \ref{co:general-nu} can be used to define a cellular basis of the quiver Schur algebra $A_\bd$ introduced in
Section \ref{q-Schur} by realizing the latter as an inverse limit of
cyclotomic quotients.  This limit is compatible with the cellular
bases of these quotients, and thus induces a cellular basis on $A_\bd$.

We can describe this system diagrammatically.  For any finite sequence
$\bnu$, we have inclusions $\tilde A^{(\nu_\ell)}\hookrightarrow
\tilde {A}^{(\nu_{\ell-1},\nu_{\ell})}\hookrightarrow \cdots
  \hookrightarrow \tilde {A}^\bnu$. Thus, for an infinite sequence
  $\bnu_\infty=(\dots,\nu_{-1},\nu_0)$ we can construct a tower of algebras
  containing each of its finite truncations; we let $\tilde A^{\bnu_\infty}$ denote the direct
  limit of this tower.   We can imagine this represented by having infinitely many
  red strands which continue off to the left of the page from that
  labeled $\nu_0$. If we let $e_\bd$ be the idempotent that acts by 1 on
  sequences where only the last vector composition is non-zero and 0
  on all others
  (pictorially, all black strands are to the right of all red strands), then
  we have that $e_\bd \tilde A^{\bnu} e_\bd =A_\bd$ for any finite
  list $\bnu$.  Thus, it follows that $e_\bd \tilde A^{\bnu_\infty} e_\bd
  =A_\bd$ as well.

Now, consider the ideal of $\tilde A^{\bnu_\infty}$ generated by the idempotents for
every sequence where a vector composition outside the last $m+1$ is
non-zero (i.e. the ideal consisting of diagrams with a black strand that is left of the
$(m+1)$th red strand).  The quotient by this ideal is just
$A^{(\nu_{-m},\dots, \nu_0)}$, and the image of $A_\bd$ in this
quotient is $A^{\nu(m)}$ where $\nu(m)=\sum_{i=-m+1}^0\nu_i$.  This nested system of
ideals induces an inverse system
\begin{equation}
\label{invsys}
\cdots\to A^{\nu(m+1)}_\bd\to
A^{\nu(m)}_\bd\to \cdots .
\end{equation}

From now on, we consider the infinite sequence $\bnu_c$ which cycles through the fundamental weights in
cyclic order $(\dots,\omega_0, \omega_1, \dots,
\omega_{e-1},\omega_0)$; any sequence in which all fundamental weights
occur infinitely many times will suffice.
\begin{prop}\label{invlim}
For $\bnu_c$, the inverse limit of  \eqref{invsys} in the category of graded algebras is $A_\bd$.
The cellular bases of Corollary
\ref{co:general-nu} are compatible under the maps of \eqref{invsys}; that is, each
basis vector in the target is the image of a unique one in the source,
and all other basis vectors in the source are sent to 0. These bases thus
induce a basis on $A_\bd$.
\end{prop}
\begin{proof}
Since the degrees of elements of $A_\bd$ are bounded below by some
integer $k$, any two-sided ideal generated by elements of degree $\geq
p$ only contains elements of degree $\geq p+2k$.  In particular,
  For each degree $h$, if we choose $g>h-2k$, the
  cyclotomic ideal $I^{\mu_g}$ for the weight $\mu_g=g\omega_0+\cdots +g\omega_{e-1}$ is generated by
  elements of degree $g$, and thus has trivial intersection with the
  elements of degree $h$.  Thus, the map
  $A_\bd\to A^{\mu_g}_\bd$ is an isomorphism
  in degree $h$ for
  all $g>h-2k$. This shows that each graded degree of $A_\bd$ is the inverse limit of
the graded degrees of the inverse system.  This is precisely
equivalent to
$A_\bd$ being the desired inverse limit in the category of graded
algebras.

Now consider the effect of the projection map  $p\colon
A^{\nu(m)}_\bd\to A^{\nu(m-1)}_\bd$ on the basis
of elements $\CST^\la$ where $\sS$ and $\sT$ are standard
tableaux of the same shape $\la$ on $m$-multi-partitions.  If the first
partitions of $\la$ is non-empty, then $p(\CST^\la)=0$.  If the first
partition of $\la$ is empty, then
$p(\CST)=C^{\la'}_{\sS',\sT'}$ where $\la'$, $\sS'$ and $\sT'$ are
obtained by just
stripping out the empty partition.  Thus, we have the desired result.
\end{proof}

This basis of $A_\bd$ has a cellular structure, with the cell data essentially given
by taking the union of the cell data for the inverse system.  For $A^{\nu(m)}_\bd$, the cell datum
is  $(\Lambda_m,M(-,\nu(m)), C,*,\deg)$ where $\Lambda_m$ is the
$m$-multi-partitions of residue $\bd$, and $M(\la,\nu(m))$ is the set of
standard tableaux in the alphabet $\{1_0,2_0,\dots\}$.  We have an inclusion
$\Lambda_m\hookrightarrow \Lambda_{m+1}$ by adding an empty partition
in the first slot.
Let $\Lambda_\infty$ be the direct limit of these maps, which we
can think of as the sequences of partitions indexed by non-positive
integers with total residue
$\bd$.

Note that the inclusion above induces a bijection
$M(\la,\nu(m))\to M(\la,\nu(m+1))$. For each
$\la\in \Lambda_\infty$, we let $\la(m)\in \Lambda_m$ be its truncation to the
first $m$ partitions.  Thus, for all integers $k$
large enough that all non-empty partitions in $\la$ occur in the last $k$ slots,
the sets $M(\la(k),\nu(k))$ are in canonical bijection.  In view of this, we define $M(\la,\infty):=  M(\la(k),\nu(k))$
for $k\gg 0$. We therefore obtain

\begin{corollary}
 $(\Lambda_\infty, M(-,\infty), C,*,\deg)$ is a graded
 cell datum for the algebra $A_\bd$.
\end{corollary}

\subsection{The combinatorics of degrees}
\label{sec:comb-degr}

We want to emphasize that our construction can
be used to define a grading on the set of semistandard multitableaux
$\sS$ by setting $\op{deg}(\sS):=\op{deg}( B_{\sS})$. This degree function extends nicely a known combinatorial definition on standard tableaux as follows:  

For a semi-standard tableau $\sS$ let $\sS(< i_j)$ (or $\sS(\leq i_j)$) be the subdiagram of all boxes with entries smaller (or equal) $i_j$ respectively. Given a box $b$ in $\sS$ with entry say $i_j$ any column in  $\sS$ with the same entry $i_j$ is {\it blocked} with respect to $b$.
\begin{definition}
\label{defdeg}
The {\bf degree} $d_{\sS}(b)$ of $b$ in $\sS$ is the number of addable minus the number of removable boxes in the diagram $\sS(< i_j)$ which are strictly below or to the right of $b$ in the first $j$ components, not in blocked columns and of the same residue as $b$. The {\bf degree} $\op{Deg}(\sS)$ of a tableau is the sum of the degrees of its boxes,
$\sum_b d_{\sS}(b)$.
\end{definition}

\begin{ex}
{\rm We compute $\op{Deg}(\sS)$ for $\sS$ as in Example \ref{ourex} with $\charge=(0,0)$ starting with the largest entry. 
There is no addable or removable box below the entry $5_2$ and no addable or removable box of the correct residue below $4_2$. But there is then an addable box below $3_2$, but no contribution for $2_2$. For $1_2$ there is one addable box to count (in the second component). The box $b$ with $1_4$ has degree zero and for $1_3$ we have now an addable box (at the position of $b$). Finally $2_1$ and $1_1$ do not contribute anything. Hence, $\op{Deg}(\sS)=3=\op{deg}(\sS)$. In case $\charge=(2,3)$ we have no contribution for  $5_2$ and  $4_2$. There is a removable box for  $3_2$ and no contributions for  $2_2$,  $1_2$ and $4_1$. There is an addable box for  $3_1$, and no contribution for  $2_1$ and  $1_1$. Hence $\op{Deg}(\sS)=0=\op{deg}(\sS)$.}
\end{ex}

Our degree function generalizes the degree for standard tableaux from \cite[\S 4.2]{LLT}, \cite[
Definition 2.4]{AM}, \cite[\S 3.5]{BKWSpecht} with the following  (easy to verify) property:

\begin{lemma}
\label{easylemma}
Let $\sS$ be an ordinary standard tableaux, i.e. $1$-standard tableaux, with $\charge_1=c$. Assume there are $n_i$ boxes with residue $i$. Then the number of addable minus the number of removable boxes of residue $i$  equals $\langle \omega_c-\sum_{j=1}^e n_j\alpha_j,\al_i\rangle$. 
\end{lemma}

\begin{prop}\label{degrees}
We have
  $\displaystyle \op{deg}(\sS)=\op{Deg}(\sS).$
\end{prop}

\excise{
We first show the case $\ell=1$.

\begin{lemma}\label{degrees1}
Proposition \ref{degrees} hold for $\ell=1$.
\end{lemma}}
\begin{proof}
First we show the claim for standard tableaux by induction on the
number of boxes. (The base of induction being obvious). Assume we are
given a standard tableau and let $b$ be the box with the largest
entry, say of residue $i$, and with entry $m_p$.  We assume that $b$ is in the $k$th
component, in particular the $k+1$st through $\ell$th components
are empty.  By assumption, the claim holds for the tableau $T$ with $b$ removed.

In terms of diagrams, removing $b$ means we remove the black strand $L$ which is
rightmost at the bottom of the diagram.  The geometric degree $\op{deg}$ will increase by 
\begin{equation}
\label{count}
\left(\sum_{j=k+1}^p \langle \omega_{\charge_j},\al_i\rangle\right)+\left(-2n_i+n_{i+1}+n_{i-1}\right)+y.
\end{equation}
Here the $\omega_{\charge_j}$ are precisely the labels of the red lines crossed by $L$, and $n_s$ is the number of black lines crossed by $L$ which are labelled (by the simple root attached to) $s$, and $y$ is a correction term if the topmost split changes degree when removing $L$. The black strands crossed are those corresponding to rows below the box
$b$ (necessarily in  the $k$th component of $\sS$).  Since the $r$th component for $k+1\leq r\leq p$
is empty, it has no removable box, and only a single addable box of residue $\charge_{r}$.  That is, the number of
addable boxes of residue $i$ in the $k'$th component for $k'>k$ is equal to the first bracket expression in \eqref{count}.  Now, consider in the $k$th component the (ordinary) standard tableau, say  $\xi$, given by the rows below $b$ and denote $c=\charge_k$.  Note that the boxes in $\xi$ correspond precisely to the black lines which $L$ crosses. These lines contribute precisely the second bracket in \eqref{count} to the degree $\op{deg}$. Finally note that the term $\langle\omega_c,\al_i\rangle=\delta_{c,i}$ supplies the correction $y$ for the top split, since $c=i$ if and only if the residue of the row
containing $b$ is a multiple of $\delta$.  Thus by Lemma \ref{easylemma}, removing $b$ has the same effect on $\op{deg}$ and
$\op{Deg}$ and the proof is completed for standard tableaux.

  Finally, we consider the passage from standard tableaux to
semi-standard.  We work by induction on the number of boxes in the
diagram which share an entry with another box. Consider the ``highest'' box that shares an entry with another (the first encountered in the reading word), and
consider the effect of raising by $1$ the entries of all other boxes with
greater or equal entries in the same alphabet.  This tableau has fewer
boxes that share entries, and so, by induction, its degree $\op{deg}(\sS)$ agrees with $\op{Deg}(\sS)$.  Let
$i$ be the residue of this box, and $k$ its entry. The effect of this transformation for the corresponding diagram and its degree $\op{deg}$ is as follows. The degree on the second part (the permutation) in Figure \ref{Belement} stays the same, only the split and merge parts could change the degree. Thus it suffices to consider the transformation
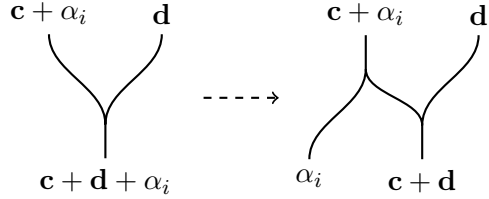
\[
\begin{tikzpicture}
 \node at (0,0) { \tikz[thick,xscale=2.5,yscale=1.5]{
\draw (0,.3) node[above] {$\bc+\al_i$} to [out=-90,in=90](.3,-.5)
(.6,.3) node[above] {$\bd$} to [out=-90,in=90] (.3,-.5)
(.3,-.5) -- (.3,-.8) node[below] {$\bc+\bd+\al_i$};}};
\draw[->,thick,dashed] (1.5,0) -- (2.5,0);
 \node at (4,0) { \tikz[thick,xscale=2.5,yscale=1.5]{
\draw (0,0) to [out=-90,in=90](.3,-.5)
(.6,.3) node[above] {$\bd$} to [out=-90,in=90] (.3,-.5)
(.3,-.5) -- (.3,-.8) node[below] {$\bc+\bd$};
\draw (-.3,-.8) node[ below]{$\alpha_i$} to[out=90,in=-90] (0,0);
\draw (0,0) to (0,.3) node[above] {$\bc+\al_i$};
}};
\end{tikzpicture}
\]
where $\bc$ is the dimension vector of the $k$'s in the same row as
our chosen box, and $\bd$ is the dimension vector of those in higher
rows.  The degrees of these morphisms are
\[ d_{i-1}-d_i+\sum c_j(d_{j-1}-d_{j})\quad \text{and}\quad
c_{i-1}-c_i+ \sum c_j(d_{j-1}-d_{j}).\]
Hence the difference in degree equals $c_{i-1}-c_i-d_{i-1}+d_i$.

On the other hand, consider the change in $\op{Deg}(\sS)$.  Nothing
changes for the degrees of the boxes below our chosen one $b$.  The degree of $b$ can be affected by every row below $b$ containing $k$'s. If that row has equal numbers
of boxes labeled $k$ with residue $i$ and $i-1$, there are no changes;
however, if there is one more with residue $i$ than $i-1$, then an
addable box of residue $i$ was created by removing the strip consisting of all boxes whose entry was changed from $k$ to $k+1$, hence the combinatorial degree changes by $1$. If there is one more $i-1$ than $i$, then an addable box of residue $i$ was
destroyed changing the degree by $-1$.  Thus, the net change is exactly $d_{i}-d_{i-1}$.  Moreover, the degrees of the boxes $b'$ whose entry was moved from $k$ to $k+1$ in the same row as $b$ could be changed. If $b'$ has residue $i$ a removable box has been created if it has $i-1$ an addable box has been removed. Hence the total difference is $c_{i-1}-c_{i}$. Thus, we arrive at the same difference, and by
induction, the proposition is proved.
\end{proof}

\subsection{The proof of cellularity}
 We will organize our proof into four parts.  The first step only concerns the tensor product algebras $T^\bnu$ and extends Hu and Mathas' basis from \cite{HM} to these algebras, the second proves that the proposed cellular basis for  ${A}^\bnu$ spans, the third shows that it  is indeed a cellular basis and the final fourth step compares the ideals with the Dipper-James-Mathas cellular ideals.
\subsection*{Step I: the case of $T^\bnu$}
 We start with the following easy observation:
\begin{lemma} We have
\begin{math}
\eT\CST\eT=
\begin{cases}
  \CST & \text{if $\sS$ and $\sT$ are standard,}\\
  0 & \text{otherwise.}
\end{cases}
\end{math}
\end{lemma}

\begin{proof}
  Recall from \eqref{eT} that $\eT$ corresponds to the vector compositions of complete
  flag type, which in turn correspond precisely to the standard
  tableaux.
\end{proof}

Thus, Theorem~\ref{A-is-cellular} would imply that the elements $\CST$
for standard tableaux give a basis of $T^\bnu$ from Proposition \ref{Bencycliso}.  However we will prove
this result first as a stepping stone to the full proof. Let $N$
denote the set of standard tableaux.
\begin{theorem}\label{T-is-cellular}
   The data $(\Lambda,N,C|_N,<,*)$ defines a cell datum for $T^\bnu$.
\end{theorem}
\begin{proof}
By restricting our construction further to tableaux $\nu$ (instead of multitableaux) and only using the last copy of $\mZ$ we obtain an algebra
\begin{equation}
\label{Tnu}
T^\nu\subset T^\bnu
\end{equation}
with a distinguished basis which  is precisely the cellular basis from \cite{HM}.

If one fixes the idempotents $\grave{\mu}_\sS$ and $\grave{\mu}_\sT$,
that is, fixes the multiset of entries in the tableaux, then the map
$(\sS,\sT)\mapsto (\sS^\diamond,\sT^\diamond)$ is injective (because
of our no gap condition).  Since we already know from \cite{HM} that
the $\CST$ for pairs of tableaux which only use the last copy of $\mZ$
are linearly independent, the map given by $a\mapsto\theta a\theta^*$
is injective, and the elements $\CST$ for standard tableaux are
linearly independent.  They are thus a basis by the dimension formula
from \cite[Lem. 4.36]{Webmerged}.  This establishes conditions (C1) and (C2) of cellularity, and (C3) is obvious from the definition as mentioned above.

Finally to see (C4), let $x\in T^\bnu$ be an arbitrary element.  Consider $xB_\sS$; since the $\CST$'s are a basis we must have $xB_{\sS}^*=\sum_{\sS',\sT'} r_x(\sS',\sS,\sT)C_{\sS',\sT'}$ for some uniquely defined coefficients $r_x(\sS',\sS,\sT')$.  For this coefficient to be not $0$ we must have $\grave\mu_\sT'=\grave\la_\sS$; in other words, $\sT'$ is a semi-standard tableaux whose type is the shape of $\sS$. Thus, either $\sT'$ is the super standard tableau $T^\sS$ of the shape
of $\sS$ (by which we mean the tableau where the entries are just the
row numbers) or the shape of $\sT'$ is above that of $\sS$ in dominance order of $\ell$-multi-partitions, hence contained in $H(\geq \mu)$, where $\mu$ is the shape of $\sS$.  If we let $r_x(\sS',\sS)=r_x(\sS',\sS, T^\sS)$ then $$x\CST=xB_\sS B_\sT^*=\sum_{\sS'}r_x(\sS',\sS)C_{\sS', T^\sS}B_\sT^*.$$
Note that $C_{\sS', T^\sS}B_\sT^*=B_\sS' B_{T^\sS}B_\sT^*=B_\sS' B_\sT^*=C_{\sS',\sT}$
since the shape of $\sT$ and $\sS'$ equal the type of $T^\sS$ and $B_{T^\sS}$ is then just the identity. Hence we have precisely condition~(C4).
\end{proof}

\begin{corollary}
$(\Lambda,N,C|_N,*,\op{deg})$ defines a graded cellular algebra.
\end{corollary}

\begin{proof}
By definition $\op{deg}(\CST)=\op{deg}(B_\sS)+\op{deg}(B_\sT)$.
\end{proof}

\subsection*{Step 2: signed permutation modules and spanning set}
Let now $n$ be a fixed natural number and let
$$R_n=\bigoplus_{|\bd|=n}R_\bd%\cong \bigoplus_{|\bd|=n}R_{\bd}
$$
be the  subalgebra of the quiver Hecke algebra of diagrams with $n$ strands.
%from Proposition \ref{VVR} constructed in \cite{VV} and will not
%distinguish in notation between the two algebras anymore.
The summand $T^\nu_n$ of the algebra $T^\nu$ from \eqref{Tnu} corresponding to partitions of $n$ is then a cyclotomic quotient $R^\nu_n$ of $R_n$.

In \cite[(4.36)]{BKKL}, an isomorphism between $R^{\nu}_n$ and a
cyclotomic Hecke algebra $\mathfrak{H}^\nu_n$ of $S_n$ with parameters
$(\zeta, \zeta^{\charge_1},\dots, \zeta^{\charge_\ell})$ was established, where $\zeta$ is an
element of the separable algebraic closure of $\K$ which satisfies
$\zeta+\zeta^2+\cdots +\zeta^{e-1}$ and $e$ is the smallest integer
where this holds.

This isomorphism $R^\nu_n\cong \mathfrak{H}^\nu_n$ from \cite[(4.36)]{BKKL} depends on the choice of certain polynomials $Q_r(\Bi)$ which we want to fix now. It will turn out to be convenient not to follow the suggested choice of \cite[(4.36)]{BKKL}, but instead fix
\begin{eqnarray}
\label{PQchoice}
Q_r({\Bi})&=&\begin{cases} 1-\zeta-\zeta y_{r+1}+y_r& \text{if } i_r=i_{r-1},\\
\frac{P_r(\Bi)-1}{y_r-y_{r+1}} &\text{if }i_{r+1}=i_{r}+ 1,\\
P_r(\Bi)-1 &\text{if }  i_{r+1}\neq i_{r},i_{r}+ 1,\\
\end{cases}
\end{eqnarray}
 with the notations from \cite[(4.27)]{BKKL} except that, as in Section \ref{KLRalgebra} (R1), our $\zeta$ is $q$ there and our $y_r, y_{r+1}$ are the $-y_r, -y_{r+1}$ there. The equation \cite[(4.28)]{BKKL} implies that ${}^{s_r}(1-P_r(s_r({\Bi})))=\zeta+P_r(\Bi)$, so the equations \cite[(4.33-35)]{BKKL} follow immediately and this choice indeed defines an isomorphism of algebras
\begin{eqnarray}
\label{BKiso}
R^\nu_n\cong \mathfrak{H}^\nu_n.
\end{eqnarray}
In the following we will identify
\begin{eqnarray}
T^\nu_n=R^\nu_n= \mathfrak{H}^\nu_n,&\text{and}& T^\nu=R^\nu:=\bigoplus_{n\geq 0} R^\nu_n =\bigoplus_{n\geq 0}\mathfrak{H}^\nu_n=:\mathfrak{H}^\nu.
\end{eqnarray}
Analogously, we will write $A^\bnu_n=\bigoplus_{|\bd|=n}A^\bnu_\bd$.

Let $M$ be a representation of
$\mathfrak{H}^\nu_n=T^\nu_n$. Recall, \cite[4.1]{BKKL}, that
$\mathfrak{H}^\nu_n$ contains the finite dimensional Iwahori-Hecke algebra as a
natural subalgebra. A vector $v\in M$ {\bf generates a sign
  representation} if it generates a 1-dimensional sign representation
for this finite Hecke algebra (i.e. $T_rv=-v$ for all generators $T_r$
in the notation of \cite{BKKL}). We first express this condition in
the standard generators $\psi_r$ of $R^\bnu_n$:
\begin{lemma}\label{sign-transform}
A vector $v\in M$ generates a sign representation if and only if it transforms under the action of $\psi_r$ by \[\psi_re_{\Bi}v=\begin{cases} 0 & \text{if } i_r=i_{r+1},\\({y_r-y_{r+1}})e_{\Bi}v &\text{if } i_{r+1}=i_{r}+ 1,\\
e_{\Bi}v &\text{if } i_{r+1}\neq i_{r},i_{r}+ 1.\\
	\end{cases}\]
\end{lemma}
\begin{proof}
  This is immediate by plugging the choices \eqref{PQchoice} into \cite[(4.38)]{BKKL} using the fact that
  $T_iv=-v$: the first case follows since in this case
  $P_r({\Bi})=1$, so $(T_r+P_r(\Bi))e(\Bi)v=0$; the last two follow
  immediately from the substitution of $-1$ for $T_r$.
\end{proof}
It might seem strange that we use here anti-invariant vectors instead of invariant vectors.  This could be fixed by picking a different isomorphism to the Hecke algebra, but it is ``hard-coded'' into the isomorphism chosen in \cite{BKKL}, since the defining quadratic relation for the Hecke algebra is $(T_r+1)(T_r-q)=0$, hence it comes along with a sign representation (but not a trivial representation where $T_r$ acts by $1$).

We are interested in a version of Lemma~\ref{sign-transform} for right $T^\nu_n$-modules $M$ (recall the obvious identification of $\mathfrak{H}^\nu_n=T^\nu_n$ with its opposite algebra):
\begin{prop}
\label{sign-transform2}
A vector $v\in M$ generates a sign representation (for the right action) if and only if it transforms under the action of $\psi_r$ by
\[ve_{\Bi}\psi_r=\begin{cases}
0 & \text{if }i_r=i_{r+1},\\
ve_{s_r(\Bi)}({y_r-y_{r+1}}) & \text{if }i_{r}=i_{r+1}+ 1,\\
ve_{s_r(\Bi)} &\text{if }i_{r}\neq i_{r+1},i_{r+1}+ 1,\\
	\end{cases}\]
where $s_r$ acts on $\Bi$ by swapping the $r$th and $r+1$th entry.
\end{prop}
\begin{proof}
This is an easy consequence from Lemma~\ref{sign-transform} using the defining relations.
\end{proof}

In \cite[\S 9.2]{Webmerged} the isomorphism \eqref{BKiso} was extended to an isomorphism
\begin{eqnarray}
\label{Beniso}
T^\bnu_n&\cong&\End_{\mathfrak{H}^\bnu_n}\left(\bigoplus_{\sum_{i=1}^\ell a_i=n} \mathfrak{H}^\nu_n u^+_{\mathbf{a}}\right)
\end{eqnarray}
where $u^+_{\mathbf{a}}$, for $\mathbf{a}=(a_1,\dots, a_\ell)\in\mathbb{Z}^\ell$, is the element defined by Dipper, James and Mathas, \cite[Definition 3.1]{DJM}, as \[u^+_{\mathbf{a}}=\prod_{s=1}^\ell\prod_{k=1}^{a_{k}}(L_k-\zeta^{\charge_s}).\]
For an $\ell$-multi-composition $\xi$, we let $\mathbf{a}_\xi=\big(0,|\xi^{(1)}|,|\xi^{(1)}|+|\xi^{(2)}|,\dots, |\xi^{(1)}|+\cdots +|\xi^{(\ell-1)}|\big)$ and, following \cite[(3.2)(ii)]{DJM}, will assume from now on that $$0\leq a_1\leq a_2\leq \cdots \leq a_\ell\leq n.$$
Thus, we can describe the cyclotomic Schur algebra $\mathscr \mathfrak{H}^\bnu_n$ from the introduction (for the parameters $(q,Q_1,Q_2,\ldots Q_\ell):=(\zeta, \zeta^{\charge_1},\dots, \zeta^{\charge_\ell})$) as the endomorphism algebra of the $T^\bnu$ module
\begin{eqnarray}
\label{Schur}
\Hom_{\mathfrak{H}^\bnu_n}\Big(\bigoplus_{|\mathbf{a}|=n} \mathfrak{H}^\bnu_n u^+_{\mathbf{a}},\bigoplus_{\xi} \mathfrak{H}^\bnu_n u^+_{\xi}x_\xi\Big)= \bigoplus_{|\xi|=n} T^\bnu  x_\xi
\end{eqnarray}
where $x_\xi$ is the projection to vectors which transform under the sign representation for the Young subgroup $S_\xi$.  We refer to the modules $T^\bnu x_\xi$ as {\bf signed permutation modules} for $T^\bnu$.

Using the cellular basis of $\mathscr \mathfrak{H}^\bnu_n$, one can directly deduce, see \cite[(4.14)]{DJM} the fact:
\begin{lemma}\label{T-per-dim}
The dimension of  $T^\bnu x_\xi$ is precisely the number of pairs $(\sS,\sT)$ of tableaux on $\ell$-multi-partitions with $n$ boxes with the same shape satisfying
\begin{itemize}
\item $\sT$ is of type $\xi$ (i.e. the number of occurrences of $i_j$ is the length of the $i$th row in the $j$th diagram of $\xi$), and
\item $\sS$ is standard.
\end{itemize}
\end{lemma}

\subsection*{Step 3: graded cellular basis}
\label{sec:proof-theorem}
We start with some preparatory lemmata. For $j\in \mZ/e\mZ$ and $n\in \mZ_{\geq 0}$ we let $e_{j;n}$ be the idempotent for the vector composition $\bmuh_{j;n}=(\al_j,\al_{j+1},\dots,\al_{j+n})$, and let $\bd_{j;n}$ be its dimension vector.
Note that for fixed $j$ and varying $n$ such dimension vectors $\bd$ are characterized by $d_{k}+1\geq d_{k-1}\geq d_{k}$
for $k\not=j$ and $d_{j-1}\leq d_{j}\leq d_{j-1}+1$.

\begin{lemma}\label{id-zero}
Let $j\in \mZ/e\mZ$ and $n\in \mZ_{\geq 0}$. Let $e_{(\bd)}$ be an idempotent corresponding to a vector composition corresponding to step $1$ flags.
In ${A}^{\omega_j}$ we have $e_{(\bd)}=0$ unless $\bd=\bd_{j;n}$ for some $n$.  Furthermore, $e_{(\bd_{j;n})}{A}^{\omega_j}e_{(\bd_{j;n})}\cong \K$.
\end{lemma}
\begin{proof}
Assume $\bd\not=\bd_{j;n}$ for any $n$. If there exists $k\not=j$ such that $d_{k}>d_{k-1}$ or  $d_k>d_{k-1}+1$ for $k=j$, then set
$R=d_{k}-d_{k-1}-1$. In either case we claim that

\[
\tikz[thick,xscale=2.5,yscale=1.5,label distance=-3pt]{
\draw[wei] (-1,0) node[below]{$\omega_j$} to (-1,1);
\draw (-.8,0) node[below]{$\bd$} -- (-.8,1) node[above]{$\bd$}
node at (-0.45,.5){$=$}
node at (-0.15,.5){};
\draw[wei] (0,0) node[below]{$\omega_j$} to (0,1);
\draw(.3,.2) to [out=90,in=-90](.1,.5)
(.3,0)-- node[below]{$\bd$} (.3,.2) to [out=90,in=-90] (.5,.5) node[right]{$\bd-\al_{k}$}
(.1,.5) to [out=90,in=-90] node[pos=0,circle,fill=black,inner sep=2pt,outer sep=0pt,label=175:{$\circledR$}]{}
(.3,.8) -- (.3,1) node[above]{$\bd$}
(.5,.5) to [out=90,in=-90] (.3,.8)
node at (1.1,0.5){$=$};
\draw[wei] (1.4,0) node[below]{$\omega_j$} to [out=90,in=-90] (1.8,.5) to [out=90,in=-90](1.4,1);
\draw (2,.2) to [out=90,in=-90](1.6,.5)
(2,0) node[below]{$\bd$} -- (2,.2) to [out=90,in=-90]
(2.2,.5) node[right]{$\bd-\al_{k}$}
(1.6,.5) to [out=90,in=-90] node[pos=0,circle,fill=black,inner sep=2pt,outer sep=0pt]{}
(2,.8) -- (2,1) node[above]{$\bd$}
(2.2,.5) to [out=90,in=-90] (2,.8)
node at (3,0.5){$=0$}
(1.3,.5) node[right]{$\circledS$}
;}
\]
where the labels $\circledR$ and $\circledS$ with $S=R-\delta_{j,k}$ denote the number of dots on the strand (following Remark \ref{KL}). Since the two cases correspond to $S=R\geq 0$ respectively $S+1=R\geq 1$, the last equality holds. The second relation is \cite[(4.2)]{Webmerged} with Proposition \ref{eAeT}. To see the first equality in case $k\not=j$ note that the second diagram corresponds to
\begin{eqnarray}
f\mapsto X(d_{k-1},d_k)(f):=\Delta_w (\prod_{r=1}^{d_{k-1}}(x_{k,1}-x_{k-1,r})x_{k,1}^Rf),
\end{eqnarray}
with
$w=s_{d_{k-1}}\cdots s_2s_1$, where $s_t$ denotes the transposition swapping the variables $x_{k,t}$ and $x_{k,t+1}$; hence it is enough to see that
$X(d_{k-1},d_k)=\op{id}$ or even $X(d_{k-1},d_k)(1)=1$.
If $d_{k-1}=0,1$ this is easily verified, and so we proceed by induction. Using formula \eqref{Demazureprop} we obtain
\begin{eqnarray*}
X(d_{k-1},d_k)&=&\Delta_w(\prod_{r=1}^{d_{k-1}-1}(x_{k,1}-x_{k-1,r})x_{k,1}^R)\\
&&+\quad\Delta_{ws_1}(\prod_{r=1}^{d_{k-1}-1}(x_{k,2}-x_{k-1,r})x_{k,1}^R)\Delta_{s_1}(x_{3,1}-x_{2,d_{k-1}})\\
&=&0+1\quad=\quad 1
\end{eqnarray*}
using twice the induction hypothesis and the fact $\Delta_{s_1}(x_{3,1}-x_{2,d_{k-1}})=1$.
In case $k=j$, the split in the diagram (just an inclusion) is followed by multiplication with $x_{j,1}^R$ and the operator $\Delta_w$ for $w=s_{d_{j-1}}\cdots s_2s_1$.
Now if $d_{k}\geq d_{k-1}$ for all $k\not= j$ and  $d_j\leq d_{j-1}+1$, we automatically have $d_k> d_{k-1}-1$ and $ d_j< d_{j-1}$. Hence if $e_{(\bd)}{A}^{\omega_j}e_{(\bd)}\not=0$ then $\bd=(\bd_{j;n})$ for some $j,n$.

In this case all elements $e_{(\bd_{j;n})}xe_{(\bd_{j;n})}$ where $x$ is of the form $\grave\mu \overset{w;h}\Longrightarrow \grave\nu$ (as in Proposition \ref{red-basis}) and $w\not=\op{id}$ or $w=\op{id}$ and $h\not=1$ can be written in terms of similar diagrams where a single strand is pulled off, so they are again zero and only the scalars remain.
\end{proof}

 Now, we consider the element
 $t^{(j;n)}:=(\prod_{k=1}^{\lfloor
   \nicefrac{n}{e}\rfloor}x_{j-1,k})e_{\bmuh_{j;n}}$, that is, $e_{\bmuh_{j;n}}$ with a dot on every
 strand labeled by $j-1$ and no dots on any others, see Figure \ref{idempotent}. This element
 plays an important role in the basis of Hu and Mathas, \cite{HM}.

\begin{lemma}\label{id-split}
Let $\alpha_{i_1}, \alpha_{i_2},\ldots \alpha_{i_n}$ be an arbitrary sequence of simple roots. Then
\[\tikz[baseline,thick,xscale=1.7, yscale=2]{
\draw[wei] (0,-.5) node[below]{$\omega_j$} -- (0,.5);
\draw (.5,-.5) node[below]{$\al_{i_1}$} to[out=90,in=-90] (1.5,-.1)
(1.5,.1) to[out=90,in=-90] (.5,.5) node[above]{$\al_{i_1}$}
(1,-.5) node[below]{$\al_{i_2}$} to[out=90,in=-90] (1.5,-.1)
(1.5,.1) to[out=90,in=-90] (1,.5) node[above]{$\al_{i_2}$}
(2,-.5) node[below]{$\al_{i_{n-1}}$} to[out=90,in=-90] (1.5,-.1)
(1.5,.1) to[out=90,in=-90] (2,.5) node[above]{$\al_{i_{n-1}}$}
(2.5,-.5) node[below]{$\al_{i_n}$} to[out=90,in=-90] (1.5,-.1)
(1.5,.1) to[out=90,in=-90] (2.5,.5) node[above]{$\al_{i_n}$}
(1.5,-.1) -- (1.5,.1);}
\quad = \quad\begin{cases} \mp t^{(j;n)} & \text{if $\al_{i_g}=\al_{j-g+1}$ for all $g$,}\\ 0 &
\text{otherwise.} \end{cases}
\]
\end{lemma}

\begin{proof}
The fact that if $\al_{i_g}\neq \al_{j-g+1}$ for some $g$ then the element is $0$ follows easily from Lemma \ref{id-zero}.
Now let $\bd=\bd_{j;n}$, and $k=j+n$.
The proof is then by induction on $n$.  When $n<e$ the claim follows directly from Lemma \ref{euler-int}. Otherwise, we claim that

\excise{Now, consider the element \[\tikz[thick,xscale=2.5,yscale=1.5]{
\draw[wei] (-.3,0) node[below]{$\omega_j$} to (-.3,1.3) ;
\draw (0,0) node[below]{$\bd$} to [out=90,in=-90](.3,.5) ;
\draw (.6,0) node[below]{$\alpha_k$} to [out=90,in=-90] (.3,.5) ;
\draw (.3,.5) -- (.3,.8) node[right, midway]{$\bd+\alpha_k$};
\draw (.3,.8) to [out=90,in=-90](.6,1.3) node[above]{$\alpha_k$};
\draw (.3,.8) to [out=90,in=-90] (0,1.3) node[above]{$\bd$};
}\]}

\begin{equation}
\label{polyb}
\tikz{
\node at (-4,0) {
\tikz[xscale=2.5,yscale=1.5,baseline=.65]{
\draw[wei] (-.3,0) node[below]{$\omega_j$} to (-.3,1.3) ;
\draw[thick] (0,0) node[below]{$\bd$} to [out=90,in=-90](.3,.5) ;
\draw (.6,0) node[below]{$\alpha_k$} to [out=90,in=-90] (.3,.5) ;
\draw[thick] (.3,.5) -- (.3,.8) node[right, midway]{$\bd+\alpha_k$};
\draw (.3,.8) to [out=90,in=-90](.6,1.3) node[above]{$\alpha_k$};
\draw[thick] (.3,.8) to [out=90,in=-90] (0,1.3) node[above]{$\bd$};
}};
\node at (-1.5,0){$=$};
\node at (1,0){\tikz[xscale=2.5,yscale=1.5,baseline=.65]{
\draw[wei] (-.4,0) node[below]{$\omega_j$} to (-.4,1.3) ;
\draw[thick] (0,0) node[below]{$\bd$} to [out=90,in=-90](0,.2) ;
\draw (.6,0) node[below]{$\alpha_k$} to [out=90,in=-90] (0,1.1) ;
\draw[thick] (0,.2) to[out=90,in=-90] (-.1,.65) to[out=90,in=-90] (0,1.1);
\draw (0,.2) to [out=90,in=-90](.6,1.3) node[above]{$\alpha_k$};
\draw[thick] (0,1.1) to [out=90,in=-90] (0,1.3) node[above]{$\bd$};
}};
\node at (3,0){$+$};
\node at (5,0){\tikz[xscale=2.5,yscale=1.5,baseline=.65]{
\draw[wei] (-.4,0) node[below]{$\omega_j$} to (-.4,1.3) ;
\draw[thick] (0,0) node[below]{$\bd$} to [out=90,in=-90](0,1.3)
node[above]{$\bd$};;
\draw (.6,0) node[below]{$\alpha_k$} to [out=90,in=-90] (.6,1.3)
node[above]{$\alpha_k$};
\node[inner xsep=30pt,inner ysep=2pt,draw,thin,fill=white] at (.3,.65){$b$};
}};}
\end{equation}

for some polynomial $b\in \bLa(\bd,\alpha_k)$.  To see this let $G$ and $F$ denote the first and second morphism. Abbreviating $a=d_k$ and setting $E=\prod_{m=1}^a(x_{k+1,m}-x_{k,a+1})$ we have
\begin{eqnarray}
G(f)&=&
\begin{cases}
E\Delta_1^{(k)}\Delta_2^{(k)}\cdots \Delta_{a-1}^{(k)}\Delta_a^{k}(f)&\text{if $d_k-d_{k+1}=1$,}\\
E\Delta_1^{(k)}\Delta_2^{(k)}\cdots \Delta_{a-1}^{(k)}(f)&\text{if $d_k=d_{k-1}$,}
\end{cases}
\end{eqnarray}
where $\Delta_i^{k}$ denotes the Demazure operator involving the variables $x_{k,i}, x_{k,i+1}$. Note that $F$ is a composition of first multiplying with an Euler class $E_1$, followed by a merge $M_1$, an inclusion and finally a merge $M_2$ as follows
\small
\begin{align*}
F&=(\bd,\al_k) \to (\bd-\al_k,\al_k,\al_k) \to (\bd-\al_k,2\al_k) \to   (\bd-\al_k,\al_k,\al_k)  \to (\bd,\al_k)(1)
\end{align*}
\normalsize
\excise{& =-(\bd-\al_k,\al_k,\al_k) \to (\bd-\al_k,2\al_k) \to   (\bd-\al_k,\al_k,\al_k)  \to (\bd,\al_k) \:\cdot\: \prod_{m=1}^{g}(x_{k,f-1}-x_{k+1,m})\\
& =-(\bd-\al_k,2\al_k) \to   (\bd-\al_k,\al_k,\al_k)  \to (\bd,\al_k) \:\cdot\: \sum_{m=1}^g (x_{k,f-1}^{g-m-1}+x_{k,f-1}^{g-m-2}x_{k,f}+\cdots + x_{k,f}^{g-m-1})e_m(\mathbf{x}_{k+1})\\
& =\begin{cases} (-1)^g &  \text{if }g=f-1\\
(-1)^g(e_1(\mathbf{x}_k)+e_1(\mathbf{x}_{k+1}))& \text{if }g=f
\end{cases}}
so we get
\begin{eqnarray*}
F(f)&=&
\begin{cases}
\Delta_1^{(k)}\Delta_2^{(k)}\cdots \Delta_{a-1}^{(k)}\Delta_a^{k}(E_1f)&\text{if $d_k-d_{k+1}=1$,}\\
\Delta_1^{(k)}\Delta_2^{(k)}\cdots \Delta_{a-1}^{(k)}(E_1f)&\text{if $d_k=d_{k-1}$.}
\end{cases}
\end{eqnarray*}
where $E_1=\prod_{m=1}^a(x_{k+1,m}-x_{k,a})$. The latter is a polynomial in $x_{k,a}$ of degree $a$. Using \eqref{Demazureprop}
we obtain
\begin{eqnarray}
F(f)&=&
\begin{cases}
\Delta_1^{(k)}\Delta_2^{(k)}\cdots \Delta_{a-1}^{(k)}\Delta_a^{k}(E_1)f+G(f)&\text{if $d_k-d_{k+1}=1$,}\\
\Delta_1^{(k)}\Delta_2^{(k)}\cdots \Delta_{a-1}^{(k)}(E_1)f+ G(f)&\text{if $d_k=d_{k-1}$.}\label{Cases}
\end{cases}
\end{eqnarray}
since all the other terms vanish. Hence the claim follows with $b=F(1)$. By an easy induction one can show
\begin{eqnarray}
\label{stupid1}
\Delta_1^{(k)}\Delta_2^{(k)}\cdots \Delta_{a-1}^{(k)}\Delta_a^{k}(x_{k,a}^n)&=&\begin{cases}1 &a=n,\\ 0&\text{otherwise,}\end{cases}
\end{eqnarray}
and so we are only interested in the leading term $(-1)^ax_{k,a}^a$ of $E_1$. On the other hand, again an easy induction shows that
\begin{eqnarray}
\Delta_1^{(k)}\Delta_2^{(k)}\cdots \Delta_{a-1}^{(k)}(x_{k,a}^n)&=&(-1)^{a-1}( e_1(k+1,a)-e_1(k,a)),
\end{eqnarray}
where $e_1(p,a)$ denotes the first elementary symmetric polynomial in the variable $x_{p,q}$, $1\leq q\leq a$.
Altogether we obtain
\begin{eqnarray}
b=F(1)&=&
\begin{cases}
(-1)^a&\text{if $d_k-d_{k+1}=1$,}\\
e_1(k+1,a)-e_1(k,a)&\text{if $d_k=d_{k-1}$.}
\end{cases}
\end{eqnarray}
Note that $d_k=d_{k-1}$ if and only if $k+1=j$.

So far we did not use the cyclotomic condition, which gives the even stronger relation
\begin{equation}
\label{polyb2}
\tikz{
\node at (-2,0) {
\tikz[xscale=2.5,yscale=1.5,baseline=.65]{
\draw[wei] (-.3,0) node[below]{$\omega_j$} to (-.3,1.3) ;
\draw[thick] (0,0) node[below]{$\bd$} to [out=90,in=-90](.3,.5) ;
\draw (.6,0) node[below]{$\alpha_k$} to [out=90,in=-90] (.3,.5) ;
\draw[thick] (.3,.5) -- (.3,.8) node[right, midway]{$\bd+\alpha_k$};
\draw (.3,.8) to [out=90,in=-90](.6,1.3) node[above]{$\alpha_k$};
\draw[thick] (.3,.8) to [out=90,in=-90] (0,1.3) node[above]{$\bd$};
}};
\node at (0,0){$=$};
\node at (2,0){\tikz[xscale=2.5,yscale=1.5,baseline=.65]{
\draw[wei] (-.4,0) node[below]{$\omega_j$} to (-.4,1.3) ;
\draw[thick] (0,0) node[below]{$\bd$} to [out=90,in=-90](0,1.3)
node[above]{$\bd$};;
\draw (.6,0) node[below]{$\alpha_k$} to [out=90,in=-90] (.6,1.3)
node[above]{$\alpha_k$};
\node[inner xsep=30pt,inner ysep=2pt,draw,thin,fill=white] at (.3,.65){$b$};
}};}
\end{equation}
by noting that the middle term in \eqref{polyb} is in fact zero by Lemma \ref{id-split} (since the idempotent after splitting of the two $\alpha_k$ causes it to vanish). Since all positive degree endomorphisms of $e_{(\bd)}$ are $0$, we have $e_1(j,a)-e_1(j-1,a)e_{\bd,\al_k}=x_{j-1;a}e_{\bd,\al_j}$ and the lemma follows therefore by induction.
\end{proof}

Given a vector composition $\bmuh$, there is an induced usual composition of length $r$ given by $\mathfrak{c}(\bmuh)=\{\sum_{j=1}^e \mu[-,j]\}$.  These compositions are endowed with the usual lexicographic order on parts (this is a refinement of dominance order).  We can define a filtration of ${A}^{\omega_j}$ by letting $A_{>\mathfrak{c}}^{\omega_j}$ be the two-sided ideal generated by $e_{\bmuh}$ for $\mathfrak{c}(\bmuh)> \mathfrak{c}$ in lexicographic order.

\begin{lemma}\label{multisplit}For a composition $\mathfrak{n}=(n_1,
n_2, \dots)$, we have

\begin{multline*}
\begin{tikzpicture}[baseline,thick,xscale=1.9, yscale=1.9]
\draw[wei] (0,-.5) node[below]{$\omega_j$} -- (0,.5);
\draw (.5,-.5) node[below,scale=.8]{$\al_{j}$} to [out=90,in=-90] (1.5,-.1);
\draw (1.5,.1) to [out=90,in=-90] (.5,.5) node[above,scale=.8] {$\al_j$};
\draw (1,-.5) node[below,scale=.8]{$\al_{j+1}$} to [out=90,in=-90]
(1.5,-.1);
\draw (1.5,.1) to [out=90,in=-90] (1,.5) node[above,scale=.8] {$\al_{j+1}$};
\draw (2,-.5) node[below,scale=.8]{$\al_{j+n_1-1}$} to [out=90,in=-90]
(1.5,-.1);
\draw (1.5,.1) to [out=90,in=-90] (2,.5) node[above,scale=.8]
{$\al_{j+n_1-1}$};
\draw (2.5,-.5) node[below,scale=.8]{$\al_{j+n_1}$} to [out=90,in=-90]
(1.5,-.1);
\draw (1.5,.1) to [out=90,in=-90] (2.5,.5) node[above,scale=.8]
{$\al_{j+n_1}$};
\draw (1.5,-.1) -- (1.5,.1);
\draw (3.5,-.5) node[below,scale=.8]{$\al_{j+n_1+1}$} to [out=90,in=-90]
(4.5,-.1);
\draw (4.5,.1) to [out=90,in=-90] (3.5,.5) node[above,scale=.8]
{$\al_{j+n_1+1}$};
\draw (4,-.5) node[below,scale=.8]{$\al_j$} to [out=90,in=-90] (4.5,-.1);
\draw (4.5,.1) to [out=90,in=-90] (4,.5) node[above,scale=.8] {$\al_j$};
\draw (5,-.5) node[below,scale=.8]{$\al_{j+n_2-1}$} to [out=90,in=-90]
(4.5,-.1);
\draw (4.5,.1) to [out=90,in=-90] (5,.5) node[above,scale=.8]
{$\al_{j+n_2-1}$};
\draw (5.5,-.5) node[below,scale=.8]{$\al_{jn_2}$} to [out=90,in=-90]
(4.5,-.1);
\draw (4.5,.1) to [out=90,in=-90] (5.5,.5) node[above,scale=.8]
{$\al_{j+n_2}$};
\draw (4.5,-.1) -- (4.5,.1);
\node[scale=2] at (6,0){$\cdots$};
\end{tikzpicture}
\quad\quad\\
\\
\equiv \quad\pm t^{(j;n_1)}| t^{(j+1;n_2)}|t^{(j+2;n_3)}|\cdots
\pmod{{A}^{\omega_j}_{>\mathfrak{c}(\bd_{j;n_1},\bd_{j+1;n_2},\dots)}.}
\end{multline*}
\end{lemma}

\begin{proof}
If the composition $\mathfrak{n}$ has 1 part, then we are done by Lemma \ref{id-split}. We assume we have $k+1$ parts and the statement is true for $k$ parts. Thus the displayed diagram is equivalent to $$t^{(j;n_1)}| t^{(j+1;n_2)}|t^{(j+2;n_3)}|\cdots|t^{(j+k-1;n_{k})}| m$$ where $m$ is the last of the double headed pitchforks, modulo $I:=A_{>\mathfrak{c}}^{\omega_j}$. By Lemma \ref{id-split}, this is the same as the desired element, plus $$t^{(j;n_1)}| t^{(j+1;n_2)}|t^{(j+2;n_3)}|\cdots|t^{(j+k-1;n_{k})}|m'$$ where $m'$ lies in the cyclotomic ideal for ${A}^{\omega_{j+k}}$. That is, it can be written so that every diagram in it has a strand at the left labeled with a single root $\alpha_q$, which in addition carries a dot if that root is $\al_{j+k-1}$. Hence we have for instance a situation

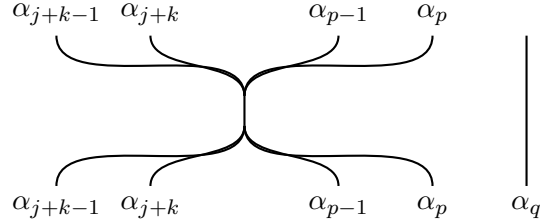
\[\tikz[baseline,thick,xscale=2.5, yscale=2]{
\draw (3,-.5) node[below]{$\al_q$} -- (3,.5);
\draw (.5,-.5) node[below]{$\al_{j+k-1}$} to[out=90,in=-90] (1.5,-.1);
\draw (1.5,.1) to[out=90,in=-90] (.5,.5) node[above]{$\al_{j+k-1}$};
\draw (1,-.5) node[below]{$\al_{j+k}$} to[out=90,in=-90] (1.5,-.1);
\draw (1.5,.1) to[out=90,in=-90] (1,.5) node[above]{$\al_{j+k}$};
\draw (2,-.5) node[below]{$\al_{p-1}$} to[out=90,in=-90] (1.5,-.1);
\draw (1.5,.1) to[out=90,in=-90] (2,.5) node[above]{$\al_{p-1}$};
\draw (2.5,-.5) node[below]{$\al_{p}$} to[out=90,in=-90] (1.5,-.1);
\draw (1.5,.1) to[out=90,in=-90] (2.5,.5) node[above]{$\al_{p}$};
\draw (1.5,-.1) -- (1.5,.1);
}
\]
In order to complete the induction, it suffices to show that this can be
written as a sum of elements each of which factors through the join of all strands (and hence is contained in $I$) or is contained in the cyclotomic ideal  ${A}^{\omega_{j+k-1}}$ (and we can repeat our argument until we finally obtain only elements in $I$ or in the cyclotomic ideal $J$ for ${A}^{\omega_{j}}$.)

If $q\not\equiv p+1\mod e$, then the diagram above is already in $J$ by Lemma \ref{id-zero}. If $q\equiv{p+1}\not\equiv  j+k-1\mod e$, then by \eqref{polyb2} we have
\[\tikz[thick,xscale=2.5,yscale=1.5, baseline=25pt]{
\draw[wei] (-.3,0) node[below]{$\omega_j$} to (-.3,1.3) ;
\draw (0,0) node[below]{$\bd$} to [out=90,in=-90](.3,.5) ;
\draw (.6,0) node[below]{$\alpha_k$} to [out=90,in=-90] (.3,.5) ;
\draw (.3,.5) -- (.3,.8) node[right, midway]{$\bd+\alpha_j$};
\draw (.3,.8) to [out=90,in=-90](.6,1.3) node[above]{$\alpha_k$};
\draw (.3,.8) to [out=90,in=-90] (0,1.3) node[above]{$\bd$};
}
\quad =\quad
\tikz[thick,xscale=2.5,yscale=1.5, baseline=25pt]{
\draw[wei] (-.3,0) node[below]{$\omega_j$} to (-.3,1.3) ;
\draw (0,0) node[below]{$\bd$} to [out=90,in=-90](0,.5) ;
\draw (.3,0) node[below]{$\alpha_k$} to [out=90,in=-90] (.3,.5) ;
\draw (.3,.5) to [out=90,in=-90](.3,1.3) node[above]{$\alpha_k$};
\draw (0,.5) to [out=90,in=-90] (0,1.3) node[above]{$\bd$};
}
\]
If $k={p+n+1}=p+1 \mod e$, then
\[\tikz[thick,xscale=2.5,yscale=1.5, baseline=25pt]{
\draw[wei] (-.3,0) node[below]{$\omega_j$} to (-.3,1.3) ;
\draw (0,0) node[below]{$\bd$} to [out=90,in=-90](.3,.5) ;
\draw (.6,0) node[below]{$\alpha_k$} to [out=90,in=-90] (.3,.5) ;
\draw (.3,.5) -- (.3,.8) node[right, midway]{$\bd+\alpha_j$};
\draw (.3,.8) to [out=90,in=-90](.6,1.3) node[above]{$\alpha_k$};
\draw (.3,.8) to [out=90,in=-90] (0,1.3) node[above]{$\bd$};
}
\quad =\quad
\tikz[thick,xscale=2.5,yscale=1.5, baseline=25pt]{
\draw[wei] (-.3,0) node[below]{$\omega_j$} to (-.3,1.3) ;
\draw (0,0) node[below]{$\bd$} to [out=90,in=-90](0,.5) ;
\draw (.3,0) node[below]{$\alpha_k$} to [out=90,in=-90](.3,1.3)
node[above]{$\alpha_k$};
\node[circle,inner sep=2.5pt,fill=black] at (.3,.65) {};
\draw (0,.5) to [out=90,in=-90] (0,1.3) node[above]{$\bd$};
}
\]
since the middle term of \eqref{polyb} vanishes (consider the idempotent after splitting of the two $\alpha_k$ and apply Lemma \ref{id-split}).
Hence our element factors through the join of all strands and hence lies in $A_{>\mathfrak{c}}^{\omega_j}$. The lemma follows.
\end{proof}

\begin{prop}\label{CST-span}
  The vectors $\CST$ span ${A}^\bnu$.
\end{prop}
\begin{proof}
  By Proposition \ref{red-basis}, we need only show that vectors of the
  form \[\gmu_1\overset{w_1^{-1};h_1}\Longrightarrow\gla\overset{w_2;h_2}\Longrightarrow\gmu_2\]
  can be written in terms of the $\CST$'s.  As usual, we induct on
  $\mathfrak{c}(\gla)$.

By Lemma \ref{multisplit}, if $\gla$ is not of the form $\gla_{\sS}$
for some semi-standard tableaux, then we can rewrite our vector to
factor through $\grave{\eta}$ which is higher in lexicographic order.
Thus, we need only to show that elements of the
form \[\gmu_1\overset{w_1^{-1};h_1}\Longrightarrow\gla_\sS\overset{w_2;h_2}\Longrightarrow\gmu_2\]
can be written in terms of the $\CST$'s (which are the special case
where $h_1=h_2=1$).  At the center of this diagram, the picture looks
precisely like that shown in the statement of Lemma \ref{multisplit},
except that some of the legs may not split all the way down to unit
vectors.  We assume that we add the action of $h_1$ at the
bottom of that portion of the diagram, which is after applying all
crossings coming from $w_1$.  By
Lemma \ref{multisplit}, the central portion of the diagram lies in
$A_{>\mathfrak{c}(\gla)}$, and so by induction, this element can be
written as a linear combination of $\CST$'s.  By induction, the result follows.
\end{proof}

Given a semi-standard tableau $\sS$, we can associate a standard tableaux $\sS^\circ$ of the same shape with the following properties:
\begin{itemize}
\item Any box contains labels from the same alphabets in $\sS$ and $\sS^\circ$.
\item Any pair of boxes with different entries in $\sS$ has entries in $\sS^\circ$ in the same order.
\item The row reading word of $\sS^\circ$ is maximal in Bruhat order amongst the standard tableau satisfying the first 2 conditions.
\end{itemize}
While the map $\sS\mapsto \sS^\circ$ is obviously not injective, it is
injective on the set of tableaux with a fixed type.

\excise{
Fix $\blah_1,\bmuh,\blah_2$ and classes $w_1\in S_{\blah_1} \backslash S_{\bd}/S_{\bmuh}, w_2\in S_{\bmuh} \backslash S_{\bd}/S_{\blah_2}$.  Then we obtain an element \[\psi_{w_1,w_2}=\blah_1\overset{w_1}\Longrightarrow\bmuh \overset{w_2}\Longrightarrow\blah_2\]
\begin{lemma} We have
\[\displaystyle \varphi\psi_{w_1,w_2}\varphi^*=\sum_{\blah_1'\in K_1,\blah_2'\in K_2} \blah_1'\overset{w_1^\circ}\Longrightarrow\bmuh \overset{1;t^{\bmuh}}\Longrightarrow\bmuh\overset{w_2^\circ}\Longrightarrow\blah_2' +c \pmod{{A}^\nu_{>\mathfrak{c}(\bmuh)}}\]
where $w_i^\circ$ is the longest coset rep of $w_i$ and $c$ is a sum of elements of the form $\blah_1'\overset{w_1';h_1}\Longrightarrow\bmuh \overset{w_2';h_2}\Longrightarrow\blah_2$ for $w_i'\leq w_i^\circ$ in Bruhat order.
\end{lemma}}

To $\gmu$ we associate a vector $\varphi_{\gmu}$ of ${A}^\bnu e_{\gmu}$ defined as follows: let $K$ be the set of  vector compositions whose
corresponding flags are complete refinements of that for $\gmu$.  For each $\gla\in K$, there is a diagram $\gmu\linj \gla$ which looks like a bunch of chicken feet where the top is labeled with the sequence $\blah$.

\begin{definition}
\label{chicken}
We let $\varphi_{\gmu}=\sum_{\gla\in K} \gmu\linj \gla$ and $\varphi=\sum_{\gmu} \varphi_{\gmu}$. We call them {\de chicken feet} vectors and their duals  $\varphi_{\gmu}^*$, $\varphi^*$ {\de pitchfork vectors}.
\end{definition}

Now, consider the map ${A}^\bnu\to T^\bnu$, $a\mapsto\varphi a\varphi^*.$
This map is obviously not injective, but it is on $e_\mathfrak{n}{A}^\bnu e_\mathfrak{m}$ for a fixed pair of compositions $\mathfrak{n},\mathfrak{m}$.

\begin{lemma}\label{CST-lin-ind} For all $\sS,\sT$, we have
  \begin{equation}\label{CST-circ}
\displaystyle  \varphi \CST\varphi^*=C_{\sS^\circ,\sT^\circ}+\sum_{\substack{\sS'<\sS^\circ \\ \sT' <\sT^\circ}} a_{\sS',\sT'}C_{\sS',\sT'}\pmod {{A}^\bnu_{>\mathfrak{c}(\grave\la_{\sS})}}.
\end{equation}

In particular, the elements $\CST$ are all linearly independent.
\end{lemma}
\begin{proof}
  First, we note that \[\varphi B_\sS=B_{\sS^\circ}+\sum_{\sS'<\sS}a_{\sS'}B_{\sS'}
  \pmod{{A}^\bnu_{>\mathfrak{c}(\gla_\sS)}}.\] The first term of the RHS
  comes from the term
  $\gmu_{\sS^\circ}\linj\gmu_{\sS}\cdot B_\sS=B_{\sS^\circ}$ of the
  product on the LHS; for
  any other pitchfork terms, we will either have a standard tableaux where
  $\ell(w_{\sS'})< \ell(w_{\sS^\circ})$ (by assumption), or a non-standard tableau, in
  which case the term lies in ${A}^\bnu_{>\mathfrak{c}(\gla_\sS)}$.
Thus, the equality follows.  If we have a non-trivial relation between
$\CST$'s, then (by multiplying by idempotents $e_{\mathfrak{n}}$ on the left and right) we
may assume that all tableaux which appear are of the same type.  Since
the $a\mapsto \varphi a\varphi^*$ is injective on such elements, we
have a non-trivial relation between the right hand sides of
\eqref{CST-circ}. This is impossible because of the
upper-triangularity in \eqref{CST-circ} and since the vectors
$C_{\sS^\circ,\sT^\circ}$ are linearly independent modulo the image of
${A}^\bnu_{>\mathfrak{c}}(\gla_\sS)$ by Theorem \ref{T-is-cellular}.
\end{proof}

\begin{proof}[Proof of Theorem \ref{A-is-cellular}]
We verify that $(\La,M,C,*,<, \op{deg})$ is a graded cell datum.
\begin{enumerate}[(C1)]
\item Clear.
\item This is the claim that the vectors $\CST$ are a basis.  They span by Lemma \ref{CST-span} and are linearly independent by Lemma \ref{CST-lin-ind}.
\item By definition $\CST^*=(B_\sS B_\sT^*)^*=B_\sT B_\sS^*=C_{\sT,\sS}.$
\item This is essentially identical to the proof of Theorem \ref{T-is-cellular}.  Consider $\sS$ of shape $\xi$.  Since the $\CST$ are a basis,  \[xB_{\sS}=\sum_{\sS',\sT} r_x(\sS,\sS',\sT)C_{\sS',\sT}\] for some coefficients $r_x(\sS,\sS',\sT)$.  Since $\sT$ must be a semi-standard tableau of type $\xi$, we must have that the shape of $\sT$ is above $\xi$ in dominance order, unless $\sT$ is the super-standard tableau.  So, \[xB_{\sS}=\sum_{\sS'} r_x(\sS,\sS')B_{\sS'}\pmod{{A}^\bnu(>\xi)}.\]
\item The degree function $\op{deg}$ clearly satisfies the required conditions.\qedhere
\end{enumerate}
\end{proof}

\section{Dipper-James-Mathas cellular ideals and graded Weyl modules}
\label{Section6}
Let again $\bnu$ be of the form
$\bnu=(\omega_{\charge_1},\dots,\omega_{\charge_\ell})$; recall
from the introduction that having for fixed $\zeta$, a primitive $n$th
root of unity in $\K$, we have an induced choice of parameters for the
cyclotomic Hecke algebra, given by
$q=\zeta, Q_i=\zeta^{\charge_I}$.

\begin{prop}
The chicken feet vector $\varphi_{\gmu}$ from Definition \ref{chicken} transforms according to the sign representation for the Young subgroup $S_{\mathfrak{c}(\gmu)}$.
\end{prop}
\begin{proof}
We verify the formulas from Proposition~\ref{sign-transform2} using Proposition~\ref{euler-int} or more precisely Remark \ref{KL}. The case $i_r=r_{r+1}$ is obvious, since it corresponds to the fact that ${\Delta_{r}}^2=0$, whereas the case $i_{r}=i_{r+1}+ 1$ corresponds to the last bullet point in Remark~\ref{KL}. The other cases are treated by the first bullet point. The claim follows.
\end{proof}

We can consider ${A}^{\bnu}\eT$ as a $T^\bnu$-representation (acting
from the right).  For each $\ell$-multi-composition $\xi$, we let
$\varphi_\xi=\sum_{\mathfrak{c}(\gmu)=\xi}\varphi_{\gmu}$, and let
$e_\xi=\sum_{\mathfrak{c}(\gmu)=\xi}e_{\gmu}$.  Then $\varphi_{\xi}$
is a vector in $e_{\xi}{A}^{\bnu}\eT$ which generates a sign
representation for $S_\xi$.  Thus, it induces a map $h_\xi\colon x_\xi
T^\bnu \to e_\xi A^\bnu\eT $ such that $x_\xi a\mapsto \varphi_\xi a$.

\begin{thm}\label{schur-isomorphism}
  The map $h_\xi$ is an isomorphism.  In particular, ${A}^\bnu\eT\cong
  \bigoplus_\xi  x_\xi T^\bnu$ as right $T^\bnu$-modules, and there is
  an isomorphism $\Phi^\bnu\colon {A}^\bnu\to \mathbf{S}^\bnu$.  On
  the cyclotomic $q$-Schur algebra of rank $n$, we obtain an induced
  isomorphism  \[\Phi^{\bnu}_n\colon\quad A^{\bnu}_{n}\to
\mathbf{S}(n;q,Q_1,\ldots, Q_\ell).\]
\end{thm}
\begin{proof}
The surjectivity of $h_\xi$ is clear; every element of $\eT {A}^\bnu e_{\xi}$ is of the form \[\sum a_{\gla}\cdot( \gmu\linj\gla)=\sum a_{\gla}e_{\gla}\varphi_{\gmu},\] where $\gla$ only contains compositions of type $e$ which correspond to simple roots $\al_i$, and $a_\gla\in T^\bnu {A}^\bnu \eT.$ Thus, the map is an isomorphism if and only if the spaces have the
same dimension.  Thanks to
Lemma \ref{T-per-dim}, the dimension of $ x_\xi T^\nu$ is the number of pairs of tableaux $(\sS,\sT)$
where $\sS$ is standard and $\sT$ is type $\xi$.  By Lemma
\ref{CST-lin-ind}, the elements $\CST$ for the same set of pairs are
linearly independent vectors in $\eT {A}^\bnu e_{\xi}$, so it must
have at least this dimension. Thus, $h_\xi$ is an isomorphism and we obtain $\Phi^\bnu\colon {A}^\bnu\to \mathbf{S}^\bnu$ by Proposition~\ref{Bencycliso} and the definition of $\mathbf{S}^\bnu$. The theorem follows.
%\todo{I did not believe the old proof at all with the surjectivity.}
%The
%surjective map $\Phi^{\bnu}_n\colon A^{\bnu}_{n}\to
%\mathbf{S}(n;q,Q_1,\ldots, Q_\ell)$ arises because we always
%have a surjective map $A\to \End_{eAe}(eA)$ for any algebra $A$ and
%idempotent $e\in A$, applied to $e_n\eT$.  Since the vectors $\CST$
%span $A^{\bnu}_{n}$, and there is the same number
%of them as the dimension of $\mathbf{S}(n;q,Q_1,\ldots, Q_\ell)$, this
%map must be an isomorphism.
%Summing over all $n$, we obtain that
%$\Phi^\bnu$ is an isomorphism.
\end{proof}

Theorem \ref{A-is-cellular} shows that $A^{\bnu}_{n}$ comes with a natural grading and hence Theorem \ref{schur-isomorphism} shows that $\Phi^{\bnu}_n\colon\quad A^{\bnu}_{n}$ is a graded version of $\mathbf{S}(n;q,Q_1,\ldots, Q_\ell).$ Following \cite{Strgrad}, we say that a $\mathbf{S}(n;q,Q_1,\ldots, Q_\ell)$-module $\overline{M}$ {\it has a graded lift} $M$ if there is a graded $A^{\bnu}_{n}$-module which is isomorphic to $\overline{M}$ after forgetting the grading. Any such $M$ is then a {\it graded lift} of $\overline{M}$. By \cite[Lemma 1.5]{Strgrad}, graded lifts of indecomposable modules are unique up to isomorphism and overall grading shifts.

\begin{thm}
\label{cellular}
Under the isomorphism \eqref{BKiso}, the cellular structure from
Theorem \ref{A-is-cellular} is intertwined with the
Dipper-James-Mathas cellular structure \cite[Def.\ 6.7]{DJM}  on the
cyclotomic q-Schur algebra in that sense that:
\begin{enumerate}
\item The cellular ideals coincide when the order in the Dipper-James-Mathas structure
  is weakened to the  lexicographic
  ordering on multi-partitions.
\item The cell modules $W^\xi$ are graded lifts of the Weyl modules of $\mathbf{S}^\bnu$, and the
  $F^\xi=W^\xi/\op{rad} W^\xi$ form (up to overall grading shift) a complete, irredundant set of
  graded lifts of simple modules.
\item Any other graded lift of a Weyl module differs (up to
  isomorphism) only by an overall shift in the grading.
\end{enumerate}
\end{thm}
\begin{proof}
  The cellular ideals of either our basis or the Dipper-James-Mathas
  basis can be defined in terms of maps between permutation modules
  factoring through those greater in lexicographic order.  That is,
  identifying $e_\xi\mathbf{S}^\bnu e_{\xi'}$ with
  $\Hom_{\mathbf{S}^\bnu}(\mathbf{S}^\bnu e_{\xi}, \mathbf{S}^\bnu e_{\xi'})$, the
  intersection with $e_\xi\mathbf{S}^\bnu (>\vartheta)e_{\xi'}$ is
  just the maps factoring through $\mathbf{S}^\bnu e_{\vartheta'}$
  with $\vartheta' <\vartheta$ in the lexicographic order.  It is clear that the same definition works for our cellular
  structure. The second statement is then clear and the third follows
  by standard arguments (e.g. \cite[Lemma 1.5]{Strgrad}), since the
  Weyl modules are the standard modules of a highest weight structure,
  and thus indecomposable.
\end{proof}

\begin{remark}
\label{Schurlevel1}
{\rm
In the case where $\nu$ is itself a fundamental weight,
$\mathbf{S}^\nu$ is the sum of the usual $q$-Schur algebras for all
different ranks.   Except for some
degenerate cases, a grading on this algebra has already been defined
by Ariki, \cite{Ariki}.}
\end{remark}

\begin{theorem}\label{same-Ariki}
  The grading defined on $\mathbf{S}^\nu$ via the isomorphism
  $\Phi^\nu$ agrees with Ariki's, \cite{Ariki}, up to graded Morita equivalence.
\end{theorem}
\begin{proof}
   Ariki's grading is defined uniquely (up to Morita equivalence) by
   the fact that there is a graded version of the Schur functor
   compatible with the Brundan-Kleshchev grading on $\mathfrak{H}^\nu$.  The
   image of an indecomposable projective of the $q$-Schur algebra
   under the Schur functor is a graded lift of an {\it indecomposable}
   summand of a permutation module; hence a graded lift is unique up to shift by \cite[Lemma 1.5]{Strgrad}.  This shows
   that such a grading is unique, and both Ariki's and our gradings
   satisfy this condition.
\end{proof}

\begin{remark}
\label{gldim}
{\rm
From our identification of the cyclotomic quiver Schur algebras  with the  cyclotomic Schur algebras,  Theorem \ref{schur-isomorphism},  it follows that  ${A}^\bnu$ and ${A}^\bnu_{\bd}$ have finite global dimension.  The indecomposable projective modules form (in the nongraded and also in the graded version) a $\mZ$-basis of the Grothendieck group. We will describe the underlying combinatorics in the next section.
}
\end{remark}

\section{Graded multiplicities and $q$-Fock space}
\label{sec:q-fock-space}
As promised in the introduction, we now draw the connection between
the algebras ${A}^\bnu$ and the theory of higher representation theory as in the
work of Rouquier \cite{Rou2KM} and Khovanov-Lauda \cite{KLIII}. We start with the combinatorics of Fock space, obtain the main result on decomposition numbers and finish with a categorification result.

\subsection{Induction and restriction functors}
\label{sec:categorical-action}
We still assume that $\bnu=(\omega_{\charge_1},\dots,\omega_{\charge_\ell})$ and also $e>2$ (although all of our results can be extended to an arbitrary sequence of weights by replacing the Fock space by any highest weight representation of $U_q(\glehat)$).
\\

As before, we identify the dimension vector $\bd$ with the element
$\sum d_i\al_i$ of the root lattice of $\slehat$;  we let $\langle -,-
\rangle$ denote the usual pairing between the root and weight lattice
of $\slehat$, in particular $\langle \omega_i,\bd\rangle=\bd_i.$ In terms of the bilinear Euler form $\{\bd',\bd''\}=
\sum_{i=1}^e\bd'_i(\bd_{i}''-\bd_{i+1}'')$ we have that \[\langle \bd',\bd''\rangle
=\{\bd',\bd''\}+ \{\bd'',\bd'\}.\]
Consider the inclusion of algebras $\gamma_\bd\colon {A}^\bnu_{\bc}\,\to
{A}^\bnu_{\bc+\bd}$ defined by
$a\mapsto a|e_\bd$ and $A^\bnu_{\bc|\bd}=\gamma_{\bd}A^\bnu_{\bc+\bd}\gamma_{\bd}$ with the appropriate idempotent $\gamma_{\bd}$ (the sum of all $e_{\grave\mu}$ where the last
  multi-composition ends with $\bd$). The most interesting case is when $\bd=\alpha_i$ for which we use the abbreviation
$\gamma_i=\gamma_\bd$ and denote the image of the inclusion also by $A^\bnu_{\bc|i}$.

\begin{definition}
Define,  for any $\bc$, the graded
  $\bd$-induction and $\bd$-restriction functors 
\begin{eqnarray*}
  \fF_\bd\colon\quad{A}^\bnu_{\bc}\mgmod\to {A}^\bnu_{\bc+\bd}\mgmod,&&  M\mapsto {\fF_\bd}M= {{A}^\bnu_{\bc+\bd}}\gamma_{\bd} \otimes_{A^\bnu_{\bc}}M, \\
 \fE_\bd\colon\quad{A}^\bnu_{\bc+\bd}\mgmod\to {A}^\bnu_{\bc}\mgmod,&&N\mapsto\fE_\bd N=
\Hom_{A^\bnu_{\bc+{\bf d}}}(\gamma_{\bd}{A}^\bnu_{\bc +\bd}, N)\langle s(\bd,\bc)\rangle.
\end{eqnarray*}
where $\langle_-\rangle$ denotes the grading shift as in Section~\ref{sec:Hall} and $s(\bd,\bc)=\{\bd,\bd\}+\{\bc,\bd\}+\{\bd,\bc\}-\sum_{i=1}^\ell\bd_{\charge_ i}$. 
\end{definition}

\begin{lemma}
\label{keep proj}
  The functors $\fF_\bd$ send projective objects to projective
  objects, and the
  functors $\fE_\bd$
  are exact.
\end{lemma}
\begin{proof}
  The vector space underlying $\fE_\bd M$ is $\gamma_{\bd}N$, the image of an
  idempotent  acting on $M$, so  $\fE_\bd$ is
  exact.  The left adjoint of an exact functor always sends
  projectives to projectives, see also the proof of Lemma \ref{pairingfunc}.
\end{proof}
Note that viewing ${A}^\bnu\mmod$ as the representation category of ${A}^\bnu$,
these functors are induced by the monoidal action of ${A}$
described earlier and its adjoints. 
In the special case $\bd=\alpha_i$ we obtain after forgetting the grading via the isomorphism,  Theorem \ref{schur-isomorphism},  \[\Phi^{\bnu}_n\colon\quad A^{\bnu}_{n}\to
\mathbf{S}(n;q,Q_1,\ldots, Q_\ell)\] the ordinary  $i$-induction and $i$-restriction functors for the cyclotomic $q$-Schur algebras,
\begin{eqnarray*}
  \overline\fF_i&:& \mathbf{S}(n;q,Q_1,\ldots, Q_\ell) \mmod\to \mathbf{S}(n+1;q,Q_1,\ldots, Q_\ell)\mmod, \\
  \overline\fE_i&:&\mathbf{S}(n+1;q,Q_1,\ldots, Q_\ell)\mmod\to \mathbf{S}(n;q,Q_1,\ldots, Q_\ell)\mmod.
\end{eqnarray*}
A detailed study of the ordinary $i$-induction and
$i$-restriction functors can be found in \cite[\S
5]{Wada}. The functors $\overline\fF_i$ and $\overline\fE_i$ are
biadjoint, hence in particular exact, \cite[Theorem 4.14]{Wada};
however, when the grading is taken into account the left and right
adjoints of $\overline\fF_i$ will differ by a shift in the grading. For general $\bd$, the functors $(\overline\fF_\bd,\overline\fE_\bd)$ form an adjoint pair, but are not biadjoint.

\excise{
\begin{theorem}  The isomorphism $\Phi^\bnu$ from Theorem \ref{schur-isomorphism} intertwines the functors $\overline\fF_i$ and $\overline\fE_i$ with the
  usual functors of $i$-induction and $i$-restriction on modules over
  the cyclotomic $q$-Schur algebra.
\end{theorem}
\begin{proof}

We have a natural inclusion of algebras $A^\bnu_n\to A^\bnu_{n+1}$ given by the map
 $\sum_i\gamma_i$, which adds a new strand of all possible labelings
 by simple roots.
 Restricted to $R^\nu_n\to R^\nu_{n+1}$, this map is intertwined by
 Brundan and Kleshchev's isomorphism with  the
 usual inclusion of affine Hecke algebra of $S_n$ into that of  $S_{n+1}$.

 Thus,  $\Phi^\bnu$ intertwines the map
 $\sum_i\gamma_i$  with the induced  inclusion $\mathbf{S}^\bnu_d\to
 \mathbf{S}^\bnu_{d+1}$.  Furthermore, the
 decomposition into the summands $\gamma_i$'s matches under the
 isomorphism $\Phi^\bnu$ with the decomposition of this
 functor according to eigenvalues of the deformed Jucys-Murphys
 element under the isomorphism \eqref{BKiso}; see \cite[(4.21)]{BKKL}.
\end{proof}
}

\subsection{Combinatorics of higher Fock space }
\label{sec:decategorification}

We now introduce the combinatorics which will control graded versions of induction and restriction functors. 

Fix  $\charge\in \mZ$ and let  $V=\mZ[q,q^{-1}]^\mZ$ with basis $u_i, i\in\mZ$. For a
partition $\la$ we denote by $\la'$ its transposed partition. 

%\todo{Check if we can define everything over Laurentpolys. I would prefer to have $\mC(q)$}

\begin{definition}
  The {\bf level 1 quantized Fock space} $\Fock_1(\charge)$ of charge
  $\charge$ is the $\mZ[q,q^{-1}]$-module freely generated by a symbol
  $\la$ for each partition $\la$.
\end{definition}

Recall the realization of Fock space in terms of the free $\mC[q,q^{-1}]$-module  $\wedge^\frac{\infty}{2}V$ of semi-infinite wedges in $V$ on basis $u_{i_1}\wedge u_{i_2}\wedge\cdots$ where the indices $i_k\in\mZ$ form an increasing sequence $i_1<i_2<\dots$ such that ${i_k}=\charge+k-1$ for $k\gg 1$:

\begin{lemma} Fix a charge $\charge$. There is an isomorphism of $\mZ[q,q^{-1}]$-modules
\begin{eqnarray}
\label{Fockiso}
\Fock_1(\charge)&\cong& \wedge^\frac{\infty}{2}V\\
  \la&\mapsto &u_\la:=u_{\charge-\la_1'}\wedge u_{\charge-\la_2'-1}\wedge
  u_{\charge-\la_3'-2}\wedge \cdots.\nonumber
  \end{eqnarray}
Under this identification, there is an addable box in column $k$ of $\la$ iff $i_k>i_{k-1}+1$. In this case the index $i_k$ taken modulo e is the residue of this unique addable box.
\end{lemma}

\begin{proof}
This follows directly from the definitions.
\end{proof}
Th space \eqref{Fockiso} carries an action of $U_q(\slehat)$ defined by Hayashi
\cite{Hayashi}. The explicit formulas crucially depend on the identification \eqref{Fockiso} and the choice of coproduct on the Hall algebra discussed in Section \ref{sec:Hall}. We chose a slightly unusual identification \eqref{Fockiso}, but better suitable for our purpose than for instance \cite{VV} (where the partition gets not transposed and the indices form a decreasing sequence). We choose the coproduct 
\begin{equation*}
\Delta(f_i)= 1\otimes f_i+f_i\otimes k_i,\quad \Delta(e_i)=k_i^{-1}\otimes e_i+e_i\otimes 1,\quad \Delta(k_i)=k_i\otimes k_i.
\end{equation*}
for $U_q(\slehat)$. It extends by \cite[(10)]{VVDuke} and the Drinfeld double construction, see e.g. \cite{Xiao}, to a coproduct of $U_q(\glehat)$ by setting on the standard generators, Section \ref{sec:Hall}, 
\begin{eqnarray}
\label{coproduct}
&\Delta(\bbf_\bd)= \displaystyle{\sum_{\bd=\bd'+\bd''}}q^{-\{\bd',\bd''\}}
\bbf_{\bd''}\otimes\bbf_{\bd'}\mathbf{k}_{\bd''},\quad \Delta(\be_\bd)= \displaystyle{\sum_{\bd=\bd'+\bd''}}q^{\{\bd'',\bd'\}}
\mathbf{k}_{\bd'}^{-1}\be_{\bd''}\otimes\be_{\bd'},&\nonumber\\
&\Delta(\bk_\bd)=\Delta(\bk_\bd)\otimes\Delta(\bk_\bd),&
\end{eqnarray}
where $U_q(\glehat)$ is the Drinfeld double of $U_e^-$. By
\cite[6.2]{VVDuke}, the Hayashi action extends to an action of $U_q(\glehat)$, such that acting with $\bbf_\bd$  on a partition $\la$ produces a  $\mZ[q,q^{-1}]$-linear combination $\bbf_\bd\la=\sum_\mu q^{m(\mu/\la)_-}\mu$ of all partitions $\mu$ obtained from $\la$ by adding $|\bd|$ boxes {\it but at most one per column} and exactly $d_i$ of residue $i$ for $i\in\mV$. We write $\op{res}(\mu/\la)=\bd$ for all such partitions $\mu$. Similarly $\be_{\bd}$ acts by removing $|\bd|$ boxes, again at most one per column, and exactly $d_i$ of residue $i$ for $i\in\mV$. To describe the coefficients $q^{m(\mu/\la)_-}$ assume $\op{res}(\mu/\la)=\bd$ and let $m_j=1$ if there was a box added in column $j$ to obtain $\mu$ from $\la$ and set $m_j=0$ otherwise for all $j\geq 1$. Using the identification $\eqref{Fockiso}$ for $\la$ define
\begin{eqnarray*}
m(\mu/\la)_-&=&\sum_{i_s<i_t}m_t(1-m_s)(\delta_{\overline{i_s},\overline{i_t}}
-\delta_{\overline{i_s+1},\overline{i_t}}).
\end{eqnarray*}
Here $t$ denotes a column where a box $b$ was added and then
$\delta_{\overline{i_s},\overline{i_t}}--\delta_{\overline{i_s+1},\overline{i_t}}$ counts the number of addable minus the number of removable boxes below $b$ with the same residue as $b$, excluding the blocked columns (see Definition \ref{defdeg}) thanks to the factor $(1-m_s)$. Similarly, we have the `reversed' statistics involving boxes above
\begin{eqnarray*}
m(\mu/\la)_+&=&\sum_{i_s>i_t}m_t(1-m_s)(\delta_{\overline{i_s},\overline{i_t}}
-\delta_{\overline{i_s+1},\overline{i_t}}).
\end{eqnarray*}
\begin{lemma}
\label{eq:hayashi}
There is an $U_q(\glehat)$-action on Fock space $\Fock_1(\charge)$ given by
\begin{eqnarray*}
&\bbf_\bd\la=\sum_{\op{res}(\mu/\la)=\bd}q^{m(\mu/\la)_-}\mu,\quad\quad
\be_{\bd}\la=\sum_{\op{res}(\la/\mu)=\bd}q^{m(\mu/\la)_+}u_\la,&\\
&\bk_{\bd}\la=q^{\langle \bd,\op{wt}(\la)\rangle}\la.&
\end{eqnarray*}
Here $\langle \bd,\op{wt}(\la)\rangle=\sum_{i=1}^e d_i(\op{add}_i-\op{rem}_i)$, where $\op{add}_i$ respectively $\op{rem}_i$ denotes the number of addable respectively removable boxes in $\la$ of residue $i$.
\end{lemma}
\begin{proof}
This follows by straightforward calculations from our definitions, see \cite{VVDuke} for details (remembering that we work with the transposed partitions).
\end{proof}

 \begin{remark}\mbox{}
{\rm We should note that conventions vary from author to author. Our conventions are transposed to \cite{VV} and also \cite[2.1]{Uglov} and
    designed so as to match the projectives of our
    categorification with the canonical basis of Fock space; different
    conventions can match it instead with the tilting modules or other
    collections of modules.
}
 \end{remark}

\begin{definition}[Higher Level Fock space]
Let $\{\charge_1,\dots,\charge_\ell\}$ be a fixed $\ell$-tuple of charges. Then $\Fock_\ell$ denotes the $\ell$-fold level $1$ fermionic Fock space
  $$\Fock_\ell=\Fock_1(\charge_1)\otimes \cdots \otimes \Fock_1(\charge_\ell),$$ with charges $\charge_1,\dots,\charge_\ell$
  equipped with the basis
  $u_{\xi}=u_{\xi^{(1)}}\otimes \cdots \otimes u_{\xi^{(\ell)}}$ where $\xi$ ranges over all $\ell$-multi-partitions. It comes with the
  usual $\C[q,q^{-1}]$-bilinear inner product $(-,-)$ where the basis
  $u_\xi$ is orthonormal. The elements $u_{\xi}$ are called {\de
    standard basis vectors}.
\end{definition}

\begin{remark}{\rm
In the bosonic realization of Fock space, \cite{KacBombay}, our standard basis is just
  the products of Schur functions $u_\xi=s_{\xi^{(1)}}s_{\xi^{(2)}}\cdots
  s_{\xi^{(\ell)}}$ in $\ell$ different alphabets. As we pointed out already in the introduction, we are {\it not} considering the higher level
Fock space studied by Uglov \cite{Uglov}, but the more naive tensor
product of level 1 Fock spaces.  This distinction is discussed
extensively in \cite[\S 3]{BKgd}, see also Theorem~\ref{twoFocks}.
}
\end{remark}

We can easily generalize our description of the Hayashi action to this
tensor product.  For boxes $(i,j,k)$ and $(i',j',k')$ in the diagram
of a multi-partition, we say that $(i',j',k')$ is {\bf below} $(i,j,k)$ if
$k'>k$ or $k'=k$ and $i'<i$. The careful reader should note that this
doesn't match some writers' conventions (for
example, from \cite{HM}; our convention will match theirs if one
indexes partitions in opposite order). Then, we can define $m(\mu/\la)_-$ for multi- partitions exactly as before, namely by going through all added boxes $b$ and count the
number of addable minus the number of removable boxes below $b$ with the same residue
as $b$, excluding the blocked columns. The definitions imply
\begin{lemma} The action is given by
\begin{eqnarray}
\label{actionFockl}
\bbf_{\bd}u_\xi=\sum_{\op{res}(\eta/\xi)=\bd}q^{m(\eta/\xi)_+}u_\eta,&& \be_{\bd}u_\eta=\sum_{\op{res}(\eta/\xi)=\bd}q^{-m(\eta/\xi)_-}u_\xi
  \end{eqnarray}
  where $\eta$ and $\xi$ ranges over pairs of $\ell$-multi-partitions such that $\eta/\xi$
  has no two boxes in the same column.
  \end{lemma}

\begin{ex}
For instance, if $e=3$, $\ell=2$ and $\charge=(0,0)$, then we have the following (where we abuse notation and write the residues into the boxes)
\begin{eqnarray*}
e_0f_0.\left(\young(012,2), \emptyset\right)&=&e_0.\left(q\cdot q\left(\young(0120,2), \emptyset\right)+1\cdot q\left(\young(012,20), \emptyset\right)+1\left(\young(012,2),\young(0)\right)\right).\\
&=&(q^2\cdot1+q\cdot q^{-1}+1\cdot q^{-2})\left(\young(012,2), \emptyset\right).
\end{eqnarray*}
where the numbers indicate the residues. Hence $e_of_0-f_0e_0$ acts by multiplication with $(q^2\cdot1+q\cdot q^{-1}+1\cdot q^{-2})=\frac{(q^3-q^{-3})}{(q-q^{-1})}$ which agrees with the action of $\frac{k_i-k_i^{-1}}{q-q^{-1}}$.
\begin{eqnarray*}
f_{(1,0,1)}.\left(\young(012,2), \emptyset\right)&=&
q^2\left(\young(0120,2,1), \emptyset\right)+q\left(\young(012,20,1), \emptyset\right)+\left(\young(012,2,1), \young(0)\right)
\end{eqnarray*}
(Note that the two boxes can't be put both in the first column.)
\end{ex}

\begin{lemma} We have the following adjunction formula for the action on $\Fock_\ell$:
\begin{eqnarray} 
\label{adjointcomb}
(\bbf_{\bd}u_\xi,u_\eta)&=&(u_\xi, q^{\{\bd,\bd\}}\be_{\bd}\mathbf{k}_{\bd}u_\eta)
\end{eqnarray}
\end{lemma}
\begin{proof}
This follows directly from \eqref{actionFockl} and the formula 
  \[m(\eta/\xi)_+ +m(\eta/\xi)_-=\{\bd,\bd\}+\langle
  \bd,\operatorname{wt}(u_\eta)\rangle,\] where $\langle
  -,-\rangle$ denotes the usual pairing between roots and weights of
  $\slehat$ as in Lemma~\ref{eq:hayashi} summed over all occurring partitions, since $\mathbf{k}_{\bd}u_\eta=q^{\langle
  \bd,\operatorname{wt}(u_\eta)\rangle}$.
\end{proof}

\subsection{A weak categorification}
\label{sec:weak-categ}

Let $K^0_q({\mathbf A}^{\bnu})$ be the split Grothendieck group of the graded category of graded projective
modules over ${A}^{\bnu}$; since the algebra $A^\bnu$ has finite
global dimension by Remark \ref{gldim}, this is canonically isomorphic to the Grothendieck group
of all $A^\bnu$-modules, not just the projectives. By Theorem \ref{cellular}(2), the Grothendieck group of  $A^\bnu\mmod$
has a basis over $\mZ[q,q^{-1}]$ given by the classes of the cell
modules ${W}^\xi$ with the chosen standard lift in the grading such that the head is in degree zero. The action of
grading shift induces a $\mZ[q,q^{-1}]$-module structure on $K^0_q({\mathbf A}^{\bnu})$ where $q$ shifts the
grading up by $1$.   We'll be interested in the  isomorphism
\begin{eqnarray}
\label{isowithFock}
\psi\colon\quad
\mathbb{Z}K^0_q({\mathbf
  A}^{\bnu})\overset{\sim}\longrightarrow \Fock_\ell,&\text{ defined by }& [W^\xi]\mapsto u_\xi.
  \end{eqnarray}
\begin{remark}
\label{rem:weight}
Note that under the isomorphism $\psi$, isomorphism classes from ${A}^\bnu_{\bc}\mgmod$ get sent to elements of weight $\op{wt}(\bc)=\sum_{i=1}^\ell\la_i-\sum_{i=1}^e c_i\alpha_i$. 
\end{remark}
We define the following grading shifting functors for any $\bc, \bd$: 
\begin{eqnarray}
\fK_{\bd}=\langle \bd, \op{wt}(c)\rangle
: &&{A}^\bnu_{\bc}\mgmod\to {A}^\bnu_{\bc}\mgmod.
\end{eqnarray}
By the above remark they are well-defined and their induced action on the Grothendieck group agrees with $\bk_\bd$ via $\psi$.
Moreover, we have the following:
\begin{thm}\label{fock-space}
The functors $\fF_\bd$ and $\fE_\bd$ induce an action of
$U_q(\glehat)$ on $\mC(q)\otimes_{\mZ[q,q^{-1}]}K^0_q({\mathbf A}^{\bnu})$;
the map $\psi$ defines an isomorphism of representations $K^0_q({\mathbf
  A}^{\bnu})\cong \Fock_\ell$.
\end{thm}
Before giving the proof of this theorem, we need some preparation.

\begin{definition}
\label{Defh}
Given a residue datum $\bmuh$, define $\bbf_{\bmuh}=\bbf_{\mu^{(r)}}\cdots
\bbf_{\mu^{(1)}}\in U^-_e$ and inductively the vectors $h_{\gmu}\in\Fock_\ell$  as
\[
h_{\gmu}=\bbf_{\bmuh(\ell)}\left(u_\emptyset\otimes h_{(\bmuh(\ell-1),\dots,\bmuh(1))}\right).\]
{\rm(}The $j$th factor acts here on elements of the $j$th-fold level $1$ Fock space $\mathbb{F}_j$.{\rm )}
\end{definition}
\begin{prop}
\label{psi}
  We have that \[\displaystyle h_{\gmu}=\sum_{\substack{\op{sh}(\sS)=\xi\\
        \op{type}(\sS)=\gmu}}q^{\op{Deg}(\sS)}u_\xi, \qquad \text{ and }  \qquad  [A^\bnu_{\gmu}]=\sum_{\substack{\op{sh}(\sS)=\xi\\
        \op{type}(\sS)=\gmu}}q^{\op{deg}(\sS)}[W^\xi].\]
In particular, it follows from \eqref{isowithFock} and Proposition \ref{degrees} that $\displaystyle\psi([A^\bnu e_{\gmu}])=h_{\gmu}$.
\end{prop}
\begin{proof}

We start with the first displayed equality and prove this by induction on $\ell$. For $\ell=1$ this is clear from the definitions and Lemma \ref{eq:hayashi}. We fix now $\ell>1$ and assume the result holds for shorter $\bmuh$. In particular we only need to consider the case where $\bmuh^{(\ell)}\not=\emptyset$ and can assume the claim to hold for $\grave\nu$ obtained from $\grave\mu$ with the last part  $\bmuh^{(\ell)}$ of $\bmuh$ removed. Let this part be $\bd$. Then
  $h_{\gmu}=\bbf_\bd h_{\grave \nu}$ and
 \[\displaystyle h_{\grave\nu}=\sum_{\substack{\op{sh}(\sS)=\xi\\
     \op{type}(\sS)=\grave \nu}}q^{-\op{deg}(\sS)}u_\xi\]
Thus, we have that
\begin{eqnarray*}
  \bbf_\bd h_{\grave \nu}&=&\sum_{\substack{\op{sh}(\sS)=\xi\\
     \op{type}(\sS)=\grave \nu}}q^{\op{deg}(\sS)}\bbf_\bd u_\xi
=\sum_{\substack{\op{sh}(\sS)=\xi\\ \op{type}(\sS)=\gmu}}q^{\op{deg}(\sS)}u_\xi.
\end{eqnarray*}
For the last equality we used Lemma \ref{actionFockl} and the definition of the comultiplication \eqref{coproduct} in comparison with the Definition \ref{defdeg} of the combinatorial degree $\op{Deg}$. This completes the proof of the first formula.

For the second equality, we apply the result \cite[2.14]{HM}; the projective
module $A^\bnu e_{\gmu}$ has a filtration by $W_\xi$ with multiplicity
spaces $\dot{W}_\xi\otimes_{A^\bnu}A^\bnu e_{\gmu}\cong
e_{\gmu}W_\xi$ and hence we are done by Theorem
\ref{A-is-cellular}: the space $e_{\gmu}W_\xi$ has a homogeneous basis
indexed by semi-standard tableaux of type $\gmu$ and shape $\xi$.
\end{proof}

 The Grothendieck group $K^0_q({\mathbf
  A}^{\bnu})$ is endowed with a $q$-bilinear pairing defined by
\begin{eqnarray}
\label{pairing}
\lsy [P],[P']\rsy&=&\dim_q \left({\dot P}\otimes_{{A}^{\bnu}} P'\right).
\end{eqnarray}
Here $\dot{P}$ is $P$ considered as a right module using the $*$-antiautomorphism and $\dim_q M=\sum \dim M_iq^i$ denotes the graded Poincare polynomial
for any finite dimensional $\mZ$-graded vector space $M=\oplus_{i\in Z}M_i$.
Then the pairing \eqref{pairing} extends
to objects $M$ and $N$ by taking the derived tensor product $\lsy
[M],[N]\rsy =\dim_q \left({\dot M}\otimes^\mathbb{L}_{{A}^{\bnu}}
  N\right)$.
Since $\{u_\xi\}$ and $\{[W^\xi]\}$ are orthonormal bases for the
natural $q$-bilinear forms on the two spaces, the map $\psi$
intertwines the inner product of $\Fock_\ell$ and the form $\lsy
-,-\rsy$ on $K^0_q({\mathbf
  A}^{\bnu})$.

Let $\Bf_{\bd}=[\cL\Bf_{\bd}]$ and  $\Be_{\bd}=[\cL\fE_{\bd}]=[\fE_{\bd}]$ be the endomorphisms on the Grothendieck groups induced by the derived functors $\cL\fF_{\bd}$ and $\cR\fE_{\bd}$, and $\Bk_{\bd}$ the endomorphism induced by $\fK_{\bd}$ (which is just multiplication with a certain power of $q$). We have then the following categorification of the adjunction formula  \eqref{adjointcomb}

\begin{lemma}
\label{pairingfunc}
With the pairing from \eqref{pairing} we have 
\begin{eqnarray}
\label{oje}
\left(\Bf_{\bd}[M],[N]\right)&=&q^{\{\bd,\bd\}} \left([M],\Be_\bd \Bk_{\bd}[N]\right).
\end{eqnarray}
\end{lemma}

\begin{proof} 
By Remark \ref{gldim} it is enough to check the formula on projectives, where all the functors are exact. 
Abbreviate $B=A^\bnu_{\bc+\bd}$ and $A=A^\bnu_{\bc}=\gamma_{\bd}B\gamma_{\bd}$ and and let $M=Ae_{\gmu}$ and $N=Be_\gla$.
Then the definitions imply
\begin{eqnarray*}
(\Bf_{\bd}[M],[N])&=&\dim_q ((\dot{\fF_\bd M})\otimes_{B} N)
=\dim_q \dot{\left(B\gamma_{\bd}\otimes_A  M\right)}\otimes_B Be_\gla\\
&=&\dim_q (e_{\gmu}A\otimes_A\gamma_{\bd}B\otimes_B) Be_\gla
=\dim_q (e_{\gmu|\bd}B e_\gla).\\
([M],\Be_\bd [N])&=&\dim_q(\dot{M}\otimes_{A} \fE_\bd N)
=q^{s(\bc,\bd)}\dim_q ({e_{\gmu}A} \otimes_A \gamma_{\bd}Be_\gla\\
&=&q^{s(\bc,\bd)}\dim_q (e_{\gmu|\bd}B e_\gla).
\end{eqnarray*}
Then the right hand side of \eqref{oje} equals $q^c$ times the left hand side, where $c={\{\bd,\bd\}}+{s(\bc,\bd)}+\sum_{i=1}^\ell\bd_{\charge_ i}-\langle\bd,\bc+\bd\rangle=0$.
\end{proof}

\begin{proof}[Proof of Theorem \ref{fock-space}]
From Proposition \ref{psi} we have $\bbf_\bd \psi ([A^\bnu e_\gmu])=\psi([A^\bnu
e_{\gmu|\bd}])$. On the other hand  $\Bf_\bd([A^\bnu e_\gmu])=[A^\bnu
e_{\gmu|\bd}]$ by definition of $\fF_{\bd}$. Hence $\psi$  intertwines  $\Bf_\bd$ with $\bbf_{\bd}$ by Remark \ref{gldim} and as we know already $\fK_\bd$ with $\bk_{\bd}$. Then it also intertwines $\Be_\bd$ with $\be_{\bd}$ by Lemma \ref{pairingfunc}, since it intertwines the two nondegenerate bilinear forms.
\end{proof}

\subsection{The involution induced by Serre twisted duality}
For any left ${A}^\bnu$-module $M$, the space
$H=\Hom_{{A}^\bnu}(M,{A}^\bnu)$ is naturally a right ${A}^\bnu$-module
via the action $(f\cdot a)(m)=f(m)\cdot a$.  Using the anti-involution
$*$, this can be turned into a
left ${A}^\bnu$-module via $(f\cdot a)(m)=f(m)\cdot a^*$ for $f\in H, m\in M, a\in{A}^\bnu$. This extends to the derived
functor \[\mD=\RHom_{{A}^\bnu}(-,{A}^\bnu)\colon \quad D^b({A}^\bnu\mmod)\to D^b({A}^\bnu\mmod).\]
We refer to this functor as {\bf Serre-twisted duality}.  

This name can be explained by the following alternate description of the same functor. Let $\star: {A}^\bnu\mmod \to {A}^\bnu\mmod$ be the duality functor given
by taking vector space dual, and then twisting the action of ${A}^\bnu$
by the anti-automorphism $*$.
Let $\Serre:D^b({A}^\bnu\mmod)\to D^b({A}^\bnu\mmod)$ be the graded version of the Serre functor. Since
${A}^\bnu_\bd$ is finite-dimensional and has finite global dimension (Remark \ref{gldim}), 
this is simply the derived functor of tensor product with the bimodule
$({A}^\bnu_\bd)^*$; thus, $\star\circ \Serre=\mD$, see \cite[pg. 37]{Happel} and \cite{MSSerre}.\\

The following gives a natural construction of a bar-involution on
$\Fock_\ell$ which we will show coincides with the construction in
\cite{BKgd}.
Let $g\mapsto \overline{g}$ be the unique $q$-antilinear automorphism of $U_e^-$ which fixes the standard generators $\bbf_\bd$.
\begin{thm}\label{psi-properties}
  For any $\ell$ and any fixed charge let $\Psi:\Fock_\ell\to \Fock_\ell$ be the map induced by $\mD$ on the Grothendieck
  group. These maps satisfy the compatibility properties
 \begin{enumerate}[({B}1)]
  \item $\Psi(g\cdot v)=\bar g\cdot \Psi(v)$ for $g\in U^-_e$.
  \item $\Psi(v\otimes u_\emptyset)=\Psi(v)\otimes u_\emptyset$.
  \item $\Psi(u_\xi)=u_\xi+\sum_{\eta<\xi}a_{\eta,\xi}u_\eta$ for
    some $a_{\eta,\xi}\in \mZ[q,q^{-1}]$
  \end{enumerate}
and are uniquely characterized by the first two properties. Moreover, the vectors $h_\gmu$ are invariant under this involution.
\end{thm}
\begin{proof}
By Proposition \ref{psi}, the vectors $h_\gmu$ correspond to (the standard lifts of) projective modules which are obviously fixed by $\mD$, hence the last statement follows. 
  \excise{First, we note that $h_\gmu$ are invariant since for any algebra
  $A$ with fixed anti-automorphism $*$
  and idempotent $e$, we have that \[\mD(Ae)=\Hom_{A}(Ae,A) =eA \cong
  Ae^*.\]  In the last step, we have made $eA$ a left $A$-module
by the formula $a\cdot (eb)=eba^*$, and thus the
  automorphism $*$ provides an isomorphism to $Ae^*$ with the obvious
  left module structure.
 Of course
  $e_\gmu^*=e_\gmu$, so
 \[ \mD({A}^\bnu e_\gmu)\cong {A}^\bnu e_\gmu.\]}
 To prove (B1), it suffices to check that it is true for a
   bar-invariant spanning set of $U^-_e$.  Such a spanning set is
   given by the monomials $\bbf_{\bmuh}$.  Since
   $\bbf_{\bmuh}h_{\gla}=h_{ \bmuh\cup\gla}$, the result follows. For
   (B2), it is enough to note that $h_{\gmu}\otimes
   u_\emptyset=h_{\gmu'}$ where
   $\gmu'=(\bmuh(1),\dots,\bmuh(\ell),\emptyset)$. Since
  properties (B1) and (B2) determine the behavior on a spanning set, they uniquely
  characterize the map. 
  
  Finally, we prove (B3) by induction.
Since $h_{\gla_{\xi}}=\sum_{\op{type}(\sS)=\xi}q^{\op{Deg}(\sS)} u_{\op{sh}(\sS)}$
we have that $h_{\gla_{\xi}}=u_\xi+\sum_{\eta<\xi}
b_{\eta,\xi}(q)u_\eta$ for $b_{\eta,\xi}\in\mZ[q,q^{-1}]$.  Thus, for $\xi$ minimal, $\Psi({u_\xi})=
\Psi({h_{\gla_\xi}})={h_{\gla_\xi}}=u_\xi$.  Now, we assume the claim holds for $\eta<\xi$ and get
\begin{align*}
  \Psi({u_\xi})&=\Psi{h_{\gla_{\xi}}-\sum_{\eta<\xi}
    b_{\eta,\xi}(q)u_\eta}=h_{\gla_{\xi}}-\Psi{\sum_{\eta<\xi}
    b_{\eta,\xi}(q)u_\eta}=u_\xi+\sum_{\eta<\xi}
  a_{\eta,\xi}'(q)u_\eta
\end{align*}
where again $a_{\eta,\xi}\in\mZ[q,q^{-1}]$.
\end{proof}

Choosing integers $(\tilde{\charge}_1,\dots,\tilde{\charge}_\ell)$
such that $\tilde{\charge}_i\equiv \charge_i\pmod e$, one has a
natural vector space isomorphism $\beta$ taking standard vectors to standard vectors
between our Fock space $\Fock_\ell$ and Uglov's Fock space
$\tilde{\Fock}_\ell$ (in \cite{Uglov}, this is denoted
$\mathbf{F}_{q}[(\tilde{\charge}_1,\dots,\tilde{\charge}_\ell)]$. On the level of  $\ell$-multipartition this isomorphism is just given by reading the components in opposite order. %\todo{Doublecheck $q$ in Uglov}
Using the isomorphism $\beta$, one can also define a $q$-antilinear involution $\Psi'$ on $\Fock_\ell$ by pulling back Uglov's bar involution $\Psi_U$. We wish to compare $\Psi'$ with $\Psi$.
\begin{definition}
A multi-charge is
{\de $m$-dominant} if for each $i$, we have $\tilde{\charge}_{i+1}-\tilde{\charge}_{i}\geq m.$
\end{definition}

If our charge is $m$-dominant, then the map $\beta$ is an isomorphism of $U_e^-$-modules on the
weight spaces of height $\leq m$, \cite[Lemma 3.20]{BKgd}; we should
note that we have several differences of convention \cite{BKgd}, but in this case, they felicitously cancel.

\begin{theorem}
\label{twoFocks}
If the multicharge
  $(\tilde{\charge}_1,\dots,\tilde{\charge}_\ell)$ is $m$-dominant, then on
  the weight spaces of height $\leq m$, we have that $\displaystyle \Psi=\Psi'.$
\end{theorem}
\begin{proof}
 Since we have given uniquely characterizing properties (B1)-(B2) of $\Psi$, we
 need only to show that $\Psi'$ satisfies these.
  Since $\beta$ is an isomorphism of $U_e^-$-modules on the
weight spaces of height $\leq m$, we have that
\[\Psi'(g\cdot v)
=\beta^{-1}(\Psi_U(\beta(g.v)))
=\beta^{-1}(\bar{g}\cdot(\Psi_U(\beta(v)))
=\bar{g}\cdot\beta^{-1}(\Psi_U(\beta(v)))
=\bar g\cdot \Psi'(v)\]
for weight vectors $v$ of height $\leq m$ and $g\in U^-_e$ (using \cite[3.31]{Uglov}); hence (B1) holds.

To see (B2) we must consider the effect
of tensoring with $u_\emptyset$ in Uglov's language:  we add a new
variable at every $\ell$th place in the bosonic realization of the
Fock space, and include all of these variable up to the charge in our
wedge, which is larger than
any of the variables corresponding to any other partitions that appear in the
wedge.
%Let us abuse notation and let $u_\xi\otimes u_\emptyset$
%denote this basis vector in $\tilde{\Fock}_\ell$.
Applying Uglov's formula \cite[3.23]{Uglov} for the bar
involution in terms of wedges, we must reverse the order of all terms
in the wedge, and use the rules  \cite[3.16]{Uglov} to reorder
them.

First, we pull all of the new variables to the right end of the
wedge.  We claim that no non-zero correction terms appear when we do
this, so the expression remains a pure wedge.
Let $u_i$ be a new variable, and $u_j$ any other.  By
\cite[3.16]{Uglov}, $u_i\wedge u_j=\pm q^{?} u_j\wedge u_i$ modulo
wedges that replace this with $u_{i+h\ell}\wedge u_{j-h\ell}$ or
$u_{j-h\ell}\wedge u_{i+h\ell}$ for some $h$ such that $i+h\ell$ is
strictly
between $i$ and $j$.  Thus, all terms which arise have the form
\[\cdots \wedge u_{i+h\ell}\wedge \cdots \wedge u_{i+(h-1)\ell}\wedge
\cdots \wedge u_{i+\ell}\wedge \cdots \wedge u_{i+h\ell}\wedge \cdots\] if
$j>i$ or
\[\cdots \wedge u_{i+h\ell}\wedge \cdots \wedge u_{i-\ell}\wedge
\cdots \wedge u_{i+(h+1)\ell}\wedge \cdots \wedge u_{i+h\ell}\wedge \cdots\]
if $i<j$.  Both of these wedges are clearly 0 by the straightening rules
if $h=1$, and by induction, we can reduce to this case. That is,
\begin{equation}
  \label{wedge-commute}
  \cdots \wedge u_i\wedge u_j\wedge \cdots =\cdots \wedge\pm q^{?} u_j\wedge u_i\wedge \cdots
\end{equation}
as long as all new variables between $i$ and $j$ appear somewhere in
the wedge. Now, reorder the old variables;  again using \cite[3.16]{Uglov}, we
see that exactly the same correction terms appear here.  Now, applying
\eqref{wedge-commute}, we see that reordering the new variables and
shuffling them back into place introduces no correction terms.  Note
that the we need not worry about the sign or power of $q$ appearing in
\eqref{wedge-commute}; by  \cite[3.24]{Uglov}, they must be correct. This shows that \[\Psi'(v\otimes
u_\emptyset)=\Psi'(v)\otimes u_\emptyset\] and completes the proof.
\end{proof}

\subsection{Decomposition numbers and canonical bases}
\label{sec:decomp-numb-canon}

For the remainder of this article, we assume that
$\operatorname{char}(\K)=0$. In the setup of Theorem~\ref{cellular}, let  $P^\xi$ be the graded projective cover of $F^\xi$ (which is the same as the projective cover of $W^\xi$).  Obviously, since we have a
surjective map ${A}^\bnu e_{\gla_\xi}\to F^\xi$, the projective $P^\xi$
is a summand of this module with multiplicity $1$.  We let $p_\xi=[P^\xi]$, and $\phi_\xi=[F^\xi]$.
\begin{theorem}
\label{decompnumb}
 Assume that
$\operatorname{char}(\K)=0$. We have  
\begin{eqnarray}
\label{multipl}
p_\xi&=&u_\xi+\sum_{\eta< \xi} a_{\xi\eta}(q)u_\eta\quad\text{
   for polynomials }
 a_{\xi\eta}(q)\in q\mZ_{\geq 0}[q].
 \end{eqnarray}
 By BGG-reciprocity,
 we also have $u_\xi=\phi_\xi+\sum_{\xi>\eta}a_{\xi\eta}(q)\phi_\eta$,
 and so the coefficients $a_{\xi\eta}$ are the graded decomposition
 numbers of ${A}^\bnu$.
\end{theorem}

\begin{proof}
First note that BGG reciprocity applies in the nongraded version, since ${A}^\bnu$ is quasi-hereditary by \cite{DJM} and Theorem \ref{cellular}. The graded version follows then by general arguments as for instance in \cite[\S 8]{MSTransl}. Therefore it remains to show \eqref{multipl}. Since the complementary summand of $e_{\gla_\xi}{A}^\bnu$ to $P_\xi$ is filtered
  by Weyl modules $W^\eta$ for $\eta<\xi$, we must have an expression as above for some
  Laurent polynomials $b_{\xi\eta}(q)\in\mZ[q]$, see Proposition \ref{psi} (the coefficients of these
  polynomials are manifestly non-negative integral, since they are graded
  multiplicities of Weyl modules in the standard filtration of
  $P^\xi$).

The negativity of powers of $q$ is equivalent to showing that
$\End_{\tilde{A}^\bnu }(\bigoplus P^\xi)$ is a positively graded algebra. For this, it
suffices to show the same for $\End_{\tilde{A}^\bnu }(\bigoplus \tilde P^\xi)$, the
corresponding modules over $\tilde{A}^\bnu$, since this ring surjects onto
$\End_{{A}^\bnu }(\bigoplus P^\xi)=\End_{\tilde{A}^\bnu }(\bigoplus
P^\xi)$ by the universal property of projectives.
This is the ring of self-extensions of the sum of shifts of simple perverse
sheaves $\bigoplus_{\grave
\mu}p_*\K_{\fQ(\grave\mu)}$, so the indecomposable projectives are all
of the form $\Ext_{\GRep_\bd}\big(L_\xi,\bigoplus_{\grave
\mu}p_*\K_{\fQ(\grave\mu)}\big)$ where $L_\xi$ is a shift of a simple perverse
sheaf appearing in $\bigoplus_{\grave
\mu}p_*\K_{\fQ(\grave\mu)}$.
The projective $\tilde P^\xi$ appears  with multiplicity 1 in $A^\bnu e_\xi$,
so $L_\xi$ appears with multiplicity 1 in $p_*\K_{\fQ(\grave{\la}_\xi)}$.
Since $p_*\K_{\fQ(\grave{\la}_\xi)}$ is Verdier self-dual, so is $L_\xi$
and thus $L_\xi$ is perverse.
Thus, we have an isomorphism %\todo{Here I have a problem to understand the proof. Why is this a graded algebra homomorphism?}
\[\Ext^*_{\GRep_\bd}(L_{\xi},L_{\xi'})\cong \Hom_{\tilde{A}^\bnu}(\tilde
P^\xi,\tilde P^{\xi'})\]
and a surjective map \[\Ext^*_{\GRep_\bd}(L_{\xi},L_{\xi'})\to
\Hom_{A^\bnu}(P^\xi,P^{\xi'}).\]  Since perverse sheaves are the heart of a
$t$-structure, this $\Ext$ group is positively graded, so only
positive shifts of Weyl modules can appear in the standard filtration
of $P^\xi$. 
\end{proof}

This property of upper-triangularity with respect to a standard
basis is one of the hallmarks of a {\bf canonical basis} (in the sense of Lusztig), the other
being fixed under a bar-involution, such as the $\Psi$ defined
above; in fact, the basis $\{p_\xi\}$ does satisfy these properties and thus can be thought of a canonical basis, justifying our naming  (A general definition of canonical bases including this one as
a special case is given in \cite{WebCB}):

\begin{theorem}\label{canonical}
   Assume $\operatorname{char}(\K)=0$. The basis $p_\xi$ is the unique basis of $\Fock_\ell$ such that:
  \begin{itemize}
  \item $\Psi(p_\xi)=p_\xi$ and
  \item $p_\xi=u_\xi+\sum_{\eta< \xi} a_{\xi\eta}(q)u_\eta$ for
 $e_{\xi\eta}(q)\in q\mZ_{\geq 0}[q]$.
  \end{itemize}
That is, $p_\xi$ is the ``canonical basis'' of the involution $\Psi$
with ``dual canonical basis'' $\phi_\xi$.
\end{theorem}

\begin{remark}{\rm 
As with the bar involution
$\Psi$, we can view this basis as a ``limit'' of Uglov's canonical
basis $G_*^+$ defined in \cite[3.25]{Uglov}, in the sense that the transition matrix from
the standard basis to
Uglov's basis stabilizes to the transition matrix for our basis on the
 weight spaces of height $\leq m$ once the multicharge is $m$-dominant.}
\end{remark}

\begin{proof}[Proof of Theorem \ref{canonical}]
  We have already shown that the second property holds. For the first,
  we note that $\mD$ fixes the projective
  ${A}^\bnu_{\gla_\xi}$, and thus fixes each of its ``isotypic
  components'' (the maximal summands of $e_{\gla_\xi}{A}^\bnu$ which are
  direct sums of the same indecomposable projective).  Since $P_\xi$
  appears with multiplicity 1, it is also invariant under $\mD$, so $\Psi(p_\xi)=p_\xi$.

 Uniqueness holds by the usual standard
 arguments: if $\{p'_\xi\}$ is another such basis, then we have that
 \[p_\xi-p'_\xi=\sum (a_{\xi\eta}(q)-a'_{\xi\eta}(q))u_\eta\] is
 $\Psi$-invariant and has only coefficients in $q\mZ[q]$ with respect to the standard basis. This is impossible by Theorem
 \ref{psi-properties}(B3) (consider a maximal standard basis vector). 
 
 Finally Theorem \ref{fock-space} together with the definition \eqref{pairing} show that the classes of the simple
 objects are dual to the classes of indecomposable projectives.
\end{proof}

%\todo{The following is wrong, right?}
\begin{corollary}
Assume $\operatorname{char}(\K)=0$. The graded decomposition numbers
$a_{\xi\eta}$ are the entries of the matrix expressing the canonical
basis in terms of the standard basis, hence given by  Proposition \ref{psi}.
\end{corollary}

When $e=\infty$, a number of interpretations of this canonical basis
as classes of indecomposable projectives have already appeared in the literature, for instance:
\begin{itemize}
\item in  category $\cO$
  of type A,  e.g. \cite{Subasis}, \cite{BS3}, \cite{FKS}, \cite{MSslk}, \cite{SS}.% using \cite{FKK}.
\item for generalized Khovanov algebras, \cite{BS3},
\item with projective modules over the algebra $T^\bnu$ constructed
  earlier in \cite[6.8]{WebCB}, which we have already shown is Morita
  equivalent to $A^\bnu$, and
\item with projective modules over what Hu and Mathas call a ``quiver
  Schur algebra'' \cite[\S
7.4]{HMQ} (which is different from our definition).
\end{itemize}
Unsurprisingly, all of these categories are graded Morita equivalent;
in fact, a direct proof of the equivalence of each pair has appeared
for certain parabolic category $\cO$ and generalized Khovanov algebras
in \cite{BS3},  for $T^\bnu$ and Hu and Mathas's algebra in \cite[Th. 5.31]{Webmerged}, for $T^\bnu$ and category $\cO$ in \cite[\S 9.2]{Webmerged}, and for Hu and Mathas's algebra and category $\cO$ by \cite[Th. C]{HMQ}, and in abstract terms in \cite{BLW}.

In the special cases \cite[Th. 45, Prop.
76]{FSS} explicit formulas are available which play an important role
in the context of link homology theories developed therein. In general it is a nontrivial task to compute these bases explicitly. 

It is a general phenomenon that canonical bases appear naturally from
categorifications with positivity and integrality properties. For
example, in all simply-laced finite dimensional Lie algebras,
Lusztig's canonical bases for tensor products of irreducible
representations were shown by the second author to arise from the
projectives in the so-called tensor algebra. This and the general interplay between
canonical bases and higher representation theory is discussed in
greater detail in \cite{WebCB}. 

Also note that our result is another striking
example of how $\glehat$ behaves like a Kac-Moody algebra, even though
it is not covered by the theory of categorification for quantum groups from \cite{KLIII}, \cite{Rou2KM}. One theme in recent years has been a recognition that theorems like
Theorem \ref{fock-space}, which prove the existence of a {\it weak}
categorification, can be considerably strengthened to a {\it strong} categorification or strong action by considering $2$-morphisms, \cite{CR04}, \cite{Rou2KM}, \cite{KLIII}, \cite{CaLa}. At the moment, no consensus definition of such a strong action for $\glehat$
exists. Some examples and partial constructions have appeared in the literature: for example, some pieces occur in 
\cite[\S 9]{HYII}, in the context of polynomial functors and \cite{ShVa}, in the context of representations of rational Cherednik
algebras; but even in these situations, the notion of a
strong action remains unclear and a formal definition is so far not available. We expect that our construction should provide such a $\glehat$-action. 

\bibliographystyle{amsalpha}

\bibliography{./ref2}

\end{document}